\documentclass[12pt]{amsart}
\usepackage{amsmath,amssymb,amsthm,graphicx,mathrsfs,url}
\usepackage[usenames,dvipsnames]{color}
\usepackage[margin=3cm]{geometry}
\usepackage[shortlabels]{enumitem}
\usepackage[colorlinks=true,linkcolor=Violet,citecolor=PineGreen]{hyperref}
\usepackage[capitalize]{cleveref}

\usepackage{comment}

\usepackage{graphicx}
\usepackage{mathrsfs} 
\numberwithin{equation}{section}

\newtheorem{theorem}{Theorem}[section]
\newtheorem{corollary}[theorem]{Corollary}
\newtheorem{proposition}[theorem]{Proposition}
\newtheorem{lemma}[theorem]{Lemma}
\newtheorem{definition}[theorem]{Definition}
\theoremstyle{remark}

\newtheorem{remark}[theorem]{Remark}

\newcommand{\p}[1]{\left(#1 \right)}

\newcommand{\dotp}[1]{\langle #1 \rangle}

\newcommand{\one}{\mathbf{1}}
\newcommand{\dd}{\mathrm{d}}

\newcommand{\Var}{\mathrm{Var}}

\newcommand{\Lip}{\mathrm{Lip}}
\newcommand{\PI}{\mathrm{PI}}

\DeclareMathOperator*{\diam}{diam}

\newcommand{\esssup}{\mathrm{ess}\,\mathrm{sup}}

\newcommand{\B}{\mathrm{B}}

\newcommand{\Cf}{\mathrm{C}}
\newcommand{\D}{\mathrm{D}}
\newcommand{\E}{\mathbf{E}}

\renewcommand{\H}{\mathrm{H}}

\renewcommand{\L}{\mathrm{L}}

\newcommand{\N}{\mathbb{N}}

\renewcommand{\P}{\mathbf{P}}
\newcommand{\Q}{\mathbb{Q}}
\newcommand{\R}{\mathbb{R}}
\renewcommand{\S}{\mathbb{S}}

\newcommand{\X}{\mathscr{X}}
\newcommand{\Y}{\mathscr{Y}}

\newcommand{\cC}{\mathscr{C}}

\newcommand{\cF}{\mathcal{F}}

\newcommand{\cH}{\mathcal{H}}

\newcommand{\cL}{\mathcal{L}}
\newcommand{\cM}{\mathcal{M}}
\newcommand{\cN}{\mathcal{N}}

\newcommand{\cQ}{\mathcal{Q}}

\newcommand{\cT}{\mathcal{T}}

\newcommand{\eps}{\varepsilon}
\renewcommand{\phi}{\varphi}

\renewcommand{\leq}{\leqslant}
\renewcommand{\geq}{\geqslant}

\usepackage{etoolbox} % for \csnumexpr

\author{Vincent Divol }
\address{Center for research in economics and Statistics, UMR CNRS 9194, 5 Avenue Henry Le Chatelier, 91120 Palaiseau, France}
\email{vincent.divol@ensae.fr}
\thanks{
The author thanks Yann Chaubet for interesting discussions and helpful comments about this work.  This work was supported by the Labex Ecodec (reference project ANR-11-LABEX-0047).
}

\begin{document}

\title{Spectral stability of empirical metric-measure Laplacians}

\vspace{-0.7cm}
\begin{abstract}
    The variance of nonparametric estimators is typically insensitive to the regularity of the object being estimated. We establish such a property for the spectra of graph Laplacian matrices at a fixed bandwidth $h>0$. Specifically, given $n$ i.i.d.~samples from a probability measure $\mu$ on a Polish metric space, we compare the eigenvalues of the empirical weighted Laplacian operator $\Delta_{\mu_n}^h$ to those of the population counterpart $\Delta_\mu^h$ under a spectral gap condition, bounding the relative error by $1/\sqrt{nv_\mu(h)}$ for eigenvalues of order smaller than $h^{-2}$, where $v_\mu(h)$ is the smallest mass of a ball of radius $h$. This bound requires very weak regularity conditions on $\mu$: it is satisfied if $\mu$ belongs to the class of coarse PI measures that we introduce. This class contains measures on metric graphs, spaces with sufficiently regular boundaries, corners, or branch points, together with discretizations or thickenings of these at scale $O(h)$. Even for measures having  densities of regularity $s>2$ on  manifolds (the only known case so far), our bound improves on the state-of-the-art by shaving off logarithmic factors.
\end{abstract}
\maketitle

\vspace{-0.7cm}
\setlength{\footskip}{30pt}

\section{Introduction}

\subsection{Graph Laplacian operators}

While bias and variance jointly determine the risk of an estimator, they are governed by  distinct  mechanisms. The bias of an estimator measures the discrepancy between an approximating model and a ground truth. In nonparametric statistics, its magnitude  depends on the regularity of the latter, encoded by the membership of the functional of interest (a density, a regression function, etc.)  in a Banach space  (a H\"older space, a Besov space, etc.). The variance measures the sensitivity of the estimator with respect to perturbation of the data. As such, it is blind to the regularity of the ground truth; only minimal properties of the data-generating measure are needed to control it. This phenomenon appears in virtually all classical examples in nonparametric statistics. A tight control of the variance of a kernel density estimator only requires the underlying density to be bounded. Similarly, the variance of a projection estimator in regression is controlled by the variance of the underlying noise, and the risk fluctuations of an empirical risk minimizer are bounded under very mild conditions on the data-generating probability. The goal of this paper is to exhibit a similar phenomenon for graph Laplacian matrices.  

Graph Laplacian matrices (and particularly their spectra) appear in  different contexts in data science. Their most popular use is  through the so-called spectral clustering algorithm, which constitutes a common off-the-shelf clustering method for graphs \cite{von2007tutorial}. Nonlinear dimension reduction methods such as Laplacian Eigenmaps \cite{belkin2003laplacian} and Diffusion Maps \cite{coifman2006diffusion} also rely on graph Laplacians, which are  used as a preprocessing step  in the popular UMAP algorithm \cite{mcinnes2018umap}. Finally, eigenfunctions of graph Laplacians have been shown to be useful for regression tasks, e.g., Rosa and Rousseau \cite{rosa_rousseau} show that they can be used to construct minimax adaptive estimators in a regression context under manifold constraints, see also \cite{hacquard2022topologically}.

Graph Laplacian matrices are 
built from a set of $n$ i.i.d.~observations $X_1,\dots,X_n$ from some (Borel) probability measure $\mu$ on a complete separable space $(\X,\rho)$ (i.e., a Polish space). The most common example is when the metric space is $\R^D$ endowed with the Euclidean distance, although many other situations of interest have  appeared in the literature, including notably $(\X,\rho)$ being equal to the Wasserstein space, see \cite{carter2009fine, talmon2013empirical, mishne2016hierarchical, zelesko2020earthmover, kileel2021manifold, lu2023discovering, xu2025manifold}. A weighted graph is obtained by choosing weights of the form $K(X_i,X_j)$, where  $K:\X\times\X\to \R$ is a nonnegative symmetric bounded measurable function,   called a kernel. The graph Laplacian matrix $\cL$ then acts on a vector $u\in \R^n$ through
\[
(\cL u)_i = \frac 2n\sum_{j=1}^n K(X_i,X_j)(u_j-u_i)\quad  \text{ for }i\in [n].
\] 
The matrix $\cL$ can be seen as an empirical analogue of the operator $\Delta_{\mu,K}$  acting on $\L^2(\mu)$ through
\begin{equation}\label{eq:unnormalized}
	\Delta_{\mu,K}u(x) =2 \int K(x,y)  (u(y)-u(x))\dd \mu(y)\quad  \text{for $\mu$-almost all $x\in \X.$}
\end{equation}
We call such an operator the  \emph{$(\mu,K)$-Laplacian operator}. 
It is immediate  that the operator $-\Delta_{\mu,K}$ is a nonnegative bounded self-adjoint operator on $\L^2(\mu)$ and that it satisfies for all $u,v\in \L^2(\mu)$,
\begin{equation}\label{eq:dirichlet}
	\dotp{u,-\Delta_{\mu,K}v}_\mu =  \iint K(x,y)(u(x)-u(y))(v(x)-v(y)) \dd \mu(x)\dd\mu(y).
\end{equation}
As such, the spectrum of $-\Delta_{\mu,K}$ is included in $[0,+\infty)$. 
Let $\lambda_\infty=\lambda_{\infty}(\mu,K)$ be the infimum of the essential spectrum of $-\Delta_{\mu,K}$ (which we take to be  the operator norm $\|\Delta_{\mu,K}\|_{\L^2(\mu)}$ of $-\Delta_{\mu,K}$ if there is no essential spectrum). For every $\ell\in \N$, we define $\lambda_\ell=\lambda_\ell(\mu,K)$ as follows. Let $\lambda_0\leq  \lambda_1\leq \cdots$ be the  eigenvalues of $-\Delta_{\mu,K}$  strictly smaller than $\lambda_\infty$, counted with multiplicity (with $\lambda_0=0$ whenever $\lambda_\infty>0$). If there are only finitely many of such eigenvalues, we set $\lambda_\ell=\lambda_\infty$ for all larger values of $\ell$. Observe that the graph Laplacian matrix $\cL$ is nothing but the $(\mu_n,K)$-Laplacian operator with $\mu_n=\frac 1n \sum_{i=1}^n\delta_{X_i}$ the empirical measure associated with the sample.

A vast body of work is concerned with the proximity of the operators $\Delta_{\mu_n,K}$ and $\Delta_{\mu,K}$ under various assumptions on the probability measure $\mu$, the most common of which is the \emph{manifold hypothesis}, which posits that $\mu$ is nearly uniform on a $d$-dimensional compact submanifold $M$ of $\X=\R^D$.  In that case, the kernel $K(x,y)$ typically takes the form 
\begin{equation}\label{eq:indicator_kernel}
	K^h(x,y) = h^{-d-2}\eta\p{\frac{|x-y|}h},
\end{equation}
for some scale parameter $h>0$ and some bounded measurable function $\eta: [0,+\infty) \to [0,+\infty)$. 
As $h\to 0$, the operator $\Delta_{\mu,K^h}$ converges towards a weighted Laplacian operator  $\Delta_\mu$ on $M$, whose drift depends on the density of $\mu$. Results in this area typically assert that if the scale $h$ is chosen as a certain function of $n$, the operator  $\Delta_{\mu_n,K^h}$ is close to $\Delta_\mu$ in some sense. Proximity can either be measured in a pointwise sense, that is by comparing $\Delta_{\mu}u$ and $\Delta_{\mu_n,K^h}u$ for some test function $u$, or in a spectral sense, by comparing the eigenfunctions or the eigenvalues of the two operators. The second question is more delicate. State-of-the-art results by Garc\'ia Trillos, Li and Venkatraman \cite{trillos2025minimax} show that if $\mu$ has a density of regularity $\Cf^{2+\eps}$ for some $\eps>0$, then, for $h$ of order $n^{-1/(d+4)}$, the proximity between the $\ell$th eigenvalues of the two operators is of order $n^{-2/(d+4)}$ (up to logarithmic factors). The reader familiar with nonparametric statistics will have recognized the usual rate for estimating a density (or a regression function) in a $\Cf^2$-model \cite{tsybakov2008nonparametric}. In such a scenario, the risk is decomposed into a squared bias of order $h^4$ and a variance of order $1/(nh^d)$. Balancing the two terms indicates that the correct value of $h$ is indeed of order $n^{-1/(d+4)}$. 

However, the bias-variance decomposition of the risk  in \cite{trillos2025minimax} is not explicit. In particular, it is not clear at all whether the variance can be controlled without a regular density, or even without a manifold structure. To be explicit, in this setting, and denoting by $\lambda_{\ell,\mu}$ the $\ell$th eigenvalue of the limit differential operator, we refer to the bias as the deterministic error $|\lambda_{\ell,\mu}-\lambda_\ell(\mu,K^h)|$ and to  the quantity $\E[(\lambda_\ell(\mu,K^h)-\lambda_\ell(\mu_n,K^h))^2]$ as the variance. It will be shown in this work  that very mild conditions on $\mu$ are required to obtain the state-of-the-art bound of order $1/(nh^d)$ for the variance.  
Such a scientific inquiry  always runs the risk of falling into the realm of vacuous generalization, namely the pursuit of abstraction for its own sake. The author would like to  believe that he is not guilty of such a crime. Indeed, in this paper's context, generalization serves two distinct purposes. 

The first one is theoretical in nature. Isolating the minimal structures required to control the variance yields a deeper understanding of the mechanisms at play. Not being able to use powerful yet technical tools such as Riemannian geometry paradoxically simplifies the proofs. Moreover, the class of spaces that naturally emerges from this analysis is far from being exotic. It includes  the class of PI spaces, which lies at the core of harmonic analysis on metric spaces as developed by Heinonen and Koskela \cite{heinonen1998quasiconformal, heinonen2015sobolev}. It is the author's hope that this paper will convince theoretical statisticians that tools from analysis on metric spaces are highly valuable and remain  underutilized in the field.
 
Second, from a practical standpoint, the manifold hypothesis \cite{cayton2005algorithms} is frequently invoked in statistical learning theory to explain the empirical success of various algorithms, including those based on graph Laplacian matrices. However, this hypothesis turns out to be too stringent for many real-world datasets that contain regions of high curvatures, singularities, or branching behaviors. More fundamentally, the manifold hypothesis is not easily operationalizable, in the sense that it does not come with an operation allowing for testing its validity \cite{bridgman1927logic}. Indeed, verifying whether a sample in a high-dimensional space can be interpolated by a low-dimensional manifold is a notoriously difficult task, for which no computationally implementable statistical test currently exists \cite{fefferman2016testing}. In contrast, we demonstrate that if the data-generating measure $\mu$ belongs to the class of coarse PI measures introduced in this paper, then the empirical measure $\mu_n$ will also belong to this class with high probability (\Cref{cor:empirical_PI}). This stability property is a first step towards tractable, implementable testing procedures for this geometric assumption, which constitutes a weak manifold hypothesis.

\subsection{Coarse PI measures}

The class of coarse PI measures that we define in this work rests on two distinct conditions. Let $\mu$ be  a nonzero locally finite measure on $(\X,\rho)$ (that is, $\mu$ is finite on balls). For technical reasons, we will also always assume that the support of $\mu$ is not reduced to a single point. We let $B(x,t)$ be the open ball of radius $t$ centered at $x\in\X$ and $\bar B(x,t)$ be the corresponding closed ball. 
First, $\mu$ is a doubling  measure  at scales larger than $r$ with doubling constant $C_\D\geq 1$ if 
\begin{equation}\label{eq:mu_condition}
\forall x\in \X,\ \forall t\geq r,\ \mu(B(x,t))\leq C_\D\mu(B(x,t/2)).
\end{equation}
We then say that $\mu\in \mathrm{DM}_r(C_\D)$.
We say the measure $\mu$ satisfies a coarse Poincaré inequality if the following holds. There exist $C_{\PI}, r>0$ and $\kappa\geq 1$ such that for all $x\in \X$, for all $t\geq r$, and every locally bounded measurable function $u$, if $B=B(x,t)$ and $\kappa B=B(x,\kappa t)$, then
\begin{equation}\label{eq:PI_h}
    \int_B |u-u_B|^2\dd \mu \leq C_{\PI} t^2 \int_{\kappa B}\Lip_{\mu,r}[u]^2 \dd \mu.
\end{equation}
In the expression above,  $u_B$ is the $\mu$-average of $u$ on $B$, $\kappa B$ is the ball with same center as $B$, of radius $\kappa t$, and 
\begin{equation}\label{eq:poincare_def}
   \forall x\in\X,\ \Lip_{\mu,r}[u](x) = \esssup_\mu \left\{ \frac{|u(x)-u(y)|}{r}: y\in B(x,r) \right\}.\footnote{\text{A reader worried about measurability  might easily convince themselves that $\Lip_{\mu,r}[u]$ is measurable.}}
\end{equation}
We denote by $\mathrm{PI}_r(C_\D,C_{\PI},\kappa)$ 
the class of \emph{coarse  PI   measures} at scale $r$ with constants $C_\D,C_{\PI},\kappa$, i.e., the set of  measures in  $\mathrm{DM}_r(C_\D)$ satisfying  \eqref{eq:PI_h}. 

The class of   coarse  PI  measures is arguably very large. First, it contains all   PI spaces, which correspond to the previous definition with $r=0$ (and $ \Lip_{\mu,r}[u]$ replaced by the local Lipschitz constant of $u$); a proof of this fact is given in \Cref{sec:PI_0}. In particular, measures having densities bounded below and above on compact $d$-dimensional manifolds  are coarse PI, together with metric graphs, or spaces containing boundaries, corners, etc.

More importantly to us, this class is closed under arbitrary modification of the
geometry below the scale $r$. The relevant notion of perturbation is the
\emph{relative Prokhorov closeness} introduced by Burago, Ivanov and
Kurylev \cite{burago2019spectral}: two measures $\mu$ and $\nu$ are \emph{$(\varepsilon,\tau)$-close}
if one may first discard from each of them at most a fraction $1-e^{-\tau}$ of its mass
at every point, in such a way that what remains can then be transported one onto the
other without moving any unit of mass by more than $\varepsilon$
(\Cref{def:prokhorov}). 

\begin{proposition}[informal version of \Cref{lem:stab_condition}]
\label{prop:stability-informal}
Let $\mu\in\mathrm{PI}_r(C_D,C_{\mathrm{PI}},\kappa)$ and let $\nu$ be $(r/2,\tau)$-close
to $\mu$. Then $\nu\in\mathrm{PI}_{3r}(C_D',C_{\mathrm{PI}}',2\kappa)$, where $C_D'$ and
$C_{\mathrm{PI}}'$ are explicit and depend only on $C_D$, $C_{\mathrm{PI}}$ and $\tau$.
\end{proposition}

 Let us give some applications of the previous proposition. A measure on a manifold
may be replaced by its discretization on a maximal $r$-separated subset of $\mathcal X$,
and both the coarse doubling property and the coarse Poincar\'e inequality still hold at
scales of order $r$ (\Cref{prop:graph_poincare}); the class is also stable under the
addition of tubular noise of magnitude of order $r$; and, most importantly, the empirical
measure $\mu_n$ of a coarse PI measure is itself a coarse PI measure at scale of order
$r$ with high probability, as soon as $n\mu(B(x,r/8))\gtrsim\log(n)$ for all
$x\in\mathcal X$ (\Cref{cor:empirical_PI}). It is then possible to verify numerically whether $\mu_n$ is a coarse PI measure by verifying whether the conditions \eqref{eq:mu_condition} and \eqref{eq:PI_h} are satisfied: doing so would give a statistical test for the hypothesis $\mu\in \PI_r$. We do not develop this here,  leaving this problem (and in particular the development of practical algorithms) to further inquiry.

\subsection{Main theorem}
For $x\in \X$, define $\eta_{\mu}^h(x) = \int \eta\p{\frac{\rho(x,y)}h}\dd \mu(y)$.
We will focus on kernels of the form 
\begin{equation}\label{eq:varying_kernel}
	K_{\mu}^h(x,y) = \frac{h^{-2}}{\sqrt{\eta_{\mu}^h(x) \eta_\mu^h(y)}}\ \eta\p{\frac{\rho(x,y)}h},\qquad x,y\in \X,\qquad h>0.
\end{equation}

\begin{remark}
   Our analysis could be carried out with kernels of the form
    \eqref{eq:indicator_kernel} instead of \eqref{eq:varying_kernel}, with fewer technical difficulties. We nonetheless favor
    \eqref{eq:varying_kernel}   since 
such kernels have the advantage of being agnostic to the dimension $d$ of the underlying space, the quantity $h^d$ being replaced by the quantity $\eta_\mu^h(x)$. In particular, these kernels can be used without having to estimate in an a priori step the dimension of $\X$, or even when the dimension of $\X$ is ill-defined. 
\end{remark}

We write $\Delta_{\mu}^h = \Delta_{\mu,K_{\mu}^h}$. 
We will assume that the function $\eta$ satisfies the following condition:\footnote{The constants $1$ and $1/2$ play no special role  in  condition \eqref{eq:eta}. All that matters is that $\eta$ is a continuous bounded function with compact support, and that it is large enough on a neighborhood of $0$.}
\begin{equation}\label{eq:eta}\tag{N}
	\text{The function $\eta$ is continuous, supported in $[0,1]$, with $\|\eta\|_\infty\leq 1$, $\eta_{|[0,1/2]}\geq 1/2$. }
\end{equation}
We are now ready to state our main theorem. We  write $v_\mu(h) =\inf_{x\in \X}\mu(B(x,h))$ for $h>0$. Also, we write $\lambda_{\ell,\mu}^h = \lambda_\ell(\mu,K_\mu^h)$.

\begin{theorem}\label{thm}
	Let $C_\D,\kappa\geq 1$, $ C_{\PI}>0$, $\gamma\in (0,1]$.  There exist  constants $\beta,C_0,C_1>0$ depending on $C_\D$, $\kappa$,  $\gamma$  and the modulus of continuity of $\eta$  such that the following holds. Let  $h>0$  and let $\mu$ be a probability measure in $\mathrm{PI}_{\beta h}(C_\D,C_{\PI}, \kappa)$. Let  $n\geq 1$ be an integer such that $nv_\mu(h)\geq C_1\log (n(1+C_{\PI}))$. 
	If $\ell\geq 1$ is such that  $\min_{k\neq \ell}|\lambda_{k,\mu}^h-\lambda_{\ell,\mu}^h|\geq \gamma \lambda_{\ell,\mu}^h$, and $C_0 \lambda_{\ell,\mu}^h\leq  h^{-2}$, then
	\begin{equation}
		\E[(\lambda_{\ell,\mu_n}^h-\lambda_{\ell,\mu}^h)^2]\leq C_0(\lambda_{\ell,\mu}^h)^2 \frac{1+C_{\PI}}{nv_\mu(h)}.
	\end{equation}
\end{theorem}
In words, simple eigenvalues of the empirical Laplacian $-\Delta_{\mu_n}^h$ are a good relative approximation of simple eigenvalues of the Laplacian $-\Delta_\mu^h$ for eigenvalues up to  order $h^{-2}$, with relative error of order $1/\sqrt{nv_\mu(h)}$.
When $v_\mu(h)$ is of order $h^d$ for some dimension parameter $d$, the previous bound is of order $1/\sqrt{nh^d}$. If, in addition, we assume that $\mu$ has a $\Cf^{2+\eps}$ positive density on a $d$-dimensional manifold $\X$, then $\lambda_{\ell,\mu}^h$ is at distance $O(h^2)$ from the $\ell$th eigenvalue $\lambda_{\ell,\mu}$ of the limit differential operator $-\Delta_\mu$, as proved in \cite{trillos2025minimax} for a slightly different operator. 
The bias-variance trade-off then reads $h^2+1/\sqrt{nh^d}$. Optimizing this trade-off yields a bound of order $n^{-2/(d+4)}$, which is  exactly the rate of convergence that was obtained in \cite{trillos2025minimax}. Actually, we even improve their rates by shaving off logarithmic factors. Importantly, the variance bound requires neither smoothness nor manifold structure. Even within the manifold setting, this is a strict improvement: for a density on $M$ merely bounded away from zero and infinity (with no Hölder regularity whatsoever) the variance term $1/(nh^d)$ persists, while only the bias degrades with the (ir)regularity of the density. Beyond manifolds, the same bound holds for any coarse PI measure. 

The condition $nv_\mu(h)\gtrsim \log n$ ensures roughly speaking that every ball of radius $h$ contains at least one observation point. For values of $h$ under this threshold (say if $nv_\mu(h)$ is of order $1$),   a nontrivial proportion of observation points is $h$-isolated. This implies that the multiplicity of zero as an eigenvalue of $-\Delta_{\mu_n}^h$ is of order $n$. In particular, the stability of the spectrum breaks down.

 The condition $\lambda_{\ell,\mu}^h\lesssim h^{-2}$ is natural. Indeed, we will show that the essential spectrum of $-\Delta_\mu^h$ lies in $[ch^{-2},+\infty)$ for some constant $c$ (\Cref{lem:ess-spectrum}). Thus, our theorem  covers eigenvalues up to a constant fraction of the bottom of the essential spectrum. Instead, \cite{trillos2025minimax} requires the condition $h^2\lambda_{\ell,\mu} \lesssim \min(1, \lambda_{\ell,\mu}^{-(d-1)/2})$, which is strictly stronger whenever $d>1$.

This manuscript (and its title!) is close in spirit to  a  paper by Burago, Ivanov, and Kurylev, who tackle the same question on general metric measure spaces, for \emph{any} perturbation $\nu$ of a measure $\mu$ \cite{burago2019spectral}. Shifting the focus to random perturbations $\mu_n$ allows the use of concentration inequalities, which yield sharper bounds. The notion of \emph{relative Prokhorov perturbation} (\Cref{def:prokhorov}) introduced by the authors is also key to this work, and allows us to derive deterministic stability results in \Cref{sec:burago} and \Cref{sec:proof}.

Let us now outline the strategy of proof. The formula \eqref{eq:dirichlet} for $\dotp{u,-\Delta_\mu^h u}_\mu$ expresses this quantity as a smoothed version of the energy of the gradient of $u$; we accordingly denote it by $\|u\|_{\dot\H^1_h(\mu)}^2$, and, for $s\geq 0$, let $\|\cdot\|_{\H^{-1}_{h,s}(\mu)}$ be the associated Bessel dual seminorm,
\begin{equation}\label{eq:hm1}
    \|v\|_{\H^{-1}_{h,s}(\mu)} = \sup\Big\{\dotp{v,u}_\mu:\ \int u\dd \mu=0,\ \|u\|_{\dot\H^1_h(\mu)}^2+s\|u\|^2_{\L^2(\mu)}\leq 1\Big\}.
\end{equation}
These play the role of the negative Sobolev norms attached to the measure $\mu$ at scale $h$. The norms  $ \|\cdot\|_{\H^{-1}_{h,s}(\mu)} $ are all equivalent for different values of $s$ (see \Cref{sec:poincare}), so one should mostly think about the case $s=1$ for now (we will elaborate later on why other values of $s$ may be relevant). %Considering other values of $s$ (namely $s=\lambda_{\ell,\mu}^h$) will be key to obtain the right dependency with respect to $\lambda_{\ell,\mu}^h$ when it becomes large (say, it is almost of order $h^{-2}$).  

Our analysis rests on a single principle: given two measures $\mu$ and $\nu$, the eigenvalues of $-\Delta_\mu^h$ and $-\Delta_\nu^h$ are close as soon as the difference $(\Delta_\mu^h-\Delta_\nu^h)\phi$ is small in $\H^{-1}_{h,s}$ for the relevant eigenfunctions $\phi$. 
Considering the $\H^{-1}_{h,s}$-norm of the difference instead of its $\L^2$-norm is crucial. Indeed, the difference $(\Delta_\mu^h-\Delta_\nu^h)\phi$ may have a large $\L^2$-norm, being dominated by high frequencies, while remaining small in $\H^{-1}_{h,s}$, which precisely damps them. Granting such a control, the passage to eigenvalue stability follows from  relative perturbation estimates that are standard in finite dimension \cite{barlow1990computing, ipsen1998relative}. The entire difficulty is therefore concentrated in bounding $\|(\Delta_\mu^h-\Delta_\nu^h)\phi\|_{\H^{-1}_{h,s}(\nu)}$.

The idea of controlling spectral stability through the $\H^{-1}_{h,s}$ seminorm, together with a multiscale bound on it, was introduced by García Trillos, Li and Venkatraman \cite{trillos2025minimax} in the manifold setting; they refer to such a bound as a \emph{multiscale Poincaré inequality}. Their argument rests on a dyadic cube decomposition of the space, whose discontinuous building blocks create technical difficulties that can only be handled by relying on the underlying Riemannian geometry. Our contribution is twofold. First, we show that this strategy extends to the far larger class of coarse PI spaces, with no manifold and  no differentiable structure available. 
Second, we observe that the almost-Lipschitz  spline systems of Auscher, Hytönen and Tapiola \cite{auscher2013orthonormal, hytonen2014almost} can be substituted for dyadic cubes; their continuity removes the difficulties mentioned above and considerably simplifies the argument. Unfortunately, their construction requires the space to be doubling at all small scales, not a coarse one. Still, we show in \Cref{sec:spline} that, on a coarse doubling space at scale $r$, their construction can be adapted to produce a spline system that has a modulus of continuity at scale $t$ of order $t^\vartheta$ for $t\gtrsim r$ and some H\"older-like parameter $\vartheta$ close to $1$, which is enough for our purpose.
Establishing the multiscale bound in the coarse PI setting is the technical core of the paper, and is the subject of \Cref{sec:poincare}. The sharp variance rate then follows by controlling $\|(\Delta_\mu^h-\Delta_{\mu_n}^h)\phi\|_{\H^{-1}_{h,s}(\mu_n)}$ through the concentration of the resulting empirical $V$-statistics, carried out in \Cref{sec:concentration}.

As in \cite{trillos2025minimax}, a byproduct of our analysis is the closeness of eigenfunctions with respect to the empirical seminorm $\|\cdot\|_{\dot\H^1_h(\mu_n)}$ and to the empirical $\L^2(\mu_n)$-norm, with the same rate of convergence of order $1/\sqrt{nv_\mu(h)}$ (\Cref{thm:eigenfunction}). We  expect, however, that such a control is suboptimal with respect to the $\L^2(\mu_n)$-norm and that the correct rate  is of order $h/ \sqrt{nv_\mu(h)}$. Obtaining such bounds likely requires  additional structure on the metric-measure space (e.g., $\X$ is a $d$-dimensional manifold). We leave this endeavor to future work.

\begin{remark}
    Operators of the form $\Delta_{\mu,K}$ are usually called
    \emph{unnormalized} Laplacians \cite{von2007tutorial}. Two popular
    variants are the \emph{random walk} operator
    $\cL^{\mathrm{rw}}u = h^{-2}(u - T_Ku/m_\mu^K)$, where
    $m_\mu^K(x) = \int K(x,y)\dd\mu(y)$ and $T_K$ is the integral operator
    with kernel $K$, and its \emph{symmetric} version
    $(m_\mu^K)^{1/2}\cL^{\mathrm{rw}}(m_\mu^K)^{-1/2}$; being conjugate, the
    two share the same spectrum. The kernel $K^h_\mu$ is the $\alpha = 1/2$ normalization of Coifman and
    Lafon \cite{coifman2006diffusion}. It is the choice for which the
    limiting operator is the generator of a Langevin diffusion on a manifold. Let $m_\mu^h = m_\mu^{K}$ for $K=K_\mu^h$. 
    For this choice of kernel, the three operators are spectrally equivalent:
    they share the same Dirichlet form, so that the min--max principle gives
    $\lambda_\ell^{\mathrm{rw}}
    = \lambda^h_{\ell,\mu}/(2m)$ for some $m\in [\inf h^2 m_\mu^h,\sup h^2 m_\mu^h]$. In the manifold case, 
    $h^2m^h_\mu = 1 + O(h^2)$, so the two spectra become equivalent in the limit $h\to 0$.
\end{remark}

\begin{remark}
    The minimax rate for estimating both eigenvalues and eigenfunctions of the limit operator $\Delta_\mu$ was recently obtained  by Chaubet and the author \cite{chaubet2025minimax}. The rates derived in this paper suggest that the rate $1/\sqrt{nv_\mu(h)}$ is not minimax for eigenvalue estimation, and that  graph Laplacian-based estimators could be debiased. We conjecture that such a debiasing method requires some smoothness, and is therefore impossible for a general coarse PI measure, suggesting that the rate $1/\sqrt{nv_\mu(h)}$ becomes minimax on this larger class. This would constitute another advantage of considering the class of coarse PI measures, which would be in this sense the ``right'' class on which to study the estimator $\lambda_{\mu_n}^h$. We leave this interesting question to future work.
\end{remark}

\subsection{Related work}
The spectral convergence of graph Laplacians has been studied in two rather different regimes.  In the fixed-bandwidth regime, which is the one adopted here,  the population object is the  operator $\Delta^h_\mu$ at the same bandwidth $h$. This regime goes back to von Luxburg, Belkin and Bousquet \cite{von2008consistency}, whose analysis of spectral clustering already isolates the condition that the eigenvalue under consideration lie below the essential spectrum (a condition of which our assumption $h^2\lambda^h_{\ell,\mu}\lesssim 1$ is the quantitative counterpart). A related line of work studies  the proximity  between spectra of integral operators and their empirical counterpart, see, e.g., \cite{koltchinskii2000, rosasco2010learning}. However, since $\Delta_\mu^h$ is not  an integral operator, this literature cannot be directly applied. 

In the manifold regime, one lets the bandwidth $h$ go to zero jointly with $n\to\infty$, and compares the spectrum of $\Delta^h_{\mu_n}$ with that of a limiting differential operator on a manifold. This line of work originates with Belkin and Niyogi \cite{belkin2006convergence} and with Coifman and Lafon \cite{coifman2006diffusion}, see also e.g., \cite{hein2005graphs, singer2006graph, gine2006empirical}. Quantitative rates were obtained by Garc\'ia Trillos, Gerlach, Hein and Slep\v{c}ev \cite{garcia_trillos_error_2020} through a variational, optimal-transport based analysis, by Calder and Garc\'ia Trillos \cite{calder2022improved} for $\eps$-graphs and $k$-NN graphs, and then by Cheng and Wu \cite{cheng2022convergence} and by Wahl \cite{wahl2024kernel} (among others). Precise recent results include the state-of-the-art bound \cite{trillos2025minimax} already mentioned and a central limit theorem   obtained in \cite{li2025central}. Mordant and Munk explore in  \cite{mordant2023statistical} the possible uses of metric-measure Laplacian operators for the statistical analysis of complex objects. At last, let us cite the recent work by Bungert and Slep\v cev \cite{bungert2025convergence} which establishes $\Gamma$-convergence and spectral convergence of graph Laplacian matrices on union of manifolds of different dimensions.

The estimates of \Cref{sec:proof} belong to the theory of relative spectral perturbation. On the deterministic side this theory originates in numerical linear algebra \cite{barlow1990computing,ipsen1998relative}, as recalled in \Cref{sec:proof}. On the statistical side, relative bounds have proved decisive for the spectral analysis of empirical covariance operators, where the eigenvalues accumulate at zero and absolute perturbation bounds of Davis-Kahan type \cite{davis1970rotation} become vacuous: see Koltchinskii and Lounici \cite{koltchinskii2017new} and, in particular, Jirak and Wahl \cite{jirak2020perturbation,jirak2023relative}, who replace the absolute spectral gap by a relative rank and thus obtain sharp expansions for empirical eigenvalues and spectral projectors. 

\subsection{Organization of the article}
Properties of coarse PI spaces are introduced in \Cref{sec:PI}. We provide in \Cref{sec:burago} deterministic stability bounds on the spectrum of weighted Laplace operators in the spirit of \cite{burago2019spectral}. The multiscale Poincaré inequality, at the core of our proof technique, is derived in \Cref{sec:poincare}. \Cref{sec:concentration} is dedicated to concentration inequalities for empirical Laplacian operators. A proof of \Cref{thm} is given in \Cref{sec:proof} together with a corollary on the proximity of eigenfunctions with respect to the $\dot\H^1_h$-norm.

\section{Coarse  PI  measures}\label{sec:PI}

We derive in this section elementary properties of coarse  PI measures. That is, we show that coarse  PI measures satisfy many ``coarse versions'' of properties satisfied by  PI measures \cite{heinonen2015sobolev}. 

Before starting, let us remark that every ball of radius $r/2$ centered at a point has positive measure for any measure $\mu\in \mathrm{DM}_r(C_\D)$. Otherwise, the doubling property could be iterated to obtain that $\mu=0$. We also state an elementary fact, which follows from the inequality $\Lip_{\mu,r}[u]\leq \tau \Lip_{\mu,\tau r}[u]$ for $\tau\geq 1$.

\begin{lemma}[Coarse PI measures are also coarse at larger scales]\label{lem:larger}
	 Let $C_\D,\kappa,\tau\geq 1$ and $C_{\PI}, r>0$. Let $\mu\in \PI_r(C_\D,C_{\PI},\kappa)$. Then, $\mu \in \PI_{\tau r}(C_\D,\tau^2 C_{\PI},\kappa)$. 
\end{lemma}

We say that a subset $\Y\subseteq \X$ of points is $t$-separated if for all distinct points $y,y'\in \Y$, $\rho(y,y')\geq t$. 

\begin{lemma}\label{lem:covering_ball}
    Let $C_\D\geq 1$ and $r>0$. 
 Let $\mu$ be a measure in $\mathrm{DM}_r(C_\D)$. Let $x_0\in \X$. Then, for $t'\geq t\geq r/2$,  the  cardinality of any $(2t)$-separated set  in $ B(x_0,t')$ is at most  $C_{\mathrm D}^{\lceil \log_2(1+2t'/t)\rceil}$.
\end{lemma}
\begin{proof}
Let $\{x_1,\dots,x_N\}$ be a $(2t)$-separated set in $B(x_0,t')$. Then, the balls $B(x_i,t)$ are pairwise disjoint and included in $B(x_0, t'+t)$. We repeatedly apply \eqref{eq:mu_condition} to find that for any integer $J\geq 1$,
    \[
\mu(B(x_0, t'+t)) \geq \sum_{i=1}^N \mu(B(x_i,t))\geq  C_\D^{-J} \sum_{i=1}^N \mu(B(x_i,2^J t )).
    \]
    If $J\geq \log_2(1+2t'/t)$, then $B(x_0, t'+t)\subseteq B(x_i ,2^J t)$, so that $\mu(B(x_i,2^J t ))\geq \mu(B(x_0, t'+t)) $, which implies that $N\leq C_\D^J$.
\end{proof}

It is a well-known fact that PI spaces are connected, see, e.g., \cite{heinonen2015sobolev}. The next lemma is a coarse version of this result.

\begin{lemma}[Coarse  PI spaces are coarsely connected]\label{lem:quasiconvex}
      Let $C_\D,\kappa\geq 1$ and $C_{\PI}, r>0$. Let $\mu$ be a measure in $\mathrm{PI}_r(C_\D,C_{\PI}, \kappa)$. Let $t\geq r$ and let $x,y\in\X$ with $\rho(x,y)<t$. 
   Then, 
   there exists a path $x=x_1,x_2,\dots,x_N=y$ in $\X$ with $N \leq 2C_{\mathrm D}^{\lceil \log_2(5+8\kappa t/r)\rceil}$ such that for all $i\in [N-1]$, $\rho(x_i,x_{i+1})< r$ and for all $i\in [N]$, $x_i\in B(x, 2\kappa t+r)$.
\end{lemma}
\begin{proof}
    Let $x\in \X$ and let $U_x$ be the (open) set of points $y\in  B(x,2\kappa t+r)$ such that there exists a path (of arbitrary length for now) as in the statement of the lemma.  If $y\in U_x$ and $y'\not \in U_x$ are both in $B(x,2\kappa t+r)$, then $\rho(y,y')\geq r$, for otherwise we could append $y'$ to a path joining $x$ to $y$. So $\Lip_{\mu,r}[\one_{U_x}](y)=0$ for $y \in B(x,2\kappa t)$. By \eqref{eq:PI_h}, the function $\one_{U_x}$ is therefore $\mu$-a.e. constant on $B(x,2t)$. This constant cannot be zero, as $\one_{U_x}$ is  equal to one on $B(x,r)$ (with $\mu(B(x,r))>0$). Thus, this constant is equal to one, meaning that $\mu$-almost all points $y\in B(x,2t)$ are in $U_x$. We have shown that $\mu(B(x,2t)\backslash U_x)=0$. Assume that some point $y\in B(x,t)$ is not in $U_x$. Then, no points $y'$ in $B(y,r)$ can be in $U_x$ (for otherwise we could append $y$ to a path joining $x$ to $y'$). But $B(y,r)\subseteq B(x,t+r)\subseteq B(x,2t)$ and  there holds $\mu(B(y,r))>0$. As $B(y,r)\subseteq B(x,2t)\backslash U_x$, this is a contradiction with $\mu(B(x,2t)\backslash U_x)=0$. In conclusion, there holds $B(x,t)\subseteq U_x$, proving the first statement.

    Let us now bound the length $N$ of a minimal path in $B(x,2\kappa t+r)$ between $x$ and $y$. Observe that   two nonconsecutive points along the path  are  at distance at least $r$ from each other, for otherwise  the path could be shortened. Thus, the points with odd indices are at distance at least $r$ from each other. There are at least $N/2$ such points, and the conclusion then follows from \Cref{lem:covering_ball}.
\end{proof}

We now show that coarse  PI spaces also satisfy a reverse doubling inequality. 
\begin{lemma}\label{lem:reverse}
    Let $C_\D,\kappa\geq 1$ and $C_{\PI},r>0$. Let $\mu \in \mathrm{PI}_r(C_\D,C_{\PI},\kappa)$. Then, for all  $x\in\X$ and $4r\leq t< \diam(\X)/2$,
    \[
    \mu(B(x,t/2)) \leq c_\D\mu(B(x,t)),
    \]
    where $c_\D = 1-C_\D^{-4}$.
\end{lemma}
\begin{proof}
    Let $B=B(x,t)$ be a ball of radius $t\geq 4r$ with $t< \diam(\X)/2$, centered at $x\in \X$. By definition of the diameter, there exists  $y\in \X$ that is not in $B$. We consider a path from $x$ to $y$ as in \Cref{lem:quasiconvex} (applied with the parameter $\rho(x,y)+r$ instead of $t$). We can find one point $x_i\in \X$ along this path with $\rho(x,x_i)\in [3t/4-r/2, 3t/4+r/2]$. But then, the ball centered at $x_i$ of radius $t/8$ is included in $B\backslash B(x,t/2)$. Thus,
\begin{align*}
    1-\frac{\mu(B(x,t/2))}{\mu(B(x,t))} = \frac{\mu(B(x,t)\backslash B(x,t/2))}{\mu(B(x,t))} \geq \frac{\mu(B(x_i,t/8))}{\mu(B(x,t))}\geq \frac{C_\D^{-4} \mu(B(x_i,2t))}{\mu(B(x,t))}\geq C_\D^{-4}.
\end{align*}
This implies that $\mu(B(x,t/2))\leq c_\D\mu(B(x,t))$ for $c_\D=1-C_\D^{-4}$.
\end{proof} 

\begin{remark}\label{rem:bounded}
    The previous lemma implies the following dichotomy: either $\mu$ has infinite mass, and $\X$ is  of infinite diameter (since we assume that $\mu$ is locally finite), or $\mu$ has finite mass, and $\X$ is bounded. Indeed, if $\diam(\X)=+\infty$, we can iterate the inequality of the lemma and obtain balls of arbitrarily large mass.
\end{remark}

It is a standard fact that PI spaces are of Hausdorff dimension at least $1$ (since they are quasiconvex, see \cite[Section 8.3]{heinonen2015sobolev}). The next lemma captures this property in a coarse way (the case $v_\mu(t)\asymp t^d$ would break down if $d<1$).

\begin{lemma}\label{lem:superlinear}
      Let $C_\D,\kappa\geq 1$ and $C_{\PI},r>0$. Let $\mu \in \mathrm{PI}_r(C_\D,C_{\PI},\kappa)$. Let $4r\leq t<\diam(\X)/2$. Then, $v_\mu(t)/t\geq c v_\mu(r)/r$, where $c=3/(16C_\D)$.
\end{lemma}
\begin{proof}
    Consider the same construction as in the previous proof, with a path as in \Cref{lem:quasiconvex} from the center $x$ of $B=B(x,t)$ to a point $y$ outside the ball. Define for $k\geq 1$ the  shell $S_k = B(x,kr)\backslash B(x,(k-1)r)$.  As long as $k\leq t/r$, $S_k\subseteq B$. Moreover, since the path  makes steps of size smaller than $r$ and finishes outside $B$, it must intersect all the shells $S_k$ for $1\leq k\leq t/r$. In particular, these shells are not empty.  Let $x_k\in S_k$ for $1\leq k\leq t/r$. Observe that if $|k-k'|>1$, then $\rho(x_k,x_{k'})\geq r$. Thus, the balls $B(x_{2\ell},r/2)$ for $\ell$ such that $ 2\ell+1\leq t/r$ are pairwise disjoint and included in $B$. There are $K=\lfloor \frac{t-r}{2r}\rfloor$ such balls and they are each of mass at least $C_\D^{-1}v_\mu(r)$ (by doubling). Thus, there holds $\mu(B)\geq KC_\D^{-1}v_\mu(r)$. Taking the infimum over all balls, we find that $v_\mu(t)\geq  K C_\D^{-1}v_\mu(r)$.
    Since $t\geq 4r$, there holds $K\geq \lfloor \frac{3t}{8r}\rfloor\geq \frac{3t}{16r}$. We therefore obtain the conclusion.
\end{proof}

 One of the convenient properties of the class of coarse  PI measures is that it is stable under small perturbations of the underlying measure. This is in stark contrast with the class of PI measures, which is not stable under ``transport-like'' perturbations (e.g., connectedness might be lost in the process). In particular, if $\mu$ is a coarse doubling PI measure,  so will the empirical measure $\mu_n$, as long as $n$ is large enough. To make this precise, we need to define what kind of perturbations are allowed.  Let $\mu$, $\nu$ be two measures on $\X$. We say that a  measure $\pi$ on $\X^2$ (with respect to the Borel  $\sigma$-algebra for the product metric) is a transport plan between $\mu$ and $\nu$ if the first marginal of $\pi$ is $\mu$, and the second one is $\nu$. We write $\cT(\mu,\nu)$ for the set of transport plans between $\mu$ and $\nu$. We say that $\mu$ and $\nu$ are $\eps$-close if there exists $\pi\in \cT(\mu,\nu)$ with cost
 \[ \cC(\pi)= \esssup_\pi\{ \rho(x,y):\ (x,y)\in \X^2\} \leq  \eps.\]
 The infimum of the $\eps$ such that $\mu$ and $\nu$ are $\eps$-close is called the  $\infty$-Wasserstein distance between $\mu$ and $\nu$. 
To illuminate the definition of the $\infty$-Wasserstein distance, let us think about the particular case where $\nu$ is obtained from $\mu$ through the pushforward by a measurable map $T$ that is close to the identity, in the sense that $\sup_{x\in \X}\rho(T(x),x)\leq \eps$. Then, $W_\infty(\mu,\nu)\leq \eps$, with a transport plan between $\mu$ and $\nu$ obtained by moving each unit of mass located at $x$ to $T(x)$. 
 
 \begin{definition}[Relative Prokhorov perturbation]\label{def:prokhorov}
    Let $\mu$ and $\nu$ be two measures on $(\X,\rho)$ and let $\eps,\tau\geq 0$. We say that $\mu$ and $\nu$ are $(\eps,\tau)$-close if there exist measures $\tilde \mu$ and $\tilde \nu$ on $(\X,\rho)$ that are $\eps$-close, such that 
    \[
    e^{-\tau}\mu\leq \tilde \mu \leq \mu,\quad   e^{-\tau}\nu\leq \tilde \nu \leq \nu.
    \]
 \end{definition}

The notion of $(\eps,\tau)$-closeness, which was introduced by Burago, Ivanov and Kurylev in \cite{burago2019spectral}, is much more permissive than the notion of closeness with respect to the $\infty$-Wasserstein distance. For instance, it is wrong in general that $W_\infty(\mu_n,\mu)$ goes to zero when $\mu$ is a probability measure (e.g., if the underlying space is disconnected), whereas we will show that, under very mild conditions, $\mu$ and $\mu_n$ are $(\eps,\tau)$-close to one another for small values of $\eps$ and $\tau$.

For a set $A\subseteq \X$ and $\eps>0$, we let 
\[
A^\eps = \{x\in \X:\ \exists y\in A,\ \rho(x,y)\leq \eps\}
\]
be the $\eps$-offset of $A$.

\begin{lemma}\label{lem:closeness_implication}
	  Let $\mu$ and $\nu$ be two $\sigma$-finite measures on $(\X,\rho)$ and let $\eps,\tau\geq 0$. Assume that $\mu$ and $\nu$ are $(\eps,\tau)$-close.  Then, for all Borel sets $A$ and $B$ with $A^\eps \subseteq B$,
	      \begin{equation}\label{eq:carac_relative}
	  	\nu(A)\leq e^\tau \mu(B) \quad \text{ and } \quad \mu(A)\leq e^\tau \nu(B).
	  \end{equation}
\end{lemma}
\begin{proof}
Let $\pi$ be a transport plan between $\tilde \mu$ and $\tilde \nu$ with $\cC(\pi)\leq \eps$. Let $A\subseteq \X$ be a Borel set. Since $\pi$ is concentrated on the set $\{(x,y):\ \rho(x,y)\leq \eps\}$, there holds for $\pi$-almost all pairs $(x,y)$ that $x\in A$ implies $y\in A^\eps\subseteq B$. Therefore,
\begin{align*}
	\mu(A) &\leq e^\tau \tilde\mu(A)=e^{\tau} \iint \one_{x\in A}\dd \pi(x,y) \leq e^{\tau}\iint \one_{y\in B} \dd \pi(x,y)= e^\tau \tilde\nu(B)\leq e^\tau \nu(B).
\end{align*}
Likewise, $\nu(A)\leq e^\tau \mu(B)$.  
\end{proof}

The $(\eps,\tau)$-closeness between $\mu$ and $\nu$ can be used to construct maps between $\L^2(\mu)$ and $\L^2(\nu)$. Assume for the sake of simplicity that $\mu=\tilde\mu$ and that $\nu=\tilde\nu$. Let $\pi$ be a transport plan between $\mu$ and $\nu$ with $\cC(\pi)\leq \eps$.  Since $\X$ is Polish, there exists a family $(\pi_x)_{x\in \X}$ of probability measures such that for all Borel sets $B\subseteq \X^2$, 
\[
\pi(B) = \int \pi_x(B_x) \dd \mu(x),
\]
where $B_x= \{ y:\ (x,y)\in B\}$, see \cite[Corollary 10.4.15]{bogachev2007conditional}. In particular, letting $B=\{\rho>\eps\}$ (with $\pi(B)=0$), we find that for $\mu$-almost all $x$, the measure $\pi_x$ is concentrated on $\bar B(x,\eps)$. We define the $\pi$-interpolation map $R_\pi:\L^2(\nu)\to\L^2(\mu)$  for $\mu$-almost all $x$ and some map $u\in \L^2(\nu)$ by
\begin{equation}
    R_\pi u(x)= \int u(y) \dd \pi_x(y).  
\end{equation}
Alternatively,  let $\iota_1:\L^2(\mu)\to \L^2(\pi)$ which maps $u$ to $(x,y)\mapsto u(x)$ and let $\cH_1$ be the image of $\iota_1$, with $P_1:\L^2(\pi)\to \cH_1$  the orthogonal projection onto $\cH_1$.  Define $\iota_2:\L^2(\nu)\to\L^2(\pi)$, $\cH_2$ and $P_2$ in a similar way. Then,  $R_{\pi}$ is equal to $P_1\circ \iota_2$ (where we identify $\cH_1$ with $\L^2(\mu)$ through the isometry $\iota_1$), see \cite[Proposition 10.4.18 and Theorem 10.1.5]{bogachev2007conditional}. Both points of view on $R_\pi$ (conditional expectation or orthogonal projection) will be useful.

\begin{lemma}[Properties of $R_\pi$]\label{lem:properties_R}
    Let $\pi$ be a locally finite measure on $\X^2$ with marginals $\mu$ and $\nu$. Let $f:\X\to [0,+\infty)$ be  a nonnegative bounded function. Assume that $\cC(\pi)\leq \eps$. Then, for all locally bounded measurable functions $u$ and for $\pi$-almost all $(x,y)\in\X^2$,
    \begin{enumerate}
        \item  $|R_\pi u(x)-u(y)|\leq 3\eps \Lip_{\nu,3\eps}[u](y)$;
        \item  $\Lip_{\mu,r}[R_\pi u](x)\leq 2(1+3\eps/r)\Lip_{\nu,r+3\eps}[u](y)$;
        \item if $u\in \L^2(\nu)$, then $\int f(x) R_\pi u(x)^2 \dd \mu(x) \leq \iint f(x) u(y)^2 \dd\pi(x,y)$.
    \end{enumerate}
\end{lemma}
    
\begin{proof}
    For $B$ an open ball, we let $\overline u_B = \mathrm{esssup}_B u$ and $\underline u_B=\mathrm{essinf}_B u$. We define $A_{B} = \{y'\in B:\ u>\overline u_B \text{ or } u<\underline u_B\}$. By definition, $\nu(A_{B})=0$. We let $\{y_i\}_{i\in \N}$ be a dense sequence in $\X$ and consider the set $A = \bigcup_{i\in \N}\bigcup_{r\in \Q, r>0} A_{B(y_i,r)}$, with $\nu(A)=0$. Thus, for $\mu$-almost $x\in \X$, there holds $\pi_x(A)=0$. We now assume that $(x,y)$ is such $\rho(x,y)\leq \eps$, $\pi_x(A \cup \{y':\ \rho(x,y')>\eps\})=0$ and $y\not\in A$ (which is true $\pi$-almost surely).
    \begin{enumerate}
        \item  There holds $R_\pi u(x)-u(y)= \int (u(y')-u(y))\dd \pi_x(y')$. Observe that for $\pi_x$-almost all $y'$, $\rho(y,y')\leq 2\eps$. Let $y_i$ be at distance less than $\delta$ from $y$ and let $t> 2\eps+\delta$ be a rational number. Write $B=B(y_i,t)$.  Since $y'\in B$ and $B\subseteq B(y,t+\delta)$, there holds $u(y')-u(y) \leq \overline u_B -u(y)\leq \esssup_{\nu}\{ u(z)-u(y):\ z\in B(y,t+\delta)\}$. Thus, $R_\pi u(x)-u(y)\leq \esssup_{\nu}\{ |u(z)-u(y)|:\ z\in B(y,t+\delta)\}$.
         Likewise, we show that $u(y)-R_\pi u(x) \leq  \esssup_{\nu}\{ |u(z)-u(y)|:\ z\in B(y,t+\delta)\}$. We have shown that 
        \[
|R_\pi u(x)-u(y)|\leq  \esssup_{\nu}\{ |u(z)-u(y)|:\ z\in B(y,t+\delta)\}
        \]
        where $\delta>0$ is a small rational number and $t>2\eps + \delta$ is rational. By taking $\delta$ and $t$ small enough, we obtain $|R_\pi u(x)-u(y)|\leq 3\eps \Lip_{\nu,3\eps}[u](y)$. 
        \item The proof is  similar to the first point;  we skip details. 
        \item By definition of the orthogonal projection, since $f\cdot R_\pi u\in \L^2(\mu)$,
    \begin{align*}
      I=  \int f(x) (R_\pi u(x))^2 \dd \mu(x) &= \dotp{\iota_1 (f R_\pi u), P_1\iota_2 u}_\pi = \dotp{\iota_1(f R_\pi u), \iota_2u}_\pi \\
        &= \iint f(x) R_\pi u(x) u(y) \dd \pi(x,y).
    \end{align*}
    Thus, by Cauchy-Schwarz inequality, 
    \[I \leq \sqrt
    {\iint f(x) u(y)^2 \dd \pi(x,y) \iint f(x) (R_\pi u(x))^2 \dd \pi(x,y)},\]
    which is equal to $\sqrt{\iint f(x) u(y)^2 \dd \pi(x,y) I}$. We conclude by dividing by $\sqrt{I}$. \qedhere
    \end{enumerate}
\end{proof}

\Cref{lem:stab_condition} is the central result of this section. It states that the coarse PI class is closed under relative Prokhorov perturbations: modifying a measure arbitrarily below the scale $r$ (by discretizing it, by thickening it, or by replacing it with a sample from it) leaves it in the class, at the cost of explicit constants. This is in stark constrast with standard regularity classes: neither manifolds, nor PI spaces, nor Ahlfors regular sets contain the empirical measures of their members, since $\mu_n$ is atomic and disconnected.
Corollaries~\ref{cor:empirical_PI} and~\ref{prop:graph_poincare} are the two instances we shall use.

\begin{proposition}[Coarse  PI measures are stable under relative Prokhorov perturbations]\label{lem:stab_condition}
    Let $C_\D,\kappa\geq 1$ and $C_{\PI}, r>0$ and $\tau\geq 0$. Let $\mu$ be a measure in $\PI_r(C_\D,C_{\PI}, \kappa)$. Let  $\nu$ be $(r/2,\tau)$-close to $\mu$. Then $\nu$ is in $\PI_{3r}(e^{2\tau}C_\D^3, C_{\PI}', 2\kappa)$, with $C'_{\PI}=64 e^{2\tau}  C_{\PI} + e^\tau$.
\end{proposition}

\begin{proof}   
       Let $x\in \X$ and $t\geq 3r$. 
       By \Cref{lem:closeness_implication}, and since $t-r\geq r$,
    \begin{align*}
        \nu(B(x,t/2))\geq e^{-\tau} \mu(B(x,(t-r)/2))\geq C_\D^{-3}e^{-\tau} \mu(B(x,4(t-r))).
    \end{align*}
    But $\mu(B(x,4(t-r)))\geq  e^{-\tau} \nu(B(x,4(t-r)-r/2))$ and $4(t-r)-r/2\geq t$, so  the measure $\nu$ is in $\mathrm{DM}_{3r}(e^{2\tau}C_\D^3)$. 
    
    Let us now prove that $\nu$ satisfies a coarse Poincaré inequality.  Let $\tilde \mu$ and $\tilde \nu$ be such that $e^{-\tau}\mu\leq \tilde\mu\leq \mu$ and $e^{-\tau}\nu\leq \tilde\nu\leq \nu$, and let $\pi$ be a transport plan between $\tilde\mu$ and $\tilde \nu$ with $\cC(\pi)\leq \eps$. Let $u$ be a locally bounded measurable function. 
    Let $B=B(x_0,t)$ be a ball of radius $t\geq 3r$,  let $B'=B(x_0,t+r/2)$ and let  $v=R_{\pi}u$.  We find that for all $c\in \R$,
    \begin{align*}
      \int_B|u-u_B|^2\dd  \nu &\leq  e^\tau\int_B |u-c|^2 \dd\tilde\nu =  e^\tau\iint \one_{y\in B} |u(y)-c|^2\dd \pi(x,y) \\ 
      &\leq e^{\tau} \iint \one_{x\in B'}  |u(y)-c|^2\dd \pi(x,y)
    \end{align*}
    where we use that since $\pi_x$ is supported on $\bar B(x,r/2)$, there holds $\one_{y\in B}\leq \one_{x\in B'}$ for $\pi$-almost all $(x,y)$. 
    We then write
    \begin{align*}
            \iint \one_{x\in B'} & |u(y)-c|^2\dd \pi(x,y) = \int_{B'} \int |u(y)-c|^2 \dd \pi_x(y) \dd \tilde \mu(x) \\
            &=  \int_{B'} \p{|v(x)-c|^2 +\int |u(y)-v(x)|^2 \dd \pi_x(y)} \dd \tilde \mu(x)
    \end{align*}
    According to \Cref{lem:properties_R} and since $s\mapsto s\Lip_{\nu,s}[u]$ is nondecreasing, we find that $|u(y)-v(x)|\leq 3r\Lip_{\nu,3r}[u](y)$ for $\pi$-almost all $(x,y)$. Thus, $\iint \one_{x\in B'}|u(y)-v(x)|^2 \dd \pi(x,y)\leq \int_{B(x_0,t+r)}  (3r\Lip_{\nu,3r}[u])^2 \dd \nu$. 
We now choose $c$ as the $\mu$-expectation of $v$ on $B'$, so that  we may  apply the Poincaré inequality for $ \mu\geq \tilde \mu$ to obtain that 
    \begin{align*}
  \int_{B'} |v(x)-c|^2 \dd\tilde \mu(x)
       &\leq C_{\PI} (t+r/2)^2 \int_{ \kappa B'}  \Lip_{\mu,r}[v]^2(x) \dd  \mu(x)\\
       &\leq C_{\PI}e^\tau (t+r/2)^2 \int_{ \kappa B'}  \Lip_{\mu,r}[v]^2(x) \dd  \pi(x,y).
    \end{align*}
    According to \Cref{lem:properties_R}, for $\pi$-almost all $(x,y)$,  
    \[\Lip_{\mu,r}[v](x)\leq 5 \Lip_{\nu,5r/2}[u](y) \leq 6\Lip_{\nu,3r}[u](y).\]
    In total,  we find that
    \[
    \int_B|u-u_B|^2\dd  \nu\leq  \p{36 C_{\PI}e^{2\tau} (t+r)^2+ e^\tau (3r)^2 } \int_{(\kappa B')^{r/2}} \Lip_{\nu,3r}[u]^2 \dd \nu.
    \]
    To conclude, observe that since $t\geq 3r$, there holds  $\kappa(t+r/2)+r/2\leq 2\kappa t$ so $(\kappa B')^{r/2}\subseteq 2\kappa B$, and
    \[
    36 C_{\PI}e^{2\tau} (t+r)^2+ e^\tau (3r)^2 \leq 64 t^2 e^{2\tau}C_{\PI}+ t^2 e^{\tau}. \qedhere
    \]
\end{proof}

We claimed in the abstract that the class of coarse PI measures is stable under ``thickening''. Indeed, if $\mu\in \PI_r(C_\D,C_{\PI},\kappa)$ and $\pi$ is any transport plan with first marginal $\mu$ and $\cC(\pi)\leq r/2$, then the previous proposition states that the second marginal $\nu$ of $\pi$ is also a coarse PI measure. This second marginal represents an ``$(r/2)$-thickening'' of $\mu$. Think for instance of the case where $\mu$ is a probability measure and $(X,Y)\sim \pi$ is obtained by first sampling $X\sim \mu$ and then $Y$ on $B(X,r/2)$ (say uniformly if $\X=\R^D$). Then, the support of $\nu$ is exactly the $(r/2)$-offset of the support of $\mu$. To put it another way, the coarse PI class at scale $r$ is stable under additive noise of scale $r/2$.

A very convenient application of the previous proposition is that it can be applied to empirical measures $\mu_n$. Indeed,  a probability measure  $\mu$ is $(\eps,\tau)$-close to the associated empirical measure $\mu_n$ for values of $\eps$ and $\tau$ large enough. 

\begin{proposition}\label{prop:empirical_closeness}
      Let $\eps>0$ and $\tau \in [0,1]$. Let $\mu$ be a probability measure on $(\X,\rho)$ and let $n\geq 1$.  Then,  with probability at least $1-2v_\mu(\eps)^{-1}e^{-nv_\mu(\eps)\tau^2/10}$, the empirical measure $\mu_n$ is $(4\eps,\tau)$-close to $\mu$.
\end{proposition}

\begin{proof}
    Consider a maximal $(2\eps)$-separated set $\X_\eps$ of $\X$. The open balls of radius $\eps$, centered at points of $\X_\eps$, are pairwise disjoint, and each of mass at least $v_\mu(\eps)$. Thus, the cardinality of $\X_\eps$ is at most $v_\mu(\eps)^{-1}$. Moreover, the maximality of the set implies that 
    every point of $\X$ is at distance at most $2\eps$ from a point of $\X_\eps$. This set can be used to create a partition $\cQ$ of $\X$ by sets of diameter at most $4\eps$, where each set contains a ball centered at a point in $\X_\eps$, of radius at least $\eps$. Moreover, the sets $Q\in \cQ$ can be chosen to be measurable (by breaking ties using an arbitrary order on the  finite set $\X_\eps$). Assume that 
      \[
    \sup_{Q\in \cQ} \frac{|\mu_n(Q)-\mu(Q)|}{\mu(Q)} \leq \frac{\tau}2.
    \]
    Elementary computations show that this implies that  for all $Q\in \cQ$,
    \begin{equation}\label{eq:control_ratio}
    e^{-\tau}\leq \frac{\mu_n(Q)}{\mu(Q)} \leq e^\tau.
    \end{equation}
    We create a partial matching between $\mu$ and $\mu_n$ by moving for each $Q\in \cQ$ a mass $m_Q=\min(\mu(Q),\mu_n(Q))$ between $\mu_{|Q}$ and $(\mu_n)_{|Q}$. This defines a matching between 
    \[ \tilde \mu=\sum_{Q\in \cQ} m_Q \frac{\mu_{|Q}}{\mu(Q)}\ \text{ and }\ \tilde \mu_n = \sum_{Q\in \cQ} m_Q \frac{(\mu_n)_{|Q}}{\mu_n(Q)},\]
    that has cost smaller than $4\eps$. Furthermore,  \eqref{eq:control_ratio} implies that  $e^{-\tau}\mu\leq \tilde \mu\leq \mu$ and $e^{-\tau}\mu_n\leq \tilde \mu_n\leq \mu_n$. That is, $\mu$ and $\mu_n$ are $(4\eps,\tau)$-close. Furthermore, for each $Q\in \cQ$, Bernstein's inequality implies that with probability at least $1-2e^{-nv_\mu(\eps)\tau^2/10}$, $|\mu_n(Q)-\mu(Q)|\leq \tau \mu(Q)/2$. We conclude with a union bound.
\end{proof}

\begin{corollary}\label{cor:empirical_PI}
    Let $C_\D,\kappa\geq 1$ and $C_{\PI}, r>0$. Let $\mu$ be a probability measure in $\PI_r(C_\D,C_{\PI}, \kappa)$. Let $n\geq 1$ be an integer. Then, with probability at least $1-2v_\mu(r/8)^{-1}e^{-nv_\mu(r/8)/10}$, the measure $\mu_n$ is in $\PI_{3r}(e^{2}C_\D^3, C_{\PI}', 2\kappa)$, with $C'_{\PI} = 64 e^2 C_{\PI}+ e$. 
\end{corollary}
\begin{proof}
	Let $\tau=1$ and $\eps=r/8$, so that $4\eps=r/2$. 
	According to \Cref{prop:empirical_closeness}, with probability $1-2v_\mu(\eps)^{-1}e^{-nv_\mu(\eps)/10}$, $\mu_n$ is $(4\eps,1)$-close to $\mu$. Thus, by \Cref{lem:stab_condition}, $\mu_n$ is in $\PI_{3r}(e^2 C_\D^3, C'_{\PI},2\kappa)$ with $C'_{\PI} = 64 e^2 C_{\PI}+ e$. 
\end{proof}

\Cref{cor:empirical_PI} is the stability-under-sampling property announced in the introduction. 
The next result states that a discrete graph built from a coarse  PI measure  satisfies a discrete Poincaré inequality. A version of this result for PI spaces was proven in  \cite{gill2015discrete}.

\begin{corollary}\label{prop:graph_poincare}
    Let $C_\D,\kappa, M\geq 1$, $C_{\PI},r>0$. Let $\mu$ be a measure in $\mathrm{PI}_r(C_\D,C_{\PI}, \kappa)$ and let $\eps\geq r/2$. Let $(B(x_i,\eps))_{i\in I}$ be a covering of $\X$, such  that each point of $\X$ belongs to at most $M$ balls $B(x_i,\eps)$. Let $\mu_\eps = \sum_{i\in I} \mu(B(x_i,\eps))\delta_{x_i}$. Then, $\mu$ and $\mu_\eps$ are $(\eps,\log(M))$-close to one another. 
    
    In particular, $\mu_\eps$ is in $\PI_{6\eps}(C_\D^3M^2, C'_{\PI},2\kappa)$, with $C'_{\PI}=256 M^2C_{\PI}\eps^2/r^2+ M$. 
\end{corollary}
\begin{proof}
	Define the measure $\tilde \mu = M^{-1}\sum_{i\in I} \mu_{|B(x_i,\eps)}$. By assumption, $\tilde \mu\leq \mu$ and $\tilde \mu\geq M^{-1}\mu$. Define $\tilde \mu_\eps = \mu_\eps/M$. A transport plan between $\tilde \mu$ and $\tilde \mu_\eps$ is obtained by transporting each $\mu_{|B(x_i,\eps)}$ to $\delta_{x_i}$. Since the cost of this transport plan is smaller than $\eps$, the measures $\mu$ and $\mu_\eps$ are $(\eps,\log(M))$-close to one another.
	
	 Since $\mu\in \mathrm{PI}_r(C_\D,C_{\PI}, \kappa)$, it is also in $\mathrm{PI}_{2 \eps}(C_\D,C_{\PI}', \kappa)$ for $C_{\PI}'=4 C_{\PI}\eps^2/r^2$ (see \Cref{lem:larger}). We may apply \Cref{lem:stab_condition} to obtain the conclusion.
\end{proof}

\section{Crude spectral stability under relative Prokhorov perturbations}\label{sec:burago}

In this section, we show preliminary stability bounds for eigenvalues. These stability bounds have the same flavor as the ones obtained by Burago, Ivanov, and Kurylev in \cite{burago2019spectral}. They are based on optimal transport theory, and will turn out to be too crude when applied to random measures. Still, they have the advantage of being applicable to any pair of measures that are $(\eps,\tau)$-close. Moreover, these crude stability bounds will constitute the first building block to obtain more refined bounds in  later sections. Our assumptions differ from \cite{burago2019spectral} in two ways: their spherical layer condition and their ball intersection volume condition are replaced by the coarse Poincaré inequality, which is stable under sampling (\Cref{cor:empirical_PI}), and the indicator kernel $\eta=\one_{[0,1]}$ is replaced by continuous kernels, whose modulus of continuity quantifies the boundary effect that the layer condition was controlling. 

\emph{In the remainder of the text, we assume that $\X$ is bounded. This is harmless, since we will eventually consider that there is a coarse PI probability measure on $\X$, which implies its boundedness, see \Cref{rem:bounded}.} 

Let us describe more precisely the spectrum of $\Delta_\mu^h$. 
 Recall that $-\Delta^h_\mu = 2(M - T)$, where $T$ is the integral operator with kernel $K^h_\mu$ and $M$ is the multiplication operator by $m^h_\mu : x\mapsto\int K^h_\mu(x,y)\dd\mu(y)$.

\begin{lemma}\label{lem:operator_norm}
    Let $C_\D,\kappa\geq 1$, $C_{\PI},h>0$ and $0<r\leq h/16$. Let $\mu\in\PI_{r}(C_\D,C_{\PI},\kappa)$. Then, $\|\Delta_\mu^h\|_{\L^2(\mu)}\geq C_\D^{-8}h^{-2}$.
\end{lemma}
\begin{proof}
    Assume first that $\diam(\X)\leq h/2$. Then, $\mu$ has finite mass (\Cref{rem:bounded}), denoted by $p$. Observe that $\eta(\rho(x,y)/h)\geq 1/2$ for all $x,y\in \X$. Thus,
    \[
    K_\mu^h(x,y) \geq \frac{h^{-2}}{2 \sqrt{\eta_\mu^h(x)\eta_\mu^h(y)}}\geq \frac{h^{-2}}{2p}.
    \]
    Since by assumption $\mu$ is not reduced to a Dirac mass, there exists a non zero function $u\in \L^2(\mu)$ with zero average. We have $\dotp{-\Delta_\mu^h u,u}_\mu \geq \frac{h^{-2}}{2p}\int (u(x)-u(y))^2 \dd\mu(x)\dd\mu(y) = h^{-2} \|u\|^2_{\L^2(\mu)}$. Thus, $\|\Delta_\mu^h\|_{\L^2(\mu)}\geq h^{-2}$.

    Assume now that $\diam(\X)>h/2$. Let $B=B(x_0,h/8)$ be a ball and let $u=\one_B$.  Define $S = B(x_0,h/4)\backslash B$. The reverse doubling property (\Cref{lem:reverse}) yields that $\mu(S) = \mu(B(x_0,h/4))-\mu(B) \geq (1-c_\D)\mu(B) = C_\D^{-4}\mu(B)$. If $(u(x)-u(y))^2>0$, then either $x\in B$ and $y\not\in B$, or $x\not\in B$ and $y\in B$. 
     For $x\in B$ and $y\in S$, there holds $\rho(x,y)\leq h/2$ and $B(x,h),B(y,h)\subset B(x_0,2h)$. Therefore,
    \[
        K_\mu^h(x,y) \geq \frac{h^{-2}}{2 \sqrt{\eta_\mu^h(x)\eta_\mu^h(y)}}\geq \frac{h^{-2}}{2\mu(B(x_0,2h))}\geq \frac{h^{-2}}{2C_\D^4 \mu(B)}.
    \]
    Thus, $\dotp{-\Delta_\mu^h u,u}_\mu \geq \frac{h^{-2}} {C_\D^4 \mu(B)}\mu(B)\mu(S)\geq h^{-2}C_\D^{-4}\mu(S)= h^{-2}C_\D^{-8}\mu(B)$. Since $\mu(B)=\|u\|^2_{\L^2(\mu)}$, we obtain the conclusion.
\end{proof}

\begin{lemma}\label{lem:ess-spectrum}
    Let $C_\D\geq1$ and $\mu\in\mathrm{DM}_{h/2}(C_\D)$. Then $T$ is Hilbert-Schmidt, $m^h_\mu$ is bounded, and
    \begin{equation}\label{eq:ess-spectrum}
        \sigma_{\mathrm{ess}}(-\Delta^h_\mu) \subseteq [C_\D^{-3/2}h^{-2},+\infty).
    \end{equation}
\end{lemma}

\begin{proof}
    By \Cref{lem:bound_kernel}, $K^h_\mu(x,y)\leq 2h^{-2}C_\D^{3/2}/\mu(B(x,h))$. Thus, $m^h_\mu\leq 2h^{-2}C_\D^{3/2}$, so that $M$ is bounded. Since $\X$ is bounded, \Cref{lem:covering_ball} shows that a maximal $(h/2)$-separated subset $\{x_i\}_{i\in I}$ of $\X$ is finite, and the balls $B(x_i,h/2)$ cover $\X$; as $B(x,h)\supseteq B(x_i,h/2)$ for $x\in B(x_i,h/2)$, we get $\int\mu(B(x,h))^{-1}\dd\mu(x)\leq|I|<\infty$ and therefore $\iint(K^h_\mu)^2\dd\mu\dd\mu\leq4h^{-4}C_\D^{3}|I|<\infty$. The operator $T$ is Hilbert-Schmidt, and therefore compact.

    Consequently $-\Delta^h_\mu$ is a compact perturbation of $2M$ and, by  Weyl's theorem, there holds $\sigma_{\mathrm{ess}}(-\Delta^h_\mu) = \sigma_{\mathrm{ess}}(2M)\subseteq\sigma(2M)$, the latter being the essential range of $2m^h_\mu$. It remains to bound $m^h_\mu$ from below. Since $\eta\leq1$ is supported on $[0,1]$, we have $\eta^h_\mu(x)\leq\mu(B(x,h))$ and, for $\rho(x,y)\leq h$, $\eta^h_\mu(y)\leq\mu(B(y,h))\leq\mu(B(x,2h))\leq C_\D\mu(B(x,h))$. As $\eta\geq1/2$ on $[0,1/2]$,
    \begin{equation}\label{eq:lower_bound_m}
        m^h_\mu(x)\geq \frac{h^{-2}}{2} \int_{B(x,h/2)}\frac{\dd\mu(y)}{\sqrt{\eta^h_\mu(x)\eta^h_\mu(y)}} \geq \frac{h^{-2}}{2C_\D^{1/2}}  \frac{\mu(B(x,h/2))}{\mu(B(x,h))} \geq \frac{h^{-2}}{2C_\D^{3/2}},
    \end{equation}
    which gives \eqref{eq:ess-spectrum}.
\end{proof}

Recall the convention that we use: $\lambda_{\infty,\mu}^h$ is  the minimum of the bottom of the essential spectrum and of $\|\Delta_\mu^h\|_{\L^2(\mu)}$. 
Therefore, the two previous lemmas  imply that $\lambda_{\infty,\mu}^h\geq C_\D^{-8}h^{-2}$. 
Below $\lambda_{\infty,\mu}^h$, the spectrum of $-\Delta^h_\mu$ consists of isolated eigenvalues of finite multiplicity, which are the values $\lambda^h_{k,\mu}$ produced by the Courant–Fischer–Weyl min-max principle.  This principle states that whenever $\lambda_{\ell,\mu}^h<\lambda_{\infty,\mu}^h$, 
\begin{equation}\label{eq:minmax_principle}
	\lambda_{\ell,\mu}^h = \min_{\dim V= \ell+1} \max_{u\in V\backslash\{0\}} \frac{\dotp{-\Delta_\mu^h u,u}_\mu}{\|u\|_{\L^2(\mu)}^2},
\end{equation}
where the minimum is taken over all $(\ell+1)$-dimensional subspaces $V$ of $\L^2(\mu)$, see, e.g., \cite[Theorem XIII.1]{reed1978iv}. 

We let $\omega:t\mapsto \sup_{|t_1-t_2|\leq t} |\eta(t_1)-\eta(t_2)|$ be the modulus of continuity of $\eta$.

\begin{proposition}\label{prop:burago}
	Let $C_\D,\kappa\geq 1$, $C_{\PI},h>0$ and let $\eps\in [ 0,h/96]$, $\tau \in [0,1]$. Consider   a measure $\mu$ in $\PI_{h/48}(C_\D,C_{\PI}, \kappa)$ and let $\nu$ be a measure $(\eps,\tau)$-close to $\mu$. Let $\ell\geq 1$ and define $\Theta = \tau+ \omega(\frac{2\eps}h)+ h^2 \lambda_{\ell,\mu}^h$. Then there is a constant $C$ depending on $C_\D$ and $\kappa$ such that whenever $C\Theta\leq 1$, there holds $|\lambda_{\ell,\nu}^h-\lambda_{\ell,\mu}^h|\leq C\Theta \lambda_{\ell,\mu}^h$.
\end{proposition}
%Note that the  proposition is true even when $\lambda_{\ell,\mu}^h=\lambda_{\infty,\mu}^h$. 

The proof of \Cref{prop:burago} is based on a common technique in the literature on graph Laplacians, which relies on the use  of the  min-max principle \eqref{eq:minmax_principle}. 
%Assume for the sake of simplicity that the infimum is attained at some $V=V_{\ell,\mu}^h$ (this is the case if $\lambda_{\ell,\mu}^h$ is strictly smaller than the bottom of the essential spectrum).
Let $V_{\ell,\mu}^h$ be the corresponding minimizer. 
 If one constructs a linear map $R: V_{\ell,\mu}^h\to \L^2(\nu)$ that is close to being an isometry and that almost preserves the Dirichlet energy, then \eqref{eq:minmax_principle} can be used with $V=R V_{\ell,\mu}^h$ to show that $\lambda_{\ell,\nu}^h$ is not much larger than $   \lambda_{\ell,\mu}^h$. 
More precisely, assume that there are  $a_1,a_2\geq 1$ with
	\begin{equation}\label{eq:prop_R}
		\forall u\in V_{\ell,\mu}^h,\ \|Ru\|^2_{\L^2(\nu)}\geq \frac{\|u\|^2_{\L^2(\mu)}}{a_1}\ \text{ and }\ \dotp{-\Delta_{\nu}^hRu,Ru}_\nu\leq a_2\dotp{-\Delta_{\mu}^hu,u}_\mu.
	\end{equation}
	Then, \eqref{eq:minmax_principle} yields
	\begin{align}
		\lambda_{\ell,\nu}^h &\leq  \max_{v\in RV_{\ell,\mu}^h\backslash\{0\}} \frac{\dotp{-\Delta_{\nu}^hv,v}_\nu}{\|v\|_{\L^2(\nu)}^2} = \max_{u\in V_{\ell,\mu}^h\backslash\{0\}} \frac{\dotp{-\Delta_{\nu}^hRu,Ru}_\nu}{\|Ru\|_{\L^2(\nu)}^2} \nonumber \\
		&\leq a_1a_2 \max_{u\in V_{\ell,\mu}^h\backslash\{0\}} \frac{\dotp{-\Delta_{\mu}^h u,u}_\mu}{\|u\|_{\L^2(\mu)}^2} = a_1a_2 \lambda_{\ell,\mu}^h. \label{eq:minmax_application}
	\end{align}

Many methods have been proposed to construct appropriate maps $R$ based, e.g., on  heat interpolation \cite{cheng2022convergence}. We will rely on a linear map $R$ built thanks to  optimal transport; such an idea was already successfully used in \cite{burago2019spectral} and \cite{garcia_trillos_error_2020}. More precisely, if $\mu$ and $\nu$ are $(\eps,\tau)$-close, there are measures $\tilde \mu$ and $\tilde \nu$ with $e^{-\tau}\mu\leq \tilde \mu\leq \mu$ and $e^{-\tau}\nu\leq \tilde\nu\leq \nu$ and a transport plan $\pi$ between $\tilde \mu$ and $\tilde \nu$ with $\cC(\pi)\leq \eps$. We let $\bar\pi$ be the transport plan between $\tilde \nu$ and $\tilde \mu$ (switching $x$ and $y$) and define  $R$ as the $\bar \pi$-interpolation function $\L^2(\mu)\to \L^2(\nu)$ defined in \Cref{sec:PI}. We call $R$ an interpolation map associated with the pair $(\mu,\nu)$, realizing the 
$(\eps,\tau)$-closeness.

\begin{lemma}\label{lem:perturbation_kernel}
	Let $C_\D\geq 1$, $h>0$ and $\tau \in [0,1]$. Let $\mu\in \mathrm{DM}_{h}(C_\D)$ and let $\nu$ be  $(\eps,\tau)$-close  to $\mu$ for some $\eps<h/2$ such that $2C_\D e  \omega(2\eps/h)<1$. 
	Let $x_1,x_2,y_1,y_2\in \X$ with $\rho(x_i,y_i)\leq \eps$ for $i=1,2$. Then, 
	\begin{equation}
		K_{\nu}^h(y_1,y_2) \leq C(\tau,\eps)\p{ 	 K_{\mu}^h(x_1,x_2)+ \frac{h^{-2}\omega\p{\frac{2\eps}h}\one_{\rho(x_1,x_2)< 2h} }{\sqrt{\eta_\mu^h(x_1)\eta_\mu^h(x_2)}}}
	\end{equation}
	where $C(\tau,\eps)= \p{e^{-\tau}-2C_\D  \omega\p{\frac{2\eps}h}}^{-1}$.
\end{lemma}
\begin{proof}
	
	Let $\tilde \mu$ and $\tilde \nu$ be such that $e^{-\tau} \mu\leq \tilde \mu\leq \mu$ and $e^{-\tau}\nu\leq \tilde \nu\leq \nu$, and let $\pi$ be a transport plan between $\tilde \mu$ and $\tilde \nu$ of cost smaller than $\eps$. 
	Then, 
	\begin{align*}
		\eta_\nu^h(y_1) &\geq \int \eta\p{\frac{\rho(y_1,y)}h} \dd \tilde \nu(y) =   \iint \one_{\rho(y_1,y)< h} \eta\p{\frac{\rho(y_1,y)}h} \dd \pi(x,y).
	\end{align*}
	If $\rho(x,y)\leq \eps$, then $\eta(\rho(y_1,y)/h) \geq \eta(\rho(x_1,x)/h) -\omega(2\eps/h)$. Moreover, $\one_{\rho(y_1,y)< h}\geq \one_{\rho(x_1,x)< h-2\eps}$. Thus,
	\begin{align*}
		\eta_\nu^h(y_1) &\geq  \int \one_{\rho(x_1,x)< h-2\eps} \p{\eta\p{\frac{\rho(x_1,x)}h}-\omega\p{\frac{2\eps}h}} \dd \tilde \mu(x).
	\end{align*}
	Observe that if $\rho(x_1,x)\in (h-2\eps,h)$, then, since $\eta(1)=0$, there holds $\eta(\rho(x_1,x)/h)\leq \omega(2\eps/h)$. Therefore, we have
	\begin{align*}
		 \eta_\nu^h(y_1) 	&\geq \int \one_{\rho(x_1,x)< h-2\eps} \eta\p{\frac{\rho(x_1,x)}h} \dd  \tilde\mu(x) - \omega\p{\frac{2\eps}h} \tilde\mu(B(x_1,h-2\eps))\\
		&\geq e^{-\tau}\eta_\mu^h(x_1) - \int \one_{\rho(x_1,x)\in [h-2\eps,h)}\eta\p{\frac{\rho(x_1,x)}h} \dd \tilde \mu(x) - \omega\p{\frac{2\eps}h}\tilde \mu(B(x_1,h-2\eps)) \\
		&\geq e^{-\tau}\eta_\mu^h(x_1) -  \omega\p{\frac{2\eps}h} \tilde\mu(B(x_1,h))\geq e^{-\tau}\eta_\mu^h(x_1) -  \omega\p{\frac{2\eps}h} \mu(B(x_1,h)).
	\end{align*}
	But, since $\eta(t)\geq 1/2$ for $t\in [0,1/2]$, there holds $\mu(B(x_1,h)) \leq C_\D \mu(B(x_1,h/2))\leq 2 C_\D\eta_\mu^h(x_1)$. We finally find that
	\[
	\eta_\nu^h(y_1) \geq \eta_\mu^h(x_1)\p{e^{-\tau}-2C_\D  \omega\p{\frac{2\eps}h}}.
	\]
	The same inequality holds for the pair $(x_2,y_2)$.
     Since there also holds  $\eta(\rho(y_1,y_2)/h)\leq \eta(\rho(x_1,x_2)/h) + \omega(2\eps/h)$, we find that
	\[
	K_\nu^h(y_1,y_2) = \frac{h^{-2}\eta\p{\frac{\rho(y_1,y_2)}h}}{\sqrt{\eta_\nu^h(y_1)\eta_\nu^h(y_2)}} \leq h^{-2} \p{e^{-\tau}-2C_\D  \omega\p{\frac{2\eps}h}}^{-1}\frac{\eta\p{\frac{\rho(x_1,x_2)}h} +\omega\p{\frac{2\eps}h}}{\sqrt{\eta_\mu^h(x_1)\eta_\mu^h(x_2)}}.
	\]
	Finally, if $\rho(x_1,x_2)\geq 2h$, then $\rho(y_1,y_2)\geq h$ and $K_\nu^h(y_1,y_2)=0$, so the desired inequality holds.
\end{proof}

\begin{lemma}\label{lem:almost_isometry}
	Let $C_\D\geq 1$, $h>0$ and $\tau\in [0,1]$. Let $\mu\in \mathrm{DM}_{h/8}(C_\D)$ and let $\nu$ be  $(\eps,\tau)$-close  to $\mu$ for some $\eps \leq h/16$. Then, letting $R$ be an interpolation map associated with the pair $(\mu,\nu)$, realizing the 
$(\eps,\tau)$-closeness, there holds 
	\[
	e^{-\tau}   \|Ru\|_{\L^2(\nu)}^2 \leq \|u\|^2_{\L^2(\mu)} \leq e^{\tau}(\|Ru\|^2_{\L^2(\nu)} + 4C_\D^5 h^2 \dotp{-\Delta_\mu^hu,u}_\mu).
	\]
\end{lemma}

\begin{proof}
    Recall  from \Cref{sec:PI} that $R=P_2\iota_1$ can be seen as an orthogonal projection. Therefore, 
\[ \|Ru\|_{\L^2(\nu)}^2 \leq e^\tau \|P_2 \iota_1 u\|^2_{\L^2(\pi)} \leq e^\tau \|\iota_1 u\|^2_{\L^2(\pi)}\leq e^\tau \|u\|^2_{\L^2(\mu)}.\]
For the reverse inequality, we use Pythagoras' theorem:
\begin{align*}
	\|u\|^2_{\L^2(\tilde\mu)} = \|\iota_1 u\|^2_{\L^2(\pi)} = \|Ru\|^2_{\L^2(\tilde\nu)} + \|(1-P_2)\iota_1u\|^2_{\L^2(\pi)},
\end{align*}
where $(1-P_2)$ is the orthogonal projection onto the orthogonal subspace of $\cH_2$ in $\L^2(\pi)$, or, equivalently, $\|(1-P_2)\iota_1u\|^2_{\L^2(\pi)}$ is the squared distance between $\iota_1u$ and $\cH_2$. Thus, for any function $v\in \L^2(\tilde \nu)$, 
\[ \|(1-P_2)\iota_1u\|^2_{\L^2(\pi)}\leq \iint (u(x)-v(y))^2 \dd \pi(x,y).
\]
Let $\tilde h=h/4$. 
We let $v:y\mapsto \eta_\mu^{\tilde h} (y)^{-1} \int \eta(\rho(x',y)/\tilde h)u(x')\dd \mu(x')$. Then, by Jensen's inequality,
\[
\iint (u(x)-v(y))^2 \dd \pi(x,y) \leq \iiint \eta_\mu^{\tilde h}(y)^{-1} \eta\p{\frac{\rho(x',y)}{\tilde h}}(u(x)-u(x'))^2 \dd \mu(x')\dd \pi(x,y). 
\]
Observe that $\eta(\rho(x',y)/\tilde h)\leq \one_{\rho(x',y)\leq \tilde h}\leq \one_{\rho(x',x)\leq h/2}$ since $\tilde h+\eps\leq h/2$.  But we have $\eta(t)\geq 1/2$ for $t\in [0,1/2]$. Thus, $\eta(\rho(x',y)/\tilde h)\leq 2 \eta(\rho(x',x)/h)$. Moreover,
\[
\eta_\mu^{\tilde h}(y) \geq \frac 12 \mu(B(y,\tilde h/2))\geq \frac{1}{2} \mu(B(x,\tilde h/2-\eps))\geq \frac{C_\D^{-4}}2 \mu(B(x, 8\tilde h-16\eps)).
\]
Since $8\tilde h-16\eps=2h-16\eps\geq h$, $\eta_\mu^{\tilde h}(y)\geq C_\D^{-4} \mu(B(x,h))/2 \geq C_\D^{-4} \eta_\mu^h(x)/2$. But also, as $\rho(x,x')\leq \eps+\tilde h\leq h/2$, 
\[
\mu(B(x,h))\geq \mu(B(x', h/2))\geq C_\D^{-1} \mu(B(x',h))\geq C_\D^{-1} \eta_\mu^h(x').
\]
Thus, $\eta_\mu^{\tilde h}(y) \geq C_\D^{-5} \sqrt{\eta_\mu^{h}(x)\eta_\mu^h(x')}/2$ and 
\[
\iint (u(x)-v(y))^2 \dd \pi(x,y) \leq 4C_\D^5h^2 \iint K_\mu^h(x,x')(u(x)-u(x'))^2 \dd \mu(x)\dd \mu(x').
\]
In total, we have proven that
\begin{equation}
	\|u\|^2_{\L^2(\mu)} \leq e^{\tau}(\|Ru\|^2_{\L^2(\tilde \nu)} + 4C_\D^5 h^2 \dotp{-\Delta_\mu^hu,u}_\mu).
\end{equation}
Since $\tilde\nu\leq\nu$, the conclusion follows.
\end{proof}

\begin{lemma}\label{lem:energy_comparison}
		Let $C_\D,\kappa\geq 1$, $C_{\PI}, h>0$ and $\tau\in [0,1]$. Let $\mu\in \mathrm{PI}_{h/4}(C_\D,C_{\PI},\kappa)$ and let $\nu$ be  $(\eps,\tau)$-close to $\mu$ for some $\eps<h/2$ such that $2C_\D  e \omega(2\eps/h)<1$. Then, letting $R$ be an interpolation map associated with the pair $(\mu,\nu)$, realizing the 
$(\eps,\tau)$-closeness, there exists a constant $C$ depending on $C_\D$ and $\kappa$ such that
	\[
	\dotp{-\Delta_\nu^hRu,Ru}_{\nu}\leq e^{2\tau}C(\tau,\eps)\p{1+C\omega\p{\frac{2\eps}h}}\dotp{-\Delta_\mu^hu,u}_\mu.
	\]
\end{lemma}

\begin{proof}
We apply the third item in \Cref{lem:properties_R} twice (once on $y_1$ and once on $y_2$) to obtain that
\begin{align*}
	\dotp{-\Delta_{\nu}^h Ru,Ru}_{\nu} &=  \iint K_{\nu}^h(y_1,y_2) (Ru(y_1)-Ru(y_2))^2 \dd \nu(y_1)\dd \nu(y_2) \\
    &\leq  e^{2\tau}\iiiint K_{\nu}^h(y_1,y_2) (u(x_1)-u(x_2))^2 \dd \pi(x_1,y_1)\dd \pi(x_2,y_2). 
\end{align*}
We then apply \Cref{lem:perturbation_kernel} to obtain that
\begin{align*}
    &\dotp{-\Delta_{\nu}^h Ru,Ru}_{\nu}\\
	& \leq e^{2\tau} C(\tau,\eps)\int \p{K_{\mu}^h(x_1,x_2)+\frac{h^{-2}\omega\p{\frac{2\eps}h}\one_{\rho(x_1,x_2)< 2h}}{\sqrt{\eta_\mu^h(x_1)\eta_\mu^h(x_2)}}} (u(x_1)-u(x_2))^2 \dd \tilde\mu(x_1)\dd \tilde\mu(x_2)
\end{align*}
Since $\tilde\mu\leq\mu$, 
this integral is decomposed into two terms, the first one yielding exactly $\dotp{-\Delta_\mu^hu,u}_\mu$. The second one is bounded in the next lemma.
\end{proof}

\begin{lemma}\label{lem:indicator_energy}
	Let $C_\D,\kappa\geq 1$, $C_{\PI},h>0$. Let $\mu$ be a measure in  $\PI_{h/4}(C_\D,C_{\PI}, \kappa)$ and let $u\in \L^2(\mu)$. Then, there exist absolute constants $a$ and $b$  such that 
	\begin{equation}\label{eq:indicator_energy}
		\iint_{\rho(x,y)< 2h}\frac{h^{-2}}{\sqrt{\eta_{\mu}^h(x)\eta_{\mu}^h(y)}}(u(x)-u(y))^2 \dd \mu(x)\dd \mu(y) \leq  a C_\D^{b\lceil \log(1+ \kappa)\rceil}\dotp{-\Delta_{\mu}^hu,u}_\mu.
	\end{equation}
\end{lemma}

\begin{proof}
	Fix $x$ and $y$ with $\rho(x,y)< 2 h$, and consider a path $x_1=x,x_2,\dots,x_N=y$ as in \Cref{lem:quasiconvex} with $t=2h$, $r=h/4$, so that all the $x_i$s are in $B(x,(4\kappa+1/4)h)$. Let $B_i$ be the ball centered at $x_i$, of radius $h/8$. If $z_i\in B_i$ and $z_{i+1}\in B_{i+1}$, then $\rho(z_i,z_{i+1})\leq h/2$. Using \eqref{eq:eta} and Jensen's inequality, we find that for all $(z_2,\dots,z_{N-1})\in B_2\times \cdots\times B_{N-1}$,
	\begin{align*}
		(u(x)-u(y))^2 &\leq N \sum_{i=1}^{N-1} (u(z_i)-u(z_{i+1}))^2  \leq 2N \sum_{i=1}^{N-1} \eta\p{\frac{\rho(z_i,z_{i+1})}{h}}(u(z_i)-u(z_{i+1}))^2,
	\end{align*}
	with $N\leq 2C_\D^{\lceil\log_2(5+64\kappa)\rceil}$. 
	Averaging, we find that
	\begin{equation}\label{eq:sum_of_squares}
		\begin{split}
			&(u(x)-u(y))^2 \leq  \frac{2N}{\mu(B_2)} \int_{B_2}\eta\p{\frac{\rho(x,z_{2})}{h}}(u(x)-u(z_{2}))^2 \dd \mu(z_2) \\
			&\quad + \sum_{i=2}^{N-2} \frac{2N}{\mu(B_i)\mu(B_{i+1})} \iint_{B_i\times B_{i+1}}\eta\p{\frac{\rho(z_i,z_{i+1})}{h}}(u(z_i)-u(z_{i+1}))^2 \dd \mu(z_i)\dd \mu(z_{i+1}) \\
			&\quad + \frac{2N}{\mu(B_{N-1})} \int_{B_{N-1}}\eta\p{\frac{\rho(z_{N-1},y)}{h}}(u(z_{N-1})-u(y))^2 \dd \mu(z_{N-1}).
		\end{split}
	\end{equation}
   Let us bound  the first integral appearing in the right-hand side of \eqref{eq:sum_of_squares}. Let $z_2\in B_2$. Since $\eta_\mu^h(y) \geq \mu(B(y,h/2))/2$ and $\rho(y,z_2) \leq (4\kappa+3)h$, the 
	 doubling property implies that there exists some integer $a>0$ such that for all $z_2\in B_2$, $\eta_\mu^h(y)\geq C_\D^{-a\lceil \log(1+ \kappa)\rceil}\eta_\mu^h(z_2)/2$. Moreover, the doubling property also implies a control of the form $\mu(B_2)\geq C_\D^{-b\lceil \log(1+ \kappa)\rceil}\mu(B(x,2h))$ for some integer $b$. 
	Thus,
	\begin{align*}
		& \frac{1}{\mu(B_2)} \int_{B_2}  \frac{h^{-2}}{\sqrt{\eta_\mu^h(x)\eta_\mu^h(y)}} \eta\p{\frac{\rho(x,z_{2})}{h}} (u(x)-u(z_2))^2 \dd \mu(z_2) \\
		&\leq \frac{2C_\D^{a\lceil \log(1+ \kappa)\rceil/2+b\lceil \log(1+ \kappa)\rceil}}{\mu(B(x,2h))}\int  \frac{h^{-2}}{\sqrt{\eta_\mu^h(x)\eta_\mu^h(z_2)}} \eta\p{\frac{\rho(x,z_{2})}{h}} (u(x)-u(z_2))^2 \dd \mu(z_2).
	\end{align*}
    We integrate this bound against $\one_{\rho(x,y)< 2h}\dd\mu(x)\dd\mu(y)$ to obtain
    \begin{align*}
      &\iiint_{\rho(x,y)< 2h}  \frac{h^{-2}}{\mu(B(x,2h))\sqrt{\eta_\mu^h(x)\eta_\mu^h(z_2)}}   \eta\p{\frac{\rho(x,z_{2})}{h}} (u(x)-u(z_2))^2 \dd \mu(z_2)\dd\mu(x)\dd\mu(y) \\
      &\leq \iint  \frac{h^{-2}}{\sqrt{\eta_\mu^h(x)\eta_\mu^h(z_2)}} \eta\p{\frac{\rho(x,z_{2})}{h}} (u(x)-u(z_2))^2 \dd \mu(z_2)\dd\mu(x),
    \end{align*}
    which is equal to $ \dotp{-\Delta_\mu^hu,u}_\mu$. 
	We bound in a similar way the other terms in \eqref{eq:sum_of_squares} by quantities, which, when integrated against $\one_{\rho(x,y)< 2h}\dd\mu(x)\dd\mu(y)$, are proportional to $\dotp{-\Delta_\mu^h u,u}_\mu$. Since the integral of the left-hand side of \eqref{eq:sum_of_squares} against $\one_{\rho(x,y)< 2h}\dd\mu(x)\dd\mu(y)$ is exactly the left-hand side of \eqref{eq:indicator_energy}, the conclusion follows.
\end{proof}

\begin{proof}[Proof of \Cref{prop:burago}]
Let $\mu$ and $\nu$ be as in the statement of the proposition. According to \Cref{lem:operator_norm} and \Cref{lem:ess-spectrum}, if the constant $C$ in \Cref{prop:burago} is large enough, then $\lambda_{\ell,\mu}^h<\lambda_{\infty,\mu}^h$. In particular, the min-max principle applies. Let $V$ be an $(\ell+1)$-dimensional subspace with $\dotp{-\Delta_\mu^hu,u}_\mu \leq  \lambda_{\ell,\mu}^h\|u\|^2_{\L^2(\mu)}$ for all $u\in V$. According to \Cref{lem:almost_isometry}, there holds
\[
\|u\|^2_{L^2(\mu)} \leq e^{\tau}(\|Ru\|^2_{\L^2(\nu)} +4C_\D^5h^2 \lambda_{\ell,\mu}^h\|u\|^2_{\L^2(\mu)}).
\]
Thus, if $8e^{\tau}C_\D^5h^2 \lambda_{\ell,\mu}^h<1$ (which is the case if $\Theta$ is small), $R$ is injective on $V$, with
\[
\|Ru\|^2_{\L^2(\nu)} \geq \|u\|^2_{\L^2(\mu)}(e^{-\tau} -4C_\D^5h^2 \lambda_{\ell,\mu}^h),
\] implying in particular that $RV\subseteq \L^2(\nu)$ is of dimension $\ell+1$. Moreover, for $u\in V$ nonzero, we may use \Cref{lem:energy_comparison} to obtain
\begin{align*}
    \frac{\dotp{-\Delta_\nu^hRu,Ru}_\nu}{\|Ru\|_{\L^2(\nu)}^2} &\leq \frac{e^{2\tau}C(\tau,\eps)\p{1+a C_\D^{b\lceil \log (1+ \kappa)\rceil}\omega\p{\frac{2\eps}h}}\dotp{-\Delta_\mu^hu,u}_\mu}{(e^{-\tau} -4C_\D^5h^2 \lambda_{\ell,\mu}^h)\|u\|^2_{\L^2(\mu)}} \\
    &\leq \frac{e^{2\tau}C(\tau,\eps)\p{1+a C_\D^{b\lceil \log(1+ \kappa)\rceil}\omega\p{\frac{2\eps}h}}}{(e^{-\tau} -4C_\D^5h^2  \lambda_{\ell,\mu}^h)}\lambda_{\ell,\mu}^h.
\end{align*}
Thus, the min-max principle implies that $\lambda_{\ell,\nu}^h$ is  smaller than the right-hand side of the above equation. Reciprocally, according to \Cref{lem:stab_condition}, since $\eps\leq h/96$, the measure $\nu$ is in $\PI_{h/16}(e^{2\tau}C_\D^3,C'_{\PI},2\kappa)$ with $C'_{\PI}= 64 e^{2\tau}C_{\PI}+e^\tau$. Moreover, we have already proved that $\lambda_{\ell,\nu}^h$ is not much bigger than $\lambda_{\ell,\mu}^h$, and, consequently, is at most of order $h^{-2}$, strictly below $\lambda_{\infty,\nu}^h$ according to \Cref{lem:operator_norm} and \Cref{lem:ess-spectrum}. Thus, we may apply the same argument with the roles of $\mu$ and $\nu$ swapped to upper bound $\lambda_{\ell,\mu}^h$ (with slightly worse constants). Finally,  expand to  first order the constants (which we can always do since $\tau$, $\omega(2\eps/h)$ and $h^2\lambda_{\ell,\mu}^h$ are small under the assumption $C\Theta\leq 1$ for large $C$) to obtain a bound as in the statement of \Cref{prop:burago}.
\end{proof}

\section{Multiscale Poincaré inequalities for weighted Laplacian operators}\label{sec:poincare}

Let $\Delta$ be the Laplace operator on the circle $\S^1$. The (homogeneous) $\dot\H^1$-seminorm of a smooth function $v$ on $\S^1$ is defined as 
\[
\|v\|^2_{\dot\H^1} = \dotp{-\Delta v,v} = \int_{\S^1} |\nabla v|^2.
\]
The (inhomogeneous) dual norm $\|\cdot\|_{ \H^{-1}}$ on the circle can be characterized up to multiplicative constants in many different ways \cite{triebel}. It is equivalent to the Besov norm of regularity $\mathrm{B}^{-1}_{2,2}$, which can (for instance) be expressed using a wavelet basis. In other words, there exists an orthonormal basis $(\psi_{j,k})_{j,k}$ indexed by $j\geq 0$ and $k\in [L_j]$ for $L_j$ of order $2^j$ such that
\begin{equation}\label{eq:equivalent_norm}
\|v\|^2_{\H^{-1}}  \asymp \sum_{j\geq 0} 2^{-2j} \sum_{k=1}^{L_j} |\dotp{v,\psi_{j,k}}|^2.
\end{equation}
Besides being orthogonal, the functions $\psi_{j,k}$ have good localization properties, e.g., their support is roughly of size $2^{-j}$ and their supports have a bounded overlap for a given level $j$. Such constructions are not limited to the circle $\S^1$, but are possible on any compact Riemannian manifold.

Constructing  (quasi-)orthonormal families of functions having similar properties has been of considerable interest to analysts, starting with David \cite{david1988morceaux} and Christ \cite{christ1990b} who introduced systems of dyadic cubes on doubling metric measure spaces. When $\mu$ is a doubling measure, Auscher, Hytönen and Tapiola  \cite{auscher2013orthonormal, hytonen2014almost} rely on randomization techniques using Christ cubes as a building block to construct orthonormal wavelet-like bases of $\L^2(\mu)$ that are almost-Lipschitz continuous, in the sense that they can be made $\vartheta$-H\"older continuous for $\vartheta$ arbitrarily close to $1$.

The goal of this section is to establish an upper bound on the dual norm 
\[
\|v\|_{\H^{-1}_{h,s}(\mu)} =   \sup\{\dotp{v,u}_\mu:\ \int u\dd \mu=0,\ \dotp{-\Delta_\mu^h u,u}_\mu  + s\|u\|^2_{\L^2(\mu)}\leq 1\}
\]
similar to the one in \eqref{eq:equivalent_norm}.  We will actually obtain a slightly worse bound, and bound morally a $\H^{-1}$-norm by a $\B^{-1}_{2,1}$-norm taking the form of 
\[
 \sum_{j\geq 0} 2^{-j}\p{ \sum_{k=1}^{L_j} |\dotp{v,\psi_{j,k}}|^2}^{1/2}.
\]

A bound of this type for $s=0$ (which \cite{trillos2025minimax} call a multiscale Poincaré inequality) is the crux of their proof that graph Laplacian eigenvalues approximate those of weighted Laplace operators. Their version (Proposition 3.6 in \cite{trillos2025minimax}) rests on a dyadic cube decomposition. The discontinuity of the indicator  functions $\psi_{j,k}$ of the dyadic cubes that appear in their bound creates many  problems later on, that can only be handled with great technical care by relying on the geometry of the underlying manifold. Instead, we show that the ``almost-Lipschitz'' constructions of Auscher, Hytönen and Tapiola can be used to our advantage, and considerably simplify the proofs.

There is another difference between our approach and the one in  \cite{trillos2025minimax}: we add a parameter $s>0$ in the definition of the dual norm. First, let us observe that this modification is apparently inconsequential. Indeed, it is clear that $\|v\|_{\H^{-1}_{h,s_2}(\mu)}\leq \|v\|_{\H^{-1}_{h,s_1}(\mu)}$ if $s_1\leq s_2$. Since, as we will soon show, there also holds an inequality of the form $\|u\|^2_{\L^2(\mu)}\lesssim \dotp{-\Delta_\mu^h u,u}_\mu$ (\Cref{rem:spectral_gap}), a reverse inequality also holds: all the norms $\H^{-1}_{h,s}(\mu)$ are mutually equivalent. 
Still, when bounding the relative approximation error for a large eigenvalue $\lambda=\lambda_\mu^h$, bounding the inhomogeneous  norm $\H^{-1}_{h,\lambda}(\mu)$ (related to the Bessel potential $(\lambda-\Delta_\mu^h)^{-1/2}$) instead of the homogeneous negative norm $\H^{-1}_{h,0}(\mu)$ (related to the Riesz potential $(-\Delta_\mu^h)^{-1/2}$) will turn out to be key to obtain the right dependency with respect to $\lambda$ in our bounds. 

Let $\delta\in (0,1)$ be a (small) parameter. We fix now, and once and for all, a family $(\X^j)_{j\geq 0}$ of subsets of $\X$, defined in the following way. 
 For $j= 0$, let $\X^0$ be a maximal $1$-separated collection of points in $\X$. Inductively, for $j\geq 1$, let $\X^j=\{x^j_k\}_{k\in [L_j]} \supseteq \X^{j-1}$ be a maximal $\delta^j$-separated collection of points in $\X$. 

 \begin{definition}\label{def:splines}
    A $\delta$-spline system on $(\X,\rho)$ up to depth $J$ is a family of nonnegative functions $s^j_k$, for $k\in [L_j]$, $0\leq j\leq J$ that satisfy for all $k,m\in [L_j]$, $0\leq j\leq J$,
    \begin{itemize}
        \item (Localization)  $\one_{B(x_k^j,\delta^j/8)}\leq s_k^j\leq \one_{B(x_k^j, 8\delta^j)}$;\footnote{The constants $1/8$ and $8$ play no special roles, and could be replaced by constants $1/C$ and $C$ for some other $C>0$.}
        \item (Reproducing properties) $s_k^j(x_m^j)=\one_{k=m}$, $\sum_{k=1}^{L_j} s_k^j=1$, and $s_k^j$ is a linear combination of finitely many functions $s_m^{j+1}$ if $j<J$.
    \end{itemize}
 \end{definition}
 %It is not complicated to show that spline systems exist by defining the $s_k^j$s as indicator functions of increasingly fine partitions subordinate to the sets $\X^j$. Exhibiting H\"older continuous   $\delta$-spline systems is much less obvious. 

  \begin{definition}
 		Let $r\geq 0$ and let $M\geq 1$ be an integer. We say that a metric space $(\X,\rho)$ is doubling at scale $r$ with doubling constant $M$ if every open ball of radius $t\geq r$ can be covered by at most  $M$ open balls of radius $t/2$. When $r=0$, we say that the metric space is doubling.
 	\end{definition}
 	
 	\begin{lemma}\label{lem:doubling_metric}
 		Let $C_\D\geq 1$ and $r>0$. 
 		Let $\mu$ be a measure in $\mathrm{DM}_r(C_\D)$. Then, the metric space $(\X,\rho)$ is doubling at scale $2r$ with doubling constant $C_\D^4$. 
 	\end{lemma}
 	
 	\begin{proof}
 		Let $B$ be an open ball of radius $t'\geq 2r$ in $\X$. Consider a maximal $(t'/2)$-separated set $A$ in  $B$. Then, the union of open balls of radius $t'/2$, centered at  points of $A$ cover $B$. Otherwise, we could find a point $x\in B$ with $\rho(x,y)\geq t'/2$ for all $y\in A$, contradicting the maximality of the set $A$. The size of such a set is at most $C_\D^4$ according to \Cref{lem:covering_ball}.
 	\end{proof}
  It was proven  by Auscher, Hytönen and Tapiola in \cite{auscher2013orthonormal, hytonen2014almost} that  on any doubling metric space $(\X,\rho)$ and for any $\vartheta\in (0,1)$, there exists a $\vartheta$-H\"older continuous $\delta$-spline system, as long as the scale parameter $\delta$ is small enough. We show in \Cref{sec:spline} that their construction can be adapted on coarse doubling metric spaces with appropriate modifications. Since we will assume that there exists at least one measure $\mu$ in $ \mathrm{DM}_r(C_\D)$, \Cref{lem:doubling_metric} implies the existence of spline systems. Let us observe that the results of this section hold for any spline systems; additional regularity properties of the splines will become useful in \Cref{sec:concentration} only.

We fix  any spline system on $(\X,\rho)$, with $J$  such that $\delta^J\geq 2r$. We then define $V_j$ to be the vector space  spanned by the functions $s^j_k$ for $k\in [L_j]$. The reproducing properties imply that the vector spaces $V_j$ are nested: $V_{j}\subseteq V_{j+1}$ if $0\leq j\leq J-1$. Let us stress that this whole construction depends only on the metric space $(\X,\rho)$, and not on any choice of measure on $(\X,\rho)$.  
For $0\leq j\leq J $ and $\mu$ a measure, we define the orthogonal projection operator $\Pi_{j,\mu}:\L^2(\mu)\to\L^2(\mu)$ onto $V_j$. 

It will be convenient to assume for the next proposition that $\diam(\X)<1$, which implies that $\X^0$ contains a single point. Consequently, $\Pi_{0,\mu}u$ is the constant function equal to $\frac{1}{\mu(\X)}\int u\dd\mu$. Since we will eventually assume that there is a probability measure $\mu$ with support $\X$, the space $\X$ is necessarily bounded (\Cref{rem:bounded}). Thus, the assumption $\diam(\X)<1$ is not restrictive up to rescaling. Alternatively, we could have chosen to define the sets $\X^j$ as $(D\delta^j)$-separated sets where $D=\diam(\X)/2$.

\begin{proposition}\label{prop:multiscale_poincare}
    Assume that $\diam(\X)<1$.  Let $C_\D,\kappa\geq 1$, $C_{\PI},h>0$ and $\delta\in (0,1)$, $s\geq 0$.  There exist  $C>0$ depending on $\kappa$, $C_\D$ and $\delta$ such that the following holds. Fix $\alpha=\delta/256$.
      Let $\mu\in \PI_{ \alpha h}(C_\D,C_{\PI},\kappa)$   and let $v\in \L^2(\mu)$.    If $J$ is such that $2\alpha h\leq \delta^J \leq h/128 $ and if $(\Pi_{j,\mu})_{0\leq j\leq J}$ are the orthogonal projection operators associated with an arbitrary $\delta$-spline system on $(\X,\rho)$, then 
     \begin{equation}
        \|v\|_{ \H^{-1}_{h,s}(\mu)}\leq C  \sqrt{1+C_{\PI}} \p{h \|v\|_{\L^2(\mu)} + \sum_{j=1}^J \min(\delta^j,s^{-1/2}) \|\Pi_{j,\mu} v\|_{\L^2(\mu)}}.
     \end{equation}
\end{proposition}
In the above proposition, we simply let $\min(t,s^{-1/2})=t$ if $s=0$. The reader should keep in mind that  \Cref{prop:multiscale_poincare} holds for any coarse PI measure: in \Cref{sec:concentration}, it will be applied to the empirical measure $\mu_n$.  The proposition is a direct consequence of the following Jackson estimate.
\begin{proposition}\label{prop:poincare_spline}
    Under the same assumptions as in \Cref{prop:multiscale_poincare}, if $u\in \L^2(\mu)$ and $0\leq j\leq J$, then
    \[
    \|(1-\Pi_{j,\mu})u\|^2_{\L^2(\mu)} \leq   C (1+C_{\PI}) \delta^{2j} \dotp{-\Delta_\mu^h u,u}_\mu,
    \]
    for some constant $C$ depending on $\delta$, $\kappa$ and $C_\D$.
\end{proposition}

\begin{remark}\label{rem:spectral_gap}
For $j=0$, there holds $\Pi_{0,\mu}u= \frac{1}{\mu(\X)}\int u\dd \mu$. The proposition then claims that for any $u\in \L^2(\mu)$, $\Var_\mu(u) \leq C(1+C_{\PI})  \dotp{-\Delta_\mu^h u,u}_\mu$, namely that $-\Delta_\mu^h$ has a spectral gap of order $(1+C_{\PI})^{-1}$ above $0$. By homogeneity, a similar inequality holds   whenever $\X$ has finite diameter, with spectral gap of order ($\diam(\X)^2(1+C_{\PI}))^{-1}$. 
\end{remark}

\begin{proof}
		Let $0\leq j\leq J$ and $k\in [L_j]$. Observe that $s_k^j$ is zero outside $B_k^j := B(x_k^j,8\delta^j)$. Let us list some elementary properties of the balls $B_k^j$:
		\begin{enumerate}[label=(B\arabic*), ref=B\arabic*]
			\item The balls $B_k^j$ for $k\in [L_j]$ cover $\X$.
			\item Any $x\in \X$ belongs to at most $C_\D^6$ distinct balls $B_k^j$ for $k\in [L_j]$. \label{tag:B2}
		\end{enumerate}
		The first item directly follows from the points $\X^j$ forming a maximal $\delta^j$-separated set. For the second one, observe that if $x$ belongs to $\ell$ distinct balls $B_k^j$, then the corresponding points $x_k^j$ are all in $B(x,8\delta^j)$, and at distance at least $\delta^j$ from one another. We then apply \Cref{lem:covering_ball}. 

Let $u\in \L^2(\mu)$. By Jensen's inequality and using the definition of the orthogonal projection, there holds for any choice of coefficients $(a_k)_{k\in [L_j]}$,
\begin{equation}\label{eq:proof_poincaré_beginning}
	\begin{split}
         \|(1-\Pi_{j,\mu})u\|^2_{\L^2(\mu)} &\leq \|u-\sum_{k=1}^{L_j} a_k s^j_k\|^2_{\L^2(\mu)}= \|\sum_{k=1}^{L_j} (u- a_k) s^j_k\|^2_{\L^2(\mu)} \\
     &\leq \sum_{k=1}^{L_j} \int (u(x)- a_k)^2 s^j_k(x) \dd \mu(x) \leq \sum_{k=1}^{L_j} \int_{B_k^j} (u(x)- a_k)^2 \dd \mu(x).
     \end{split}
\end{equation}
We fix an index $k$. Recall that the spline system is defined up to depth $J$, where $2\alpha h\leq \delta^J \leq h/128 $. 
Since the balls $B_m^J$ for $m\in [L_J]$ cover $\X$, they cover in particular 
 $ B_k^j$. Moreover, only balls $B_m^J$ centered at points $x_m^J$ lying in $2B_k^j$ can intersect this set.  Thus,
\begin{align*}
     \int_{B_k^j} (u(x)- a_k)^2 \dd \mu(x) &\leq \sum_{x_m^J\in 2B_k^j} \int_{B_m^J}(u(x)- a_k)^2 \dd \mu(x).
\end{align*}
We let $u_m$ be the $\mu$-average of $u$ on $B_m^J$. There holds 
\begin{equation}\label{eq:poincare_zero}
 \int_{B_k^j} (u(x)- a_k)^2 \dd \mu(x)\leq  \sum_{x_m^J\in 2B_k^j}\p{ \mu(B_m^J) (u_m - a_k)^2+ \int_{B_m^J}(u(x)- u_m)^2 \dd \mu(x)}.
\end{equation}

$\bullet$ We first bound $\sum_{x_m^J\in 2B_k^j} \int_{B_m^J}(u(x)- u_m)^2 \dd \mu(x)$.  We use Jensen's inequality:
\[
\int_{B_m^J}(u(x)- u_m)^2 \dd \mu(x) \leq \frac{1}{\mu(B_m^J)} \iint_{B_m^J\times B_m^J} (u(x)-u(y))^2 \dd \mu(x)\dd \mu(y).
\]
If $(x,y)\in B_m^J\times B_m^J$, then $\rho(x,y)\leq 16\delta^J\leq h/2$. Since $\rho(x,y)\leq h/2$, by \eqref{eq:eta}, $\eta(\rho(x,y)/h)\geq 1/2$. This implies that
\[
\int_{B_m^J}(u(x)- u_m)^2 \dd \mu(x) \leq 2h^2 \iint_{B_m^J\times B_m^J} \frac{\sqrt{\eta_\mu^h(x)\eta_\mu^h(y)}}{\mu(B_m^J)}K_\mu^h(x,y) (u(x)-u(y))^2 \dd \mu(x)\dd \mu(y).
\]
Since $\eta\leq 1$ and is supported on $[0,1]$, $\eta_\mu^h(x) \leq \mu(B(x,h))$. 
If $x\in B_m^J$, then $B(x,h)\subseteq B(x_m^J,h+8\delta^J)$. But $h\leq \delta^J/(2\alpha)$, so $\eta_\mu^h(x)\leq \mu(B(x_m^J,(8+1/(2\alpha))\delta^J))\leq C_\D^{p} \mu(B_m^J)$, where $p=\lceil \log_2(1+1/(16\alpha))\rceil$. Since the same inequality holds for $\eta_\mu^h(y)$, we obtain,
 \[
\int_{B_m^J}(u(x)- u_m)^2 \dd \mu(x) \leq 2C_\D^p h^2 \iint_{B_m^J\times B_m^J}K_\mu^h(x,y) (u(x)-u(y))^2 \dd \mu(x)\dd \mu(y).
\]
We sum  these inequalities over $m\in [L_J]$ with $x_m^J\in 2B_k^j$. Recall  that a given point $x\in \X$ belongs to at most $C_\D^6$ balls $B_m^J$.  Thus,
\begin{align}
	& \sum_{x_m^J\in 2B_k^j} \int_{B_m^J}(u(x)- u_m)^2 \dd \mu(x) \nonumber \\
	&\qquad \leq 	 2C_\D^p h^2 \sum_{x_m^J\in 2B_k^j} \iint_{B_m^J\times B_m^J}K_\mu^h(x,y) (u(x)-u(y))^2 \dd \mu(x)\dd \mu(y) \nonumber\\
	&\qquad \leq 2C_D^{p+6} h^2 \iint_{3B_k^j\times \X} K_\mu^h(x,y) (u(x)-u(y))^2 \dd \mu(x)\dd \mu(y). \label{eq:first_poincare}
\end{align}

$\bullet$ We now bound $ \sum_{x_m^J\in 2B_k^j} \mu(B_m^J) (u_m - a_k)^2$. Introduce  $\mu_J = \sum_{m\in [L_J]} \mu(B_m^J)\delta_{x_m^J}$. According to \Cref{prop:graph_poincare}, the measure $\mu_J$ is in $\PI_{48 \delta^J}(C'_\D, C'_{\PI},2\kappa)$ where $C'_\D = C_\D^{15}$ and $C'_{\PI} = C_1\cdot C_\D^{12} C_{\PI} \alpha^{-2}+ C_\D^6$ 
for some absolute constant  $C_1>0$. 
We interpret the collection of numbers $u_m$ as a real-valued function $\tilde u$ defined over the set $\{x_m^J\}_{m\in[L_J]}$, and let $a_k$ be the  average of $\tilde u$ on $B(x_k^j,48\delta^j)=6B_k^j$ with respect to the discrete measure $\mu_J$. The  Poincaré inequality for $\mu_J$ gives the bound 
\begin{align}
	\sum_{x_m^J\in 2B_k^j}\mu(B_m^J) (u_m - a_k)^2 &\leq\sum_{x_m^J\in 6B_k^j}\mu(B_m^J) (u_m - a_k)^2 \nonumber\\
	&\leq C'_{\PI} (48\delta^j)^2 \int_{B(x_k^j, \kappa' \delta^j)} \Lip_{\mu_J,48\delta^J}[\tilde u]^2 \dd \mu_J,\label{eq:poincare_graph}
	\end{align}
where $\kappa' = 96\kappa$. 
Let $x_m^J\in B(x_k^j, \kappa' \delta^j)$ and let $x_\ell^J$ be at distance less than $48\delta^J$ from $x_m^J$.  
By Jensen's inequality,
\begin{align*}
	|u_m-u_\ell|^2 \leq \frac{1}{\mu(B_m^J)\mu(B_\ell^J)} \iint_{B_m^J\times B_\ell^J} |u(x)-u(y)|^2 \dd \mu(x)\dd \mu(y).
\end{align*}
Points $x$ and $y$ in the integral above are at distance less than $16\delta^J+ 48\delta^J\leq h/2$ from one another. Thus, by \eqref{eq:eta} that $\eta(\rho(x,y)/h)\geq 1/2$ and 
\[
	|u_m-u_\ell|^2\leq 2h^2 \iint_{B_m^J\times B_\ell^J} \frac{\sqrt{\eta_\mu^h(x)\eta_\mu^h(y)}}{\mu(B_m^J)\mu(B_\ell^J)}K_\mu^h(x,y) (u(x)-u(y))^2 \dd \mu(x)\dd \mu(y).
\]
We have already argued that $\eta_\mu^h(x)$ and $\eta_\mu^h(y)$ are smaller respectively than $C_\D^p\mu(B_m^J)$ and $C_\D^p \mu(B_\ell^J)$. The fact that $\rho(x_\ell^J,x_m^J)\leq 48\delta^J$ implies that $B_m^J\subseteq B(x_\ell^J, 64\delta ^J)$. The doubling property then yields that $\mu(B_m^J)$ is smaller than $C_\D^{2a}\mu(B_\ell^J)$ for some absolute constant $a$. Thus,
\[
\sqrt{\eta_\mu^h(x)\eta_\mu^h(y)}\leq C_\D^{p+a} \mu(B_\ell^J)
\]
and, since $h^2\leq \delta^{2J}/(4\alpha^2)$,
\[
	|u_m-u_\ell|^2\leq \frac{C_\D^{p+a} \alpha^{-2}}{2\mu(B_m^J)} \delta^{2J} \iint_{B_m^J\times B_\ell^J} K_\mu^h(x,y) (u(x)-u(y))^2 \dd \mu(x)\dd \mu(y).
	\]
Thus, $\Lip_{\mu_J,48\delta^J}[\tilde u]^2(x_m^J)$ is smaller than
\[
  \frac{C_\D^{p+a}\alpha^{-2}}{2\mu(B_m^J)} \frac{\delta^{2J}}{(48\delta^J)^2} \sup_{x_\ell^J\in B(x_m^J,48\delta^J)}  \iint_{B_m^J\times B_\ell^J} K_\mu^h(x,y) (u(x)-u(y))^2 \dd \mu(x)\dd \mu(y).
\]
We crudely bound the supremum by a sum and use that a point can only belong to $C_\D^6$ balls $B_\ell^J$ to obtain that 
\[
\Lip_{\mu_J,48\delta^J}[\tilde u]^2(x_m^J)\lesssim   \frac{1}{\mu(B_m^J)}  \iint_{B_m^J\times \X} K_\mu^h(x,y) (u(x)-u(y))^2 \dd \mu(x)\dd \mu(y).
\]
Thus, expressing the integral in \eqref{eq:poincare_graph} as a sum, we find that
\begin{align*}
\sum_{x_m^J\in 2B_k^j}\mu(B_m^J) (u_m - a_k)^2 &\lesssim\delta^{2j}   \sum_{x_m^J\in B(x_k^j,\kappa'\delta^j)}   \iint_{B_m^J\times \X} K_\mu^h(x,y) (u(x)-u(y))^2 \dd \mu(x)\dd \mu(y) \\
&\lesssim \delta^{2j}  \iint_{B(x_k^j,(\kappa'+8)\delta^j)\times \X} K_\mu^h(x,y) (u(x)-u(y))^2 \dd \mu(x)\dd \mu(y),
\end{align*}
where we use once again that each point is at most in $C_\D^6$ balls $B_m^J$. 
\medskip

$\bullet$ Putting \eqref{eq:first_poincare} together with this last inequality, we find that term $\int_{B_k^j}(u(x)-a_k)^2\dd \mu(x)$ appearing in \eqref{eq:poincare_zero} can be bounded up to a multiplicative constant by
\[
   (C_{\PI}+1)\delta^{2j}\iint_{B(x_k^j,(\kappa'+8)\delta^j)\times \X} K_\mu^h(x,y) (u(x)-u(y))^2 \dd \mu(x)\dd \mu(y).
\]
Finally, we sum all these equations over $k\in [L_j]$. The fact that $\X^j$ is a maximal $\delta^j$-separated set implies that a given point $x$ can only belong to $C_\D^b$ balls $B(x_k^j,(\kappa'+8)\delta^j)$ for some $b$ large enough (proportional to $\log(1+\kappa)$). Thus, starting again from \eqref{eq:proof_poincaré_beginning}, we find that
\[
   \|(1-\Pi_{j,\mu})u\|^2_{\L^2(\mu)}\lesssim (C_{\PI}+1) \delta^{2j} \iint K_\mu^h(x,y) (u(x)-u(y))^2 \dd \mu(x)\dd \mu(y),
\]
concluding the proof.
\end{proof}

\begin{proof}[Proof of \Cref{prop:multiscale_poincare}]
    Let $u,v\in \L^2(\mu)$ with $\int u\dd \mu=0$ and $\dotp{-\Delta_\mu^h u,u}_\mu+s\|u\|^2_{\L^2(\mu)}\leq 1$.  
We may decompose $u$ as
\[
u = \sum_{j= 0}^{J-1} (\Pi_{j+1,\mu}-\Pi_{j,\mu})u + (1-\Pi_{J,\mu})u.
\]
Indeed, $\Pi_{0,\mu}u=0$ if $\int u\dd\mu=0$ and $\diam(\X)<1$ (\Cref{rem:spectral_gap}). 
The nestedness of the $V_j$s implies that for all $v \in \L^2(\mu)$,
\[
\dotp{v,u}_\mu = \sum_{j=0}^{J-1} \dotp{ (\Pi_{j+1,\mu}-\Pi_{j,\mu})u, (\Pi_{j+1,\mu}-\Pi_{j,\mu})v}_\mu + \dotp{u,(1-\Pi_{J,\mu})v}_\mu.
\]
We apply Cauchy-Schwarz inequality and the triangle inequality to bound
\[
\dotp{v,u}_\mu\leq \sum_{j=0}^{J-1} \|(\Pi_{j+1,\mu}-\Pi_{j,\mu})u\|_{\L^2(\mu)}\|(\Pi_{j+1,\mu}-\Pi_{j,\mu})v\|_{\L^2(\mu)} + \|(1-\Pi_{J,\mu})u\|_{\L^2(\mu)}\|v\|_{\L^2(\mu)}.
\]
We bound $\|(\Pi_{j+1,\mu}-\Pi_{j,\mu})v\|_{\L^2(\mu)}= \|(\Pi_{j+1,\mu}-\Pi_{j,\mu})\Pi_{j+1,\mu}v\|_{\L^2(\mu)} \leq \|\Pi_{j+1,\mu}v\|_{\L^2(\mu)}$ and $\|(\Pi_{j,\mu}-\Pi_{j+1,\mu})u\|_{\L^2(\mu)}\leq \|(1-\Pi_{j,\mu})u\|_{\L^2(\mu)}+\|(1-\Pi_{j+1,\mu})u\|_{\L^2(\mu)}$. This last quantity can be bounded in two different ways: either using that the operators $(1-\Pi_{j,\mu})$ are orthogonal projection, implying that the quantity is smaller than $2\|u\|_{\L^2(\mu)}\leq 2s^{-1/2}$, or by applying \Cref{prop:poincare_spline}, which yields a bound of the form $\sqrt{C(1+C_{\PI})}\delta^{j}$. 
\end{proof}

\section{Concentration inequalities for empirical Laplacian operators}\label{sec:concentration}

Recall the proof strategy outlined in the introduction. To establish the proximity between $\lambda_{\ell,\mu_n}^h=\lambda_\ell(\mu_n,K_{\mu_n}^h)$ and $\lambda_{\ell,\mu}^h=\lambda_\ell(\mu,K_\mu^h)$, the main step consists in bounding $\|\Delta_\mu^h \phi-\Delta_{\mu_n}^h\phi\|_{ \H^{-1}_{h,s}(\mu_n)}$ for $\phi$ an eigenfunction associated with $\lambda_{\ell,\mu}^h$. Since $\mu_n$ is a coarse PI measure with high probability (\Cref{cor:empirical_PI}), this can be done using the multiscale Poincaré inequality (\Cref{prop:multiscale_poincare}). In the remaining step, we therefore must bound the terms $\|\Pi_{j,\mu_n}(\Delta_\mu^h \phi-\Delta_{\mu_n}^h\phi)\|_{\L^2(\mu_n)}^2$  appearing in \Cref{prop:multiscale_poincare}. By decomposing this projection in an orthonormal basis, we arrive at the  fundamental problem of bounding terms of the form $\dotp{\Delta_\mu^h u-\Delta_{\mu_n}^hu,w}_{\mu_n}$ for two test functions $u,w\in \L^2(\mu)$. There are two differences between $\Delta_\mu^h$ and $\Delta_{\mu_n}^h$: the base measure $\mu$ is replaced by $\mu_n$, and the kernel $K_\mu^h$ is replaced by the kernel $K_{\mu_n}^h$. The second difference can be handled rather straightforwardly (see \Cref{sec:kernel}), so let us focus on the first one, by considering the difference $\Delta_{\mu,K_\mu^h}-\Delta_{\mu_n,K_\mu^h}$.

We introduce the quantity
\begin{equation}
 T_{n,\mu}^h[u,w]= \dotp{\Delta_{\mu,K_\mu^h}u-\Delta_{\mu_n,K_{\mu}^h}u,w}_{\mu_n}.
\end{equation}
Let $\psi(x,y) = K_\mu^h(x,y)(u(x)-u(y))(w(x)-w(y))$ and $f(x)=w(x)\Delta_\mu^hu(x)$. 
We decompose $T_{n,\mu}^h[u,w]$ into
\begin{align*}
     T_{n,\mu}^h[u,w] &=  (\dotp{-\Delta_{\mu_n,K_{\mu}^h}u,w}_{\mu_n} - \dotp{-\Delta_\mu^h u,w}_\mu) + (\dotp{-\Delta_\mu^h u,w}_\mu +\dotp{\Delta_\mu^h u,w}_{\mu_n})\\
     &= Q_{n,\mu}^h[u,w] + L_{n,\mu}^h[u,w], 
\end{align*}
where $Q_{n,\mu}^h[u,w]= \frac{1}{n^2}\sum_{i,j=1}^n (\psi(X_i,X_j)-p)$ for $p=\E[\psi(X_1,X_2)]$ and $L_{n,\mu}^h[u,w]=\frac 1n\sum_{i=1}^n (f(X_i)-\E[f(X_i)])$.

First, if $\|w\|_\infty\leq 1$ and is supported on a ball $B$, then
\begin{equation}\label{eq:second_moment_L}
    \E[L_{n,\mu}^h[u,w]^2]\leq \frac{1}{n} \int_B (\Delta_\mu^hu)^2\dd \mu.
\end{equation}
Let us now control the second moment of $Q_{n,\mu}^h[u,w]$, by further assuming that $w$ satisfies
\begin{equation}\label{eq:pseudo_holder}
    \sup_{x,y\in \X,\ \rho(x,y)\leq h}|w(x)-w(y)|\leq L h^\vartheta
\end{equation}
for some $L,\vartheta\geq 0$. We also assume that $\mu\in \mathrm{DM}_{h}(C_\D)$. 

\begin{lemma}\label{lem:bound_kernel}
	Let $h>0$, $C_\D\geq 1$, $x,y\in \X$ and let $\mu\in \mathrm{DM}_h(C_\D)$. There holds	$K_\mu^h(x,y) \leq \frac{2h^{-2}C_\D^{3/2}}{\mu(B(x,h))}$.
\end{lemma}
\begin{proof}
	We may assume that $\rho(x,y)< h$, for otherwise $K_\mu^h(x,y)=0$. 
	Observe  that since $\eta(t)\geq 1/2$ for $t\in [0,1/2]$, there holds $\eta_\mu^h(x)\geq \mu(B(x,h/2))/2\geq C_\D^{-1}\mu(B(x,h))/2$. Furthermore, if $\rho(x,y)<h$, then 
	\[\eta_\mu^h(y)\geq \mu(B(y,h/2))/2\geq C_\D^{-2}\mu(B(y,2h))/2\geq C_\D^{-2}\mu(B(x,h))/2.\]
	 Thus, since $\eta\leq 1$, we find that
	\[
	K_\mu^h(x,y) = \frac{h^{-2}}{\sqrt{\eta_\mu^h(x)\eta_\mu^h(y)}} \eta\p{\frac{\rho(x,y)}h}\leq \frac{2h^{-2}C_\D^{3/2}}{\mu(B(x,h))}. \qedhere
	\]
\end{proof}

Observe that $ Q_{n,\mu}^h$ is a $V$-statistic of order $2$ with kernel $\psi$. Introduce $m(x) = \int \psi(x,y)\dd \mu(y)$. A standard computation gives (using that $\psi(x,x)=0$)
\begin{equation}\label{eq:Vstat_bound}
    \E[Q_{n,\mu}^h[u,w]^2]\leq \frac{4}{n} \E[m(X)^2] + \frac{3}{n^2}\E[\psi(X,Y)^2].
\end{equation}

Recall that $B^h$ is the $h$-offset of the ball $B$. Because $w$ satisfies \eqref{eq:pseudo_holder}, there holds for all $x\in \X$ and $y\in B(x,h)$ that $|w(x)-w(y)|\leq L h^\vartheta \one_{y\in B^h}$. Thus,
\[
|m(x)|\leq  Lh^{\vartheta} \int \one_{y\in B^h}K_\mu^h(x,y) |u(x)-u(y)|\dd \mu(y).
\]
By Cauchy-Schwarz inequality and \Cref{lem:bound_kernel},
\begin{align*}
|m(x)|^2&\leq L^2h^{2\vartheta}  \mu(B(x,h)) \sup_{y\in \X}K_\mu^h(x,y)\int \one_{y\in B^h}K_\mu^h(x,y)  (u(x)-u(y))^2\dd \mu(y) \\
&\leq 2L^2 C_\D^{3/2} h^{2 (\vartheta-1)}\int \one_{y\in B^h}K_\mu^h(x,y)  (u(x)-u(y))^2\dd \mu(y).
\end{align*}
In particular,
\begin{equation}\label{eq:bound_An}
\E[m^2(X)]\leq 2L^2h^{2(\vartheta-1)}C_\D^{3/2} \iint \one_{y\in B^h} K_\mu^h(x,y)  (u(x)-u(y))^2\dd\mu(x)\dd \mu(y).
\end{equation}
 \Cref{lem:bound_kernel} also implies that 
\begin{align*}
    \E[\psi(X,Y)^2]  &\leq L^2 h^{2\vartheta} \iint \one_{y\in B^h}K_\mu^h(x,y)^2(u(x)-u(y))^2\dd\mu(x)\dd\mu(y)\\
   & \leq \frac{2L^2h^{2(\vartheta-1)}C_\D^{3/2}}{v_\mu(h)} \iint \one_{y\in B^h}K_\mu^h(x,y)(u(x)-u(y))^2\dd\mu(x)\dd\mu(y).
\end{align*}
Putting together this estimate with  \eqref{eq:second_moment_L}, \eqref{eq:Vstat_bound} and \eqref{eq:bound_An}, we arrive at the following proposition.

\begin{proposition}\label{prop:bound_empirical_energy}
    Let $h>0$, $C_\D\geq 1$, and let $\mu\in \mathrm{DM}_{h}(C_\D)$be a probability measure. 
    Let $u\in \L^2(\mu)$ and let $w$ be a measurable function satisfying \eqref{eq:pseudo_holder}, supported on a ball $B$, with $\|w\|_\infty\leq 1$. Then, for all integers $n\geq 1$, 
    \begin{align*}
        &\E[T_{n,\mu}^h[u,w]^2]\leq \frac{2}{n} \int_{B} (\Delta_\mu^h u)^2\dd \mu\\
        &\qquad  + 16 L^2 C_\D^{3/2}h^{2(\vartheta-1)} \p{\frac{1 }{n} + \frac{1}{n^2v_\mu(h)}}  \iint  \one_{y\in B^h}K_\mu^h(x,y)(u(x)-u(y))^2 \dd \mu(x)\dd \mu(y).
        \end{align*}
\end{proposition}

The next step is to apply this proposition together with the multiscale Poincaré inequality (\Cref{prop:multiscale_poincare}) to obtain a bound on $\|(\Delta_{\mu,K_\mu}-\Delta_{\mu_n,K_\mu})u\|_{\dot\H^{-1}_h(\mu_n)}$.

\begin{proposition}\label{prop:empirical_dual_norm}
    Assume that $\diam(\X)=1/2$. Let $C_\D,\kappa\geq 1$, $C_{\PI}, h>0$.  There exist a constant $\beta_0\in (0,1)$ depending on $C_\D$, and a constant $C>0$ depending additionally on  $\kappa$  such that the following holds. Let $\mu\in \PI_{\beta_0 h}(C_\D,C_{\PI},\kappa)$ be a probability measure.   Let $n$ be an integer with $nv_\mu(h)\geq 1$.  Let $E_0$ be the event where $\mu_n$ and $\mu$ are $(\beta_0 h,1)$-close. Then,  for all $u\in \L^2(\mu)$ and all $s>0$, 
    \[
    \E[\|(\Delta_{\mu,K_\mu^h}-\Delta_{\mu_n,K_\mu^h})u\|_{\H^{-1}_{h,s}(\mu_n)}^2\one_{E_0}] \leq C\frac{C_{\PI}+1}{nv_\mu(h)} \p{\|u\|_{\dot \H^1_h(\mu)}^2+ s^{-1}\|\Delta_\mu^h u\|^2_{\L^2(\mu)} }.
    \]
\end{proposition}
The choice $\diam(\X)=1/2$ is arbitrary. All that matters is that $\diam(\X)<1$ so that we may apply \Cref{prop:multiscale_poincare}. This can be always ensured up to rescaling. 
The remainder of the section is dedicated to the proof of \Cref{prop:empirical_dual_norm}. Let $\beta_0\in (0,1)$ that we will fix later and let $\mu \in  \PI_{\beta_0 h}(C_\D,C_{\PI},\kappa)$.

Consider a spline system $(s_k^j)_{0\leq j\leq J,k\in [L_j]}$ as in \Cref{prop:splines}, with $2\alpha h\leq \delta^J\leq h/128$, where we choose $\delta$ small enough so that the parameter $\vartheta$ is strictly larger than $1/2$ (any $\delta<(37C_\D^{-4})^2$ is enough) and $\alpha = \delta/256$ is as in \Cref{prop:multiscale_poincare}. 

Assume that $\beta_0\leq \alpha/24$, so $\mu$ is in particular a coarse PI measure at scale $\alpha h/3$. 
Introduce the event $E_0$ where $\mu_n$ and $\mu$ are $(\alpha h/24,1)$-close, implying in particular that $\mu_n$ is in $\PI_{\alpha h}(C_\D',C'_{\PI},2\kappa)$, where the constants $C_\D'$ and $C'_{\PI}$  depend on $C_\D$ and $C_{\PI}$ (see \Cref{cor:empirical_PI}). 
Let $w= (\Delta_{\mu,K_\mu^h}-\Delta_{\mu_n,K_\mu^h})u$ and let $m_j=\min(\delta^j,s^{-1/2})$. When $E_0$ is satisfied, we may apply \Cref{prop:multiscale_poincare}, which yields the inequality
\[
        \|w\|_{ \H^{-1}_{h,s}(\mu_n)}\leq C  \sqrt{C_{\PI}+1} \p{h \|w\|_{\L^2(\mu_n)} + \sum_{j=1}^J m_j \|\Pi_{j,\mu_n} w\|_{\L^2(\mu_n)}}
\]
for some constant $C$ depending on $C_\D$, $\kappa$ and $\delta$. We first bound $\|w\|_{\L^2(\mu)}$.

   \begin{lemma}\label{lemma:empirical_L2_norm} Let $h>0$, $C_\D\geq 1$, $\mu\in \mathrm{DM}_h(C_\D)$  and let $u\in \L^2(\mu)$. Then,
        \[ \E[\|(\Delta_{\mu,K_\mu^h}-\Delta_{\mu_n,K_\mu^h})u\|^2_{\L^2(\mu_n)}]\leq \frac{16h^{-2}C_\D^{3/2}}{nv_\mu(h)} \|u\|^2_{\dot\H^1_h(\mu)}.\]
    \end{lemma}
\begin{proof}
    Introduce the functions $f:(x,y) \mapsto 2K_\mu^h(x,y)(u(y)-u(x))$ and $g: x\mapsto\int f(x,y)\dd\mu(y)$. There holds
    \[
    \|(\Delta_{\mu,K_\mu^h}-\Delta_{\mu_n,K_\mu^h})u\|^2_{\L^2(\mu_n)} = \frac{1}{n}\sum_{i=1}^n \p{g(X_i)-\frac{1}{n}\sum_{j=1}^n f(X_i,X_j)}^2
    \]
     Since $f(x,x)=0$, we obtain that
    \[
   \E[\|(\Delta_{\mu,K_\mu^h}-\Delta_{\mu_n,K_\mu^h})u\|^2_{\L^2(\mu_n)}]= \E\left[  \p{g(X_1)- \frac{1}{n}\sum_{j=2}^n f(X_1,X_j)}^2 \right].
    \]
    We reason conditionally on $X_1$ and develop the product to obtain that
    \begin{align*}
         \E&[\|(\Delta_{\mu,K_\mu^h}-\Delta_{\mu_n,K_\mu^h})u\|^2_{\L^2(\mu_n)}] \leq \frac{2}{n} \E[f(X_1,X_2)^2] \\
         &\leq \frac{8}{n} \iint K_\mu^h(x,y)^2(u(x)-u(y))^2 \dd \mu(x)\dd\mu(y) \leq \frac{16h^{-2}C_\D^{3/2}}{nv_\mu(h)} \|u\|^2_{\dot \H^1_h(\mu)},
    \end{align*}
    where we use \Cref{lem:bound_kernel} at the last line.
\end{proof}

    Fix $j\in [J]$ and let us express $\|\Pi_{j,\mu_n} w\|_{\L^2(\mu_n)}$ more explicitly. Let $G_j$ be the Gram matrix of the functions $(s_k^j)_{k\in [L_j]}$ with respect to $\L^2(\mu_n)$, i.e., $G_{j,k,\ell} = \dotp{s_k^j,s_\ell^j}_{\mu_n}$ for $k,\ell\in [L_j]$. Also, let $\mu_{n,k}^j=\mu_n(B(x_k^j,\delta^j))$ and let $M_j$ be the matrix with entries $M_{j,k,\ell} = G_{j,k,\ell}/\sqrt{\mu_{n,k}^j\mu_{n,\ell}^j}$. 
    According to \cite[Lemma 6.5]{auscher2013orthonormal},\footnote{The authors in \cite{auscher2013orthonormal} state their result for specific spline systems and for doubling measures at scale $r=0$, but their proof only involves the properties of splines stated in \Cref{def:splines} and volumes of balls of radii $\geq c_0\delta^j$, and thus applies verbatim to coarse doubling measures; and thus $\mu_n$ in particular.}  
    if $\mu_n$ is doubling at scale larger than $c_0\delta^j$ for some absolute constant $c_0\in (0,1)$ (which is the case if $\beta_0$ is small enough), the matrix $M_j$ is  invertible and 
    for all $k\in [L_j]$
    \[
     \sum_{\ell\in [L_j]} |M^{-1}_{j,k,\ell}|=\sum_{\ell\in [L_j]} \sqrt{\mu_{n,k}^j\mu_{n,\ell}^j} |G^{-1}_{j,k,\ell} |\leq C
    \]
    for some constant $C$ depending on $\delta$ and $C_\D$. Since $\mu_n$ and $\mu$ are $(\alpha h/24,1)$-close with $\alpha h/24\leq \delta^J/48\leq \delta^j/48$,  \Cref{lem:closeness_implication} implies that
    \[
    \mu_{n,k}^j \geq e^{-1}\mu(B(x_k^j,\delta^j - \alpha h/24))\geq e^{-1}\mu(B(x_k^j,\delta^j/2))\geq e^{-1}C_\D^{-1}v_\mu(\delta^j).
    \]
Thus, one obtains that
    \[
    v_\mu(\delta^j) \sum_{\ell\in [L_j]} |G^{-1}_{j,k,\ell}|\leq C
    \]
    for some other constant $C$. The orthogonal projection $\Pi_{j,\mu_n}$ can be expressed using the Gram matrix through
    \[
    \|\Pi_{j,\mu_n} w\|_{\L^2(\mu_n)}^2 = \sum_{k,\ell\in[L_j]} G^{-1}_{j,k,\ell} \dotp{w,s_k^j}_{\mu_n} \dotp{w,s_\ell^j}_{\mu_n}.
    \] 
We use the inequality $ab\leq \frac 12 (a^2+b^2)$ to obtain,
\[
\|\Pi_{j,\mu_n} w\|_{\L^2(\mu_n)}^2 \leq \sum_{k,\ell\in[L_j]} |G^{-1}_{j,k,\ell}| \dotp{w,s_k^j}_{\mu_n}^2 \leq \frac{C}{v_\mu(\delta^j)}\sum_{k\in [L_j]}\dotp{w,s_k^j}_{\mu_n}^2.
\]
Moreover, $\dotp{w,s_k^j}_{\mu_n}^2$ is bounded in expectation thanks to \Cref{prop:bound_empirical_energy}. Indeed, $s_k^j$ satisfies \eqref{eq:pseudo_holder} with $L$ of order $\delta^{-j\vartheta}$ and is supported on the ball $B_k^j=B(x_k^j,8\delta^j)$.  Since we assume that $nv_\mu(h)\geq 1$, the $1/n$ term dominates the $1/(n^2v_\mu(h))$ term in the bound of \Cref{prop:bound_empirical_energy}. Thus,  there holds
\begin{align*}
    \E[\dotp{w,s_k^j}_{\mu_n}^2] &\lesssim \frac{1}{n}\int_{B_k^j} |\Delta_\mu^h u|^2\dd \mu\\
    &\qquad + \frac{\delta^{-2j\vartheta} h^{2(\vartheta-1)}}{n} \iint \one_{y\in (B_k^j)^h} K_\mu^h(x,y)(u(x)-u(y))^2\dd \mu(x)\dd \mu(y)
\end{align*}
%where we use that the H\"older constant of $s_k^j$ is of order $\delta^{-j\vartheta}$. 
Since $h\leq \delta^j/(2\alpha)$, $(B_k^j)^h\subseteq B(x_k^j , c_1 \delta^j)$ for some constant $c_1>0$. 
As in \eqref{tag:B2}, we may show that $\sum_{k\in [L_j]} \one_{y\in B(x_k^j,c_1\delta^j)}$ is bounded by a constant depending on $C_\D$. Thus, summing these estimates give
\begin{equation}\label{eq:bound_pijmu}
 \E[\|\Pi_{j,\mu_n} w\|_{\L^2(\mu_n)}^2\one_{E_0}] \lesssim  \frac{\delta^{-2j\vartheta} h^{2(\vartheta-1)}}{nv_\mu(\delta^j)} \|u\|^2_{\dot\H^1_h(\mu)} + \frac{1}{nv_\mu(\delta^j)}\|\Delta_\mu^h u\|^2_{\L^2(\mu)}.
\end{equation}
We are now ready to conclude. First, by \Cref{lemma:empirical_L2_norm} , 
\begin{align*}
	\E[\|w\|^2_{\H^{-1}_{h,s}(\mu_n)}\one_{E_0}]\lesssim (1+C_{\PI}) \p{\frac{\|u\|^2_{\dot\H^1_h(\mu)}}{nv_\mu(h)}  +\E\left[\p{\sum_{j=1}^J m_j \|\Pi_{j,\mu_n}w\|_{\L^2(\mu_n)}}^2 \one_{E_0} \right]}.
\end{align*}
We bound the second expectation using  Minkowski's integral inequality in $\L^2(\P)$: %Cauchy-Schwarz inequality. Indeed, writing $\delta^j$ as  $\delta^{-a j}\delta^{(a+1)j}$, we obtain
\begin{align*}
&\E\left[\p{\sum_{j=1}^J m_j\|\Pi_{j,\mu_n}w\|_{\L^2(\mu_n)}}^2 \one_{E_0} \right]  \leq  \p{\sum_{j=1}^J  m_j \p{\E\left[\|\Pi_{j,\mu_n}w\|_{\L^2(\mu_n)}^2 \one_{E_0} \right]}^{1/2}} ^2.
\end{align*}
Since $\sqrt{a+b}\leq \sqrt{a}+\sqrt{b}$, we may use \eqref{eq:bound_pijmu} to obtain that the previous quantity is smaller up to a multiplicative constant than
\begin{align*}
 \p{ \frac{h^{\vartheta-1} \|u\|_{\dot \H^1_h(\mu)} }{\sqrt{n} }\sum_{j=1}^J \frac{m_j \delta^{-j\vartheta} }{\sqrt{v_\mu(\delta^j)}}   +\frac{\|\Delta_\mu^h u\|_{\L^2(\mu)}}{\sqrt{n}} \sum_{j=1}^J \frac{m_j}{\sqrt{v_\mu(\delta^j)}}}^2.
\end{align*}
%&\lesssim \frac{h^{2(\vartheta-1)} \|u\|^2_{\dot \H^1_h(\mu)} \delta^{-2a J}}{n}\sum_{j=1}^J \frac{\delta^{2j(a+1-\vartheta)} }{v_\mu(\delta^j)}  +\frac{\|\Delta_\mu^h u\|^2_{\L^2(\mu)}\delta^{-2a J}}{n} \sum_{j=1}^J \frac{\delta^{2ja}m_j^2}{v_\mu(\delta^j)}.
We have two  sums to bound. Observe first that the computation would be straightforward if $\mu$ were the uniform measure on a $d$-dimensional manifold, so that  $v_\mu(\delta^j)$  is of order $\delta^{jd}$. In the general case, we can replace the existence of a well-defined dimension $d$ by \Cref{lem:superlinear}, which we can use to  bound $1/v_\mu(\delta^j)\lesssim \delta^{J-j}/v_\mu(\delta^J)$. Since  $m_j\leq \delta^j$ and $\delta^J$ is of order $h$, we may use the doubling property to obtain that $v_\mu(\delta^J)\gtrsim v_\mu(h)$, yielding a bound on the first sum of order
\[
\frac{h^{\vartheta-1} h^{1/2}\|u\|_{\dot \H^1_h(\mu)} }{\sqrt{nv_\mu(h)}}\sum_{j=1}^J \delta^{j(1-\vartheta)}\delta^{-j/2}\lesssim \frac{\|u\|_{\dot \H^1_h(\mu)} }{\sqrt{nv_\mu(h)}},
\]
where we use that $\vartheta>1/2$. 
Let us now consider the second sum. We use the bound $m_j\leq s^{-1/2}$ and once again that $1/v_\mu(\delta^j)\lesssim h\delta^{-j}/v_\mu(h)$ to obtain that this sum is bounded by
\[
\frac{\|\Delta_\mu^h u\|_{\L^2(\mu)}h^{1/2}}{\sqrt{snv_\mu(h)}} \sum_{j=1}^J \delta^{-j/2} \lesssim \frac{\|\Delta_\mu^h u\|_{\L^2(\mu)}}{\sqrt{snv_\mu(h)}}.
\]
This concludes the proof.

\section{Fine spectral stability under relative Prokhorov perturbations}\label{sec:proof}

The stability estimates of \Cref{sec:burago} are global in nature: they compare the whole spectra of $\Delta^h_\mu$ and $\Delta^h_\nu$, and their quality is governed by the size of the perturbation itself, through the quantity $\tau + \omega(\frac{2\eps}h)$. %This is unavoidable for a bound that transports an entire $(\ell+1)$-dimensional subspace; it is also what makes these bounds too crude for our purposes.
 When $\nu = \mu_n$ and $\mu$ is the uniform distribution on a compact $d$-dimensional manifold, the parameter $\eps$ has to be taken of the order of the covering radius of the sample, so that $\omega(2\eps/h)$ is at best of order $(\log n/n)^{1/d}/h$, whereas the rate we are after is $(nv_\mu(h))^{-1/2}\asymp (nh^d)^{-1/2} $.
 
 In this section we establish a stability estimate of a different nature, obtained by applying the perturbed operator $\Delta^h_\mu-\Delta^h_\nu$ to a  population eigenfunction $\phi$   and measure the resulting residual $v=(\Delta^h_\mu - \Delta^h_\nu)\phi$ in the dual Sobolev norm $\|v\|_{\dot{\mathrm H}^{-1}_{h,s}(\nu)}$. 
%It is precisely the quantity that the multiscale Poincaré inequality of \Cref{sec:poincare} and the concentration estimates of \Cref{sec:concentration} control at the rate $(nv_\mu(h))^{-1/2}$. \Cref{sec:burago} does not disappear from the picture, however: its crude bounds are what allows us to match the indices of the two spectra, that is, to know that the eigenvalue of $-\Delta^h_\nu$ closest to $\lambda^h_{\ell,\mu}$ is $\lambda^h_{\ell,\nu}$ and not some neighbour.
These estimates  are relative perturbation bounds: they control $|\lambda^h_{\ell,\mu} - \lambda^h_{\ell,\nu}|/\lambda^h_{\ell,\mu}$ rather than the difference itself. Such bounds have a long history in numerical linear algebra, going back to Barlow and Demmel \cite{barlow1990computing}, see the survey by Ipsen \cite{ipsen1998relative}. The natural weight attached to a spectral point $\lambda$ in this theory is $(\lambda-t)^2/\lambda$ rather than $(\lambda-t)^2$, which is exactly the weight that will appear in \Cref{lem:residual} below.

We first isolate the abstract mechanism. Let $A$ be a bounded, nonnegative self-adjoint operator on a Hilbert space $(\cH,\dotp{\cdot,\cdot})$ and assume that $0$ is a simple eigenvalue of $A$, with normalized eigenvector $e$. Let $\|A\|$ be the operator norm of $A$, $\sigma(A)$ be its spectrum, and let  $E_A$ be the spectral measure of $A$, so that $A = \int\lambda\dd E_A(\lambda)$, and write $E^v_A(B) = \dotp{E_A(B)v,v}$ for $v\in\mathcal H$ and $B$ a measurable set. In analogy with the definition of $\|\cdot\|_{\H^{-1}_{h,s}(\nu)}$, we define for $s\geq 0$ the norm
\[
    \|v\|_{A^{-1/2},s}  = \sup\left\{\dotp{v,u}:\ \dotp{u,e} = 0,\ \dotp{Au,u} + s\|u\|^2\leq 1 \right\}\in[0,+\infty].
\]

\begin{lemma}\label{lem:residual}
    Let  $v\in\cH$,  $t\in\R$ and $s> 0$. Then,
   \begin{equation}\label{eq:residual-identity}
        \|Av - tv\|_{A^{-1/2},s}^2
        = \int_{(0,+\infty)}\frac{(\lambda-t)^2}{\lambda+s}\dd E^v_A(\lambda).
    \end{equation}
\end{lemma}

\begin{proof}
    Write $w=Av-tv$ and $P_0=E_A(\{0\})$, the orthogonal projection onto $\ker A=\R e$. Since $s>0$, the operator $(A+s)^{1/2}$ is a bounded bijection of $\{e\}^\perp$, and $\dotp{ Au,u}+s\|u\|^2=\|(A+s)^{1/2}u\|^2$;
substituting $y=(A+s)^{1/2}u$ in the definition of the seminorm therefore gives
\[
  \|w\|_{A^{-1/2},s}=\sup_{y\in\{e\}^\perp, \|y\|\leq 1} \dotp{ (A+s)^{-1/2}(1-P_0)w,y} =\|(A+s)^{-1/2}(1-P_0)w\| .
\]
The square of the right-hand side is $\int_{(0,\infty)}(\lambda+s)^{-1}dE^w_A$.
Finally $w=\psi(A)v$ with $\psi(\lambda)=\lambda-t$, so $dE^w_A=(\lambda-t)^2\,dE^v_A$, which gives the conclusion.
\end{proof}

\begin{corollary}\label{cor:residual}
    Let $v\in\mathcal H$, $t\in\R$, $s>0$, and let $S\subseteq\sigma(A)\backslash\{0\}$ be a  measurable set. Then
    \begin{align}
        \p{\inf_{\lambda\in S}\tfrac{(\lambda-t)^2}{\lambda+s} }E^v_A(S) & \leq \|Av-tv\|^2_{A^{-1/2},s}, \label{eq:res-eigenvalue}\\
        \p{\inf_{\lambda\in S}\frac{(\lambda-t)^2}{\lambda (\lambda+s)}} \dotp{ AE_A(S)v, E_A(S)v} &\leq \|Av-tv\|^2_{A^{-1/2},s}.  \label{eq:res-eigenvector}
    \end{align}
\end{corollary}

\begin{proof}
    Restrict the integral in \eqref{eq:residual-identity} to $S$. For \eqref{eq:res-eigenvalue}, bound the integrand from below by its infimum on $S$. For \eqref{eq:res-eigenvector}, write $\frac{(\lambda-t)^2}{\lambda+s} = \frac{(\lambda-t)^2}{\lambda (\lambda+s)}\lambda$ and
    use $\int_S\lambda\dd E^v_A(\lambda)
    = \langle AE_A(S)v,E_A(S)v\rangle$.
\end{proof}

 The next lemma states that among all nonzero spectral points of $-\Delta^h_\nu$, the one that best approximates $\lambda^h_{\ell,\mu}$ (in the relative sense of \Cref{cor:residual}) is $\lambda^h_{\ell,\nu}$.

\begin{lemma}\label{lem:index-matching}
    
    Let $C_\D,\kappa\geq 1$, $C_\PI,h>0$, $\gamma\in(0,1]$, $\tau\in[0,1]$. There exist absolute constants $a_1,a_2>0$ and $c>0$ depending on $C_\D$ and $\kappa$ such that the following holds. Let $\mu\in\mathrm{PI}_{a_1 h}(C_\D,C_\PI,\kappa)$ be a finite measure and let $\nu$ be $(\eps,\tau)$-close from $\mu$ for some $\eps\leq a_2h$. Assume that for some integer $\ell\geq 1$, $\min_{k\neq\ell}|\lambda^h_{k,\mu}-\lambda^h_{\ell,\mu}|\geq\gamma\lambda^h_{\ell,\mu}$. Write $t=\lambda^h_{\ell,\mu}$ and define $\Theta=\tau+\omega(2\eps/h)+h^2t$. If $\Theta\leq c\gamma$, then $\lambda^h_{\ell,\nu}$ is a simple eigenvalue of $-\Delta^h_\nu$, and
    \begin{equation}\label{eq:localization}
        |\lambda^h_{\ell,\nu}-t|\leq \gamma t/8,
    \end{equation}
    \begin{equation}\label{eq:dichotomy}
        \sigma(-\Delta^h_\nu)\setminus\{0,\lambda^h_{\ell,\nu}\}\subseteq
        (0,(1-3\gamma/4)t]\cup[(1+3\gamma/4)t,+\infty).
    \end{equation}
    In particular $\min_{k\neq\ell}|\lambda^h_{k,\nu}-\lambda^h_{\ell,\nu}|\geq\gamma\lambda^h_{\ell,\nu}/2$, and for every $s\geq 0$ the function $\lambda\in\sigma(-\Delta^h_\nu)\setminus\{0\}\mapsto(\lambda-t)^2/(\lambda+s)$ attains its minimum at $\lambda^h_{\ell,\nu}$.
\end{lemma}

\begin{proof}
    For $s\geq 0$, let $g_s:\lambda\mapsto(\lambda-t)^2/(\lambda+s)$. One can check that  the function $g_s$ is decreasing on $(0,t]$ and increasing on $[t,+\infty)$.

    First, by making the constants $a_1$, $a_2$ and $c$ small enough, we are in position to  apply \Cref{prop:burago} at index $\ell$. This gives \eqref{eq:localization} for small $c$, since $|\lambda^h_{\ell,\nu}-t|\leq C\Theta t\leq Cc\gamma t$. In particular $\lambda^h_{\ell,\nu}\leq 9t/8\leq 9\Theta h^{-2}/8$, so that if $\Theta$ is small enough, $\lambda^h_{\ell,\nu}$ lies strictly below $\lambda_{\infty,\nu}^h$ by \Cref{lem:operator_norm} and \Cref{lem:ess-spectrum}.

    Let us prove \eqref{eq:dichotomy}, and first consider $\lambda\in\sigma(-\Delta^h_\nu)\setminus\{0\}$ with $\lambda<\lambda^h_{\ell,\nu}$. Then $\lambda$ is an eigenvalue below the essential spectrum, say $\lambda=\lambda^h_{k,\nu}$ with $1\leq k<\ell$. The gap condition implies that $\lambda^h_{\ell-1,\mu}\leq(1-\gamma)t$, so \Cref{prop:burago} at index $\ell-1$ yields, for $\Theta$ small enough,
    \begin{equation*}
        \lambda\leq\lambda^h_{\ell-1,\nu}\leq(1+C\Theta)(1-\gamma)t\leq(1-3\gamma/4)t.
    \end{equation*}

    We now claim that $\lambda^h_{\ell+1,\nu}\geq(1+3\gamma/4)t$. Suppose otherwise. Then $\lambda^h_{\ell+1,\nu}<2t\leq2\Theta h^{-2}$, so that, if $\Theta$ is small enough, \Cref{prop:burago} may be applied \emph{with the roles of $\mu$ and $\nu$ exchanged}, at index $\ell+1$ (this is legitimate since $\nu$ is also coarse PI at scale $3a_1h$ by \Cref{lem:stab_condition}, with $3a_1h\leq h/48$ for  $a_1$ small enough).  Writing $\Theta'=\tau+\omega(2\eps/h)+h^2\lambda^h_{\ell+1,\nu}\leq2\Theta$, it gives
    \begin{equation*}
        \lambda^h_{\ell+1,\mu}\leq(1+C\Theta')\lambda^h_{\ell+1,\nu}\leq(1+2C\Theta)(1+3\gamma/4)t<(1+\gamma)t
    \end{equation*}
    if $c$ is small enough, contradicting the spectral gap condition at $\lambda^h_{\ell,\mu}$. Let now $\lambda\in\sigma(-\Delta^h_\nu)$ with $\lambda>\lambda^h_{\ell,\nu}$ and let $C'_\D$ be the doubling constant of $\nu$. If $\lambda<(C'_\D)^{-3/2}h^{-2}$, then $\lambda$ is below the bottom of the essential spectrum, so $\lambda=\lambda^h_{k,\nu}$ for some $k>\ell$, and $\lambda\geq\lambda^h_{\ell+1,\nu}\geq(1+3\gamma/4)t$. Otherwise $\lambda\geq(C'_\D)^{-3/2}h^{-2}$, which is larger than $2\Theta h^{-2}\geq2t\geq(1+3\gamma/4)t$ if $\Theta$ is small enough. This proves \eqref{eq:dichotomy}, and also the simplicity of $\lambda^h_{\ell,\nu}$, since $\lambda^h_{\ell+1,\nu}>\lambda^h_{\ell,\nu}$.

    The relative gap for $\nu$ follows: by \eqref{eq:localization} and \eqref{eq:dichotomy}, any $\lambda\in\sigma(-\Delta^h_\nu)\setminus\{0,\lambda^h_{\ell,\nu}\}$ satisfies $|\lambda-\lambda^h_{\ell,\nu}|\geq(3\gamma/4-\gamma/8)t\geq\frac59\gamma\lambda^h_{\ell,\nu}$, where we use $\lambda^h_{\ell,\nu}\leq 9t/8$.

    It remains to compare the values of $g_s$. On the one hand, \eqref{eq:localization} gives $g_s(\lambda^h_{\ell,\nu})\leq(\gamma t/8)^2/((1-\gamma/8)t+s)$. On the other hand, by \eqref{eq:dichotomy} and the monotonicity of $g_s$, any other nonzero spectral point $\lambda$ satisfies
    \begin{equation*}
        g_s(\lambda)\geq\min\big(g_s((1-3\gamma/4)t),g_s((1+3\gamma/4)t)\big)=\frac{(3\gamma t/4)^2}{(1+3\gamma/4)t+s}.
    \end{equation*}
    Since $(a+s)/(b+s)\leq a/b$ for $a\geq b>0$ and $s\geq0$, and since $\gamma\leq1$,
    \begin{equation*}
        \frac{g_s(\lambda^h_{\ell,\nu})}{g_s(\lambda)}\leq\frac1{36}\cdot\frac{(1+3\gamma/4)t+s}{(1-\gamma/8)t+s}\leq\frac1{36}\cdot\frac{1+3\gamma/4}{1-\gamma/8}\leq\frac1{18},
    \end{equation*}
    which concludes.
\end{proof}

To exploit \eqref{eq:res-eigenvalue} we need a lower bound on
$E^\varphi_{-\Delta^h_\nu}(\sigma\setminus\{0\})
= \|\varphi\|^2_{\mathrm L^2(\nu)}-\langle\varphi,e\rangle^2_\nu$, where
$e = \mathbf 1/\sqrt{\nu(\X)}$. Both terms are integrals against $\nu$ of a
function built from $\mu$, and must therefore be transferred along a transport plan $\pi$ defining the $(\eps,\tau)$-closeness.

\begin{lemma}[Modulus of continuity of the kernel]\label{lem:kernel-mod}
  Let $C_\D\geq 1$, $h>0$ and let $\mu\in\mathrm{DM}_{h/2}(C_{\mathrm D})$ be a probability measure.
Let $x,y\in\X$ with $\rho(x,y)\leq\eps\leq h/4$. Then, for
    all $z\in\X$,
    \begin{equation}\label{eq:kernel-mod}
        \left|K^h_\mu(x,z)-K^h_\mu(y,z)\right|
        \leq 8C_\D^{6}\omega(\eps/h)
        \frac{h^{-2}}{\mu(B(z,h))}\,\one_{\rho(x,z)< 2h}.
    \end{equation}
\end{lemma}

\begin{proof}
    Write $\eta_x=\eta_\mu^h(x)$ and $V = \mu(B(x,h))$. Since
    $\eta\leq1$ is supported on $[0,1]$ and $\eta\geq1/2$ on
    $[0,1/2]$, and since $B(y,h/2)\supseteq B(x,h/4)$ and
    $B(y,h)\subseteq B(x,2h)$, coarse doubling at scale $h/2$ gives
    \begin{equation}\label{eq:eta-two-sided}
        \tfrac12C_{\mathrm D}^{-2}V\leq\eta_x,\eta_y\leq
        C_{\mathrm D}V .
    \end{equation}
    Moreover $|\rho(x,w)-\rho(y,w)|\leq\eps$ for all $w$, and the
    integrand below vanishes unless $\rho(x,w)\leq h+\eps\leq
    2h$, so that
    \begin{align*}
        |\eta_x-\eta_y| &\leq \int\left|
        \eta\p{\frac{\rho(x,w)}h}- \eta\p{\frac{\rho(y,w)}h}
        \right|\dd\mu(w)\\
       & \leq \omega(\eps/h)\mu(B(x,2h))
         \leq C_{\mathrm D}\omega(\eps/h)V.
    \end{align*}
    Thus, by \eqref{eq:eta-two-sided},
    \[\left|\eta_x^{-1/2}-\eta_y^{-1/2}\right|
    = |\eta_y-\eta_x|\p{\sqrt{\eta_x\eta_y}
    (\sqrt{\eta_x}+\sqrt{\eta_y})}^{-1}
    \leq 2C_\D^{4} \omega(\eps/h)V^{-1/2}.\] Splitting
    \begin{equation*}
        \frac{\eta(\rho(x,z)/h)}{\sqrt{\eta_x}}
        -\frac{\eta(\rho(y,z)/h)}{\sqrt{\eta_y}}
        = \frac{\eta(\rho(x,z)/h)-\eta(\rho(y,z)/h)}{\sqrt{\eta_x}}
        +\eta\bigl(\tfrac{\rho(y,z)}h\bigr)
        \Bigl(\frac1{\sqrt{\eta_x}}-\frac1{\sqrt{\eta_y}}\Bigr),
    \end{equation*}
    both summands are bounded by $2C_\D^{4}\omega(\eps/h)V^{-1/2}$ and vanish
    unless $\rho(x,z)<2h$. Dividing by $\sqrt{\eta_z}\geq
    \p{\frac12C_{\mathrm D}^{-1}\mu(B(z,h))}^{1/2}$ and using
    $V\geq C_{\mathrm D}^{-2}\mu(B(z,h))$, valid whenever
    $\rho(x,z)\leq2h$ since then $B(z,h)\subseteq B(x,4h)$, yields
    \eqref{eq:kernel-mod}.
\end{proof}

\begin{lemma}\label{lem:averaging}   Let $C_\D\geq 1$, $h>0$ and let $\mu\in\mathrm{DM}_{2h}(C_{\mathrm D})$.
    For $f\in\mathrm L^1(\mu)$ and $x\in \X$, let
    \begin{equation*}
        \mathcal Mf(x) = \int
        \frac{\mathbf 1_{\rho(x,z)< 2h}}{\mu(B(z,h))}\,|f(z)|\dd\mu(z).
    \end{equation*}
    Then $\|\mathcal Mf\|_{\mathrm L^1(\mu)}\leq
    C_{\mathrm D}\|f\|_{\mathrm L^1(\mu)}$ and
    $\|\mathcal Mf\|_{\mathrm L^2(\mu)}\leq
    C_{\mathrm D}^2\|f\|_{\mathrm L^2(\mu)}$. Moreover, there holds
    $\int_{B(x,2h)}\mu(B(z,h))^{-1}\dd\mu(z)\leq C_{\mathrm D}^3$ for all
    $x\in\X$.
\end{lemma}

\begin{proof}
    Let $k(x,z) = \mathbf 1_{\rho(x,z) < 2h}/\mu(B(z,h))$. On the one
    hand \[\int k(x,z)\dd\mu(x) = \mu(B(z,2h))/\mu(B(z,h))\leq
    C_{\mathrm D},\] which gives the $\mathrm L^1$ bound by Fubini's theorem. On the other
    hand, if $\rho(x,z) < 2h$ then $B(x,h)\subseteq B(z,4h)$, so
    $\mu(B(x,h))\leq C_{\mathrm D}^2\mu(B(z,h))$ and therefore
    \[\int k(x,z)\dd\mu(z)=\int_{B(x,2h)}\mu(B(z,h))^{-1}\dd\mu(z)\leq
    C_{\mathrm D}^2\mu(B(x,2h))/\mu(B(x,h))\leq C_{\mathrm D}^3.\] Schur's
    test gives $\|\mathcal M\|_{\mathrm L^2\to\mathrm L^2}\leq
    (C_{\mathrm D}^3C_{\mathrm D})^{1/2} = C_{\mathrm D}^2$, which concludes the proof.
\end{proof}

We fix $\mu\in \mathrm{DM}_{h/4}(C_\D)$ and let $\varphi\in\mathrm L^2(\mu)$ satisfy $-\Delta^h_\mu\varphi
= \lambda\varphi$ with $\|\varphi\|_{\mathrm L^2(\mu)} = 1$ and $\lambda\leq C_\D^{-3/2}h^{-2}/2$. Writing $T$ for
the integral operator with kernel $K^h_\mu$ and $D = 2m^h_\mu-\lambda$, the
eigenvalue equation reads $D\varphi = 2T\varphi$, and we define the
pointwise representative
\begin{equation}\label{eq:representative}
    \tilde\varphi(x) := \frac{2T\varphi(x)}{2m^h_\mu(x)-\lambda},
\end{equation}
which agrees $\mu$-almost everywhere with $\varphi$ (and is well-defined at every
point $x$, since  $m^h_\mu(x)>\lambda/2$, see \eqref{eq:lower_bound_m}). All statements below refer to this representative, that is we identify $\phi$ and $\tilde\phi$. Making such a choice is necessary: indeed, $\phi$ is \emph{a priori} only defined as an element of $\L^2(\mu)$, and not in a pointwise manner. We must choose a representative to consider $\phi$ as an element of $\L^2(\nu)$ for a discrete measure $\nu=\mu_n$. 
 
\begin{lemma}\label{lem:increment}
   Let $C_\D\geq 1$, $h>0$, and let $\mu\in\mathrm{DM}_{h/4}(C_{\mathrm D})$ be a probability measure. Let $\phi$ satisfy $-\Delta^h_\mu\varphi
= \lambda\varphi$ as in \eqref{eq:representative} and assume that $\lambda\leq C_\D^{-3/2}h^{-2}/2$. Let $\Phi = \mathcal M\varphi+|\varphi|$, so that
    $\|\Phi\|_{\mathrm L^1(\mu)}\leq 2C_\D$ and
    $\|\Phi\|_{\mathrm L^2(\mu)}\leq 2C_\D^2$. Then, for all $x\in\X$ and all
    $y$ with $\rho(x,y)\leq\eps\leq h/4$,
    \begin{equation}\label{eq:increment}
        |\varphi(y)-\varphi(x)|\leq 32C_\D^{11}\omega(\eps/h) \Phi(x).
    \end{equation}
\end{lemma}

\begin{proof}
    Observe that
    \begin{equation}\label{eq:increment-identity}
        \varphi(y)-\varphi(x)
        = \frac{2}{D(y)}\int(K^h_\mu(y,z)-K^h_\mu(x,z))
        (\varphi(z)-\varphi(x))\dd\mu(z).
    \end{equation}
    Indeed, the right-hand side equals
    $\frac2{D(y)}(T\varphi(y)-T\varphi(x))
    -\frac{2\varphi(x)}{D(y)}(m^h_\mu(y)-m^h_\mu(x))$. Since
    $2(m^h_\mu(y)-m^h_\mu(x)) = D(y)-D(x)$ and $2T\varphi(x) = D(x)\varphi(x)$,
    it simplifies to $\frac{2T\varphi(y)}{D(y)}-\varphi(x)
    = \varphi(y)-\varphi(x)$. 

    According to \eqref{eq:lower_bound_m}, we have $D(y)\geq C_\D^{-3/2}h^{-2}/2$. Applying
\Cref{lem:kernel-mod} to \eqref{eq:increment-identity} and bounding
    $|\varphi(z)-\varphi(x)|\leq|\varphi(z)|+|\varphi(x)|$, we obtain for $C=32C_\D^8$
    \begin{equation*}
        |\varphi(y)-\varphi(x)|
        \leq C \omega(\eps/h)
        \int_{B(x,2h)}\frac{|\varphi(z)|+|\varphi(x)|}{\mu(B(z,h))}\dd\mu(z)
        \leq C\omega(\eps/h)
        (\cM\varphi(x)+C_{\mathrm D}^3|\varphi(x)|),
    \end{equation*}
    the last step by the final assertion of Lemma~\ref{lem:averaging}. The
    bounds on $\Phi$ follow from Lemma~\ref{lem:averaging} together with
    $\|\varphi\|_{\mathrm L^1(\mu)}\leq\|\varphi\|_{\mathrm L^2(\mu)}=1$.
\end{proof}

\begin{lemma}\label{prop:mass}
    Consider the same setting as in \Cref{lem:increment},  and let $\nu$ be $(\eps,\tau)$-close to $\mu$ with
     $\tau\leq1$. Assume that $\int\varphi\dd\mu = 0$. Write
    $e = \mathbf 1/\sqrt{\nu(\X)}$. Then, there are absolute constants $a,b>0$ such that
    \begin{equation*}
        \left|\dotp{\phi,e}_\nu \right|
        \leq aC_\D^b (\omega(\eps/h)+\tau),
        \qquad
        \left| \|\varphi\|^2_{\mathrm L^2(\nu)}-1\right|
        \leq aC_\D^b (\omega(\eps/h)+\tau).
    \end{equation*}
    In particular there is $c_0>0$, depending only on $C_{\mathrm D}$, such
    that $\omega(\eps/h)+\tau\leq c_0$ implies
    \begin{equation}\label{eq:mass}
        \|\varphi\|^2_{\mathrm L^2(\nu)}-\langle\varphi,e\rangle^2_\nu
        \geq \frac{15}{16}.
    \end{equation}
\end{lemma}

\begin{proof}
    Let $\tilde \mu$ and $\tilde\nu$ with $e^{-\tau}\mu\leq \tilde\mu\leq \mu$ and $e^{-\tau}\nu\leq \tilde\nu\leq \nu$. Let $\pi$ be a transport plan between $\tilde\mu$ and $\tilde \nu$ with $\cC(\pi)\leq \eps$. Observe that for any bounded measurable function $f:\X\to \R$, since the Radon-Nikodym derivatives take values in $[1,e^\tau]$, 
    \begin{align}
          \left|\int f\dd\nu-\int f\dd\mu\right|
       &= \left| \int \p{f(y)\tfrac{\dd\nu}{\dd\tilde\nu}(y)
        -f(x)\tfrac{\dd\mu}{\dd\tilde\mu}(x)}\dd\pi(x,y)\right| \nonumber \\
        & \hspace{-1cm}\leq \int |f(y)-f(x)|\frac{\dd\nu}{\dd\tilde\nu}(y)\dd\pi(x,y)
        +\int |f(x)|\left|\frac{\dd\nu}{\dd\tilde\nu}(y)
        -\frac{\dd\mu}{\dd\tilde\mu}(x)\right|\dd\pi(x,y) \nonumber\\
        & \hspace{-1cm}\leq e^\tau \int |f(y)-f(x)|\dd\pi(x,y)+ (e^\tau-1)\int |f|\dd \mu. \label{eq:transfer}
    \end{align}
    Apply \eqref{eq:transfer} to $f = \varphi$. By \eqref{eq:increment}
    and Lemma~\ref{lem:increment},
    \[\int|\varphi(y)-\varphi(x)|\dd\pi(x,y)\leq 32C_\D^{11}\omega(\eps/h)
    \int\Phi\dd\tilde\mu\leq 64C_\D^{12}\omega(\eps/h),\] 
     while
    $\int|\varphi|\dd\mu\leq1$ and $e^\tau-1\leq2\tau$ for
    $\tau\leq1$. Since $\int\varphi\dd\mu = 0$ and $\nu(\X)\geq e^{-\tau}\mu(\X)=e^{-\tau}$ by \Cref{lem:closeness_implication}, the first bound follows. The second bound is obtained in a similar way, while the third one is a direct consequence of the two first bounds.
\end{proof}

Finally, we arrive at our main deterministic bound.

\begin{proposition}\label{prop:deterministic}
    Let $C_\D,\kappa\geq1$, $C_\PI,h>0$, $\gamma\in(0,1]$ and $s>0$. There exists  absolute constants $a_1,a_2>0$ and $c\in(0,1]$ depending on $C_\D$ and  $\kappa$  such that the following holds. Let $\mu\in\mathrm{PI}_{a_1h}(C_\D,C_\PI,\kappa)$ be a probability measure. Let $\ell\geq1$ satisfy the gap condition $\min_{k\neq\ell}|\lambda^h_{k,\mu}-\lambda^h_{\ell,\mu}|\geq\gamma\lambda^h_{\ell,\mu}$ and write $t=\lambda^h_{\ell,\mu}$. Let $\nu$ be $(\eps,\tau)$-close to $\mu$ for some $\eps\leq a_2h$, with $\Theta=\tau+\omega(2\eps/h)+h^2t\leq c\gamma$. 
    Let $\varphi$ be an eigenfunction of $-\Delta^h_\mu$ associated with $t$, normalized by $\|\varphi\|_{\mathrm L^2(\mu)}=1$ and represented as in \eqref{eq:representative}. Then $\lambda^h_{\ell,\nu}$ is a simple eigenvalue of $-\Delta^h_\nu$ and, writing $\pi^h_{\ell,\nu}$ for the orthogonal projection of $\mathrm L^2(\nu)$ onto the $\ell$th eigenspace of $-\Delta^h_\nu$ and $R=\|(\Delta^h_\mu-\Delta^h_\nu)\varphi\|_{\mathrm H^{-1}_{h,s}(\nu)}$,
    \begin{align}
        (t-\lambda^h_{\ell,\nu})^2
        &\leq 2(t+s)R^2,
        \label{eq:eigenvalue-bound}\\
        \|(1-\pi^h_{\ell,\nu})\varphi\|^2_{\dot{\H}^1_h(\nu)}
        +s\|(1-\pi^h_{\ell,\nu})\varphi\|^2_{\L^2(\nu)}
        &\leq\frac{11}{\gamma^2}\p{1+\frac st}^2R^2
        +\frac s{\nu(\X)}\Bigl(\int\varphi\dd\nu\Bigr)^2 .
        \label{eq:eigenvector-bound}
    \end{align}
\end{proposition}

\begin{proof}
    First, if $\eps\leq a_2h$ for $a_2$ small enough, then $\nu$ is a coarse PI measure at scale $3a_1h$ by \Cref{lem:stab_condition}. 
    Write $A=-\Delta^h_\nu$ and $e=\one/\sqrt{\nu(\X)}$. The operator $A$ is bounded on $\L^2(\nu)$, self-adjoint and nonnegative, and $\ker A=\R e$: if $\dotp{ Au,u}_\nu=0$ then, $K^h_\nu$ being positive on $\{\rho\leq h/2\}$, we have $\Lip_{\nu,h/2}[u]=0$ $\nu$-almost everywhere, and the coarse Poincaré inequality for $\nu$ applied to a ball containing $\X$ forces $u$ to be $\nu$-almost everywhere constant. Thus $A$ satisfies the hypotheses of \Cref{lem:residual}, with $\|\cdot\|_{A^{-1/2},s}=\|\cdot\|_{\H^{-1}_{h,s}(\nu)}$ and $\dotp{ A\cdot,\cdot}_\nu=\|\cdot\|^2_{\dot{\H}^1_h(\nu)}$. The same argument applied to $\mu$ implies that $\lambda_{\ell,\mu}^h>0$ and that $\int\varphi\dd\mu=0$ since $\varphi$ is orthogonal to constant functions. Finally, the representative \eqref{eq:representative} satisfies $-\Delta^h_\mu\varphi(x)=t\varphi(x)$ for all $x$, so that
    \begin{equation}\label{eq:residual-is-w}
        A\varphi-t\varphi=(\Delta^h_\mu-\Delta^h_\nu)\varphi
        \qquad\text{in }\L^2(\nu).
    \end{equation}

    Let us first prove \eqref{eq:eigenvalue-bound}. Apply \eqref{eq:res-eigenvalue} with $S=\sigma(A)\setminus\{0\}$, for which $E^\varphi_A(S)=\|\varphi\|^2_{\L^2(\nu)}-\dotp{\varphi,e}^2_\nu$, since $E_A(\{0\})$ is the rank-one projection onto $\mathbb Re$. If $c$, $a_1$ and $a_2$ are small enough, \Cref{prop:mass} bounds this quantity from below by $15/16$, while \Cref{lem:index-matching} identifies $\lambda^h_{\ell,\nu}$ as the minimizer of $\lambda\mapsto(\lambda-t)^2/(\lambda+s)$ over $S$. Together with \eqref{eq:residual-is-w}, this gives
    \begin{equation*}
        \frac{15}{16}\frac{(\lambda^h_{\ell,\nu}-t)^2}{\lambda^h_{\ell,\nu}+s}\leq R^2 .
    \end{equation*}
    By \eqref{eq:localization} and $\gamma\leq1$ we have $\lambda^h_{\ell,\nu}+s\leq\frac98(t+s)$. Thus, $(\lambda^h_{\ell,\nu}-t)^2\leq 2(t+s)R^2$.

    Let us now prove \eqref{eq:eigenvector-bound}. Let $S=\sigma(A)\setminus\{0,\lambda^h_{\ell,\nu}\}$ and $P_0=E_A(\{0\})$. The three projections in the decomposition $1=P_0+\pi^h_{\ell,\nu}+E_A(S)$ being mutually orthogonal, and $AP_0=0$,
    \begin{equation*}
        \|(1-\pi^h_{\ell,\nu})\varphi\|^2_{\dot{\H}^1_h(\nu)}
        =\dotp{ AE_A(S)\varphi,E_A(S)\varphi}_\nu,
        \qquad
        \|(1-\pi^h_{\ell,\nu})\varphi\|^2_{\L^2(\nu)}
        =E^\varphi_A(S)+\dotp{\varphi,e}^2_\nu .
    \end{equation*}
    Let $\lambda\in S$ and let $\zeta=\frac{3\gamma}7\frac t{t+s}$. By \eqref{eq:dichotomy}, either $\lambda\leq(1-3\gamma/4)t$, in which case $t-\lambda\geq\frac{3\gamma}4t$ and $\lambda+s\leq t+s$, or $\lambda\geq(1+3\gamma/4)t$, in which case, $\lambda\mapsto(\lambda-t)/(\lambda+s)$ being increasing,
    \begin{equation*}
        \frac{\lambda-t}{\lambda+s}
        \geq\frac{(3\gamma/4)t}{(1+3\gamma/4)t+s}
        \geq\frac47\frac{(3\gamma/4)t}{t+s} .
    \end{equation*}
    In both cases $|\lambda-t|\geq\zeta(\lambda+s)$, so that
    \begin{equation*}
        \frac{(\lambda-t)^2}{\lambda(\lambda+s)}\geq\zeta^2\frac{\lambda+s}\lambda\geq\zeta^2,
        \qquad
        \frac{(\lambda-t)^2}{\lambda+s}\geq\zeta^2(\lambda+s)\geq\zeta^2s .
    \end{equation*}
    Hence \eqref{eq:res-eigenvector} and \eqref{eq:res-eigenvalue} give respectively $\dotp{ AE_A(S)\varphi,E_A(S)\varphi}_\nu\leq\zeta^{-2}R^2$ and $sE^\varphi_A(S)\leq\zeta^{-2}R^2$. We conclude by $2\zeta^{-2}=\frac{98}{9\gamma^2}(1+s/t)^2$ and $\dotp{\varphi,e}^2_\nu=(\int\varphi\dd\nu)^2/\nu(\X)$.
\end{proof}

\begin{remark}\label{rem:reusable}
    \Cref{prop:deterministic} makes no reference to sampling: it
    applies to any pair $(\mu,\nu)$ of $(\eps,\tau)$-close measures. We will apply the proposition to $\nu=\mu_n$ equal to the empirical measure, but any family of random measures (e.g., coming from non independent samples) for which relative Prokhorov closeness can be established will satisfy the same spectral guarantees.
\end{remark}

We are now ready to prove \Cref{thm}. First observe that since there exists a coarse PI measure on $\X$ with finite mass, the space $\X$ has finite diameter (see the remark after \Cref{lem:reverse}). But then, observe that the inequality stated in the theorem is invariant by rescaling, so we might assume that $\diam(\X)=1/2$: this will allow us to apply \Cref{prop:empirical_dual_norm}. We assume that $\mu\in\PI_{\beta h}(C_\D,C_\PI,\kappa)$ with $\beta$ small enough so that \Cref{prop:empirical_dual_norm} and \Cref{prop:deterministic} apply. We write $t=\lambda^h_{\ell,\mu}$. We consider the event $E_1$ where $\mu_n$ and $\mu$ are $(\eps,\tau)$-close, where $\eps$ and $\tau$ are small enough so that  \Cref{prop:deterministic} applies and $\eps\leq \beta_0h$. In particular, $E_1\subseteq E_0$ (where $E_0$ is defined in \Cref{prop:empirical_dual_norm}). We apply \Cref{prop:deterministic} with the mass parameter $s=t$: on $E_1$, \eqref{eq:eigenvalue-bound} yields
\[
    (\lambda^h_{\ell,\mu_n}-t)^2\leq 4t\|(\Delta^h_{\mu_n}-\Delta^h_\mu)\varphi\|^2_{\H^{-1}_{h,t}(\mu_n)}.
\]
However, the norm bounded in \Cref{prop:empirical_dual_norm} is not $\|(\Delta^h_{\mu_n}-\Delta^h_\mu)\varphi\|_{\H^{-1}_{h,t}(\mu_n)}$ but $\|(\Delta_{\mu_n,K^h_\mu}-\Delta^h_\mu)\varphi\|_{\H^{-1}_{h,t}(\mu_n)}$. We bound the residual error in the next lemma, proven in \Cref{sec:kernel}. Since $\|\cdot\|_{\H^{-1}_{h,s}(\mu_n)}\leq\|\cdot\|_{\dot\H^{-1}_{h,0}(\mu_n)}$ for every $s\geq 0$, the lemma applies verbatim to the norm with $s>0$.

\begin{lemma}\label{lem:kernel_variation}
	Let $C_\D\geq 1$, $h>0$. Let $\mu\in\mathrm{DM}_h(C_\D)$ and let $n$ be an integer with $nv_\mu(h)\geq 16C_\D$. Then, there exists an absolute constant $C$ and an event $E_2$ of probability at least $1-2ne^{-nv_\mu(h)/(CC_\D^2)}$ such that for all $u\in\L^2(\mu)$,
	\[
	\E[\|(\Delta^h_{\mu_n}-\Delta_{\mu_n,K^h_\mu})u\|^2_{\dot\H^{-1}_{h,0}(\mu_n)}\one_{E_2}]\leq CC_\D\frac{\|u\|^2_{\dot\H^1_h(\mu)}}{nv_\mu(h)}.
	\]
\end{lemma}

\begin{lemma}\label{lem:crude_operator_norm}
	There holds $\lambda^h_{\ell,\mu_n}\leq 6nh^{-2}$.
\end{lemma}

\begin{proof}
	We bound the operator norm of $\Delta^h_{\mu_n}$. Let $u\in\L^2(\mu_n)$. There holds by Jensen's inequality
	\begin{align*}
		\|\Delta^h_{\mu_n}u\|^2_{\L^2(\mu_n)}&=\frac1n\sum_{i=1}^n\p{\frac2n\sum_{j=1}^nK^h_{\mu_n}(X_i,X_j)(u(X_i)-u(X_j))}^2\\
		&\leq\max_{i,j}K^h_{\mu_n}(X_i,X_j)^2\frac4{n^2}\sum_{i,j=1}^n(u(X_i)-u(X_j))^2\\
		&\leq 8\max_{i,j}K^h_{\mu_n}(X_i,X_j)^2\|u\|^2_{\L^2(\mu_n)}.
	\end{align*}
	However, since $\eta(0)\geq 1/2$, $\eta^h_{\mu_n}(X_i),\eta^h_{\mu_n}(X_j)\geq 1/(2n)$, so $K^h_{\mu_n}(X_i,X_j)\leq 2nh^{-2}$. This concludes the proof, as $\sqrt{32}\leq 6$.
\end{proof}

Let $E$ be the intersection of $E_1$ and $E_2$. Using \Cref{prop:empirical_closeness}, we can apply the doubling inequality sufficiently many times, together with the condition $nv_\mu(h)\gtrsim\log(n)$ to obtain that $\P(E_1)\geq 1-Ce^{-nv_\mu(h)/C}$ for some constant $C$. Together with the probability estimate in \Cref{lem:kernel_variation}, we obtain a bound of the form $\P(E^c)\leq Ce^{-nv_\mu(h)/C}$. Therefore, by \Cref{lem:crude_operator_norm},
\begin{align*}
	\E[(t-\lambda^h_{\ell,\mu_n})^2]&\leq\E[(t-\lambda^h_{\ell,\mu_n})^2\one_E]+\E[(t-\lambda^h_{\ell,\mu_n})^2\one_{E^c}]\\
	&\leq 4t\,\E[\|\Delta^h_\mu\varphi-\Delta^h_{\mu_n}\varphi\|^2_{\H^{-1}_{h,t}(\mu_n)}\one_E]+(2t^2+72n^2h^{-4})Ce^{-nv_\mu(h)/C}.
\end{align*}
Finally, we bound $\|\Delta^h_\mu\varphi-\Delta^h_{\mu_n}\varphi\|_{\H^{-1}_{h,t}(\mu_n)}\leq\|\Delta^h_\mu\varphi-\Delta_{\mu_n,K^h_\mu}\varphi\|_{\H^{-1}_{h,t}(\mu_n)}+\|\Delta_{\mu_n,K^h_\mu}\varphi-\Delta^h_{\mu_n}\varphi\|_{\dot\H^{-1}_{h,0}(\mu_n)}$. The first term is bounded thanks to \Cref{prop:empirical_dual_norm} applied with $s=t$, and the second one thanks to \Cref{lem:kernel_variation}. In total, since $\|\varphi\|^2_{\dot\H^1_h(\mu)}=t$ and $\|\Delta^h_\mu\varphi\|^2_{\L^2(\mu)}=t^2$, so that $\|\varphi\|^2_{\dot\H^1_h(\mu)}+t^{-1}\|\Delta^h_\mu\varphi\|^2_{\L^2(\mu)}=2t$, we find that
\[
\E[\|\Delta^h_\mu\varphi-\Delta^h_{\mu_n}\varphi\|^2_{\H^{-1}_{h,t}(\mu_n)}\one_E]\leq Ct\frac{C_\PI+1}{nv_\mu(h)}
\]
for some constant $C$ depending on $C_\D$ and $\kappa$. This is the announced bound, up to the remainder term.

It remains to show that the remainder term $(2t^2+72n^2h^{-4})e^{-nv_\mu(h)/C}$ is negligible in front of $t^2(1+C_\PI)/(nv_\mu(h))$. First, $t\lesssim h^{-2}$, so the prefactor is of order $n^2h^{-4}$. Recall that we assume that $nv_\mu(h)\geq C_1\log n$, for some large constant $C_1$. Recall that $t\geq\lambda^h_{1,\mu}\geq c(1+C_\PI)^{-1}$ by \Cref{rem:spectral_gap}, so that it suffices to bound $n^2h^{-4}e^{-nv_\mu(h)/C}$ by $(1+C_\PI)^{-1}(nv_\mu(h))^{-1}$. This is less direct to handle, since $h$ and $n$ are only related through the inequality $nv_\mu(h)\geq C_1\log n$. However, the reverse doubling property (\Cref{lem:reverse}) implies that
\[
\mu(B(x,h))\leq c_\D^k\mu(B(x,2^kh))\leq c_\D^k
\]
whenever $2^kh< 1/4$. Letting $k$ be of order $\log(1/h)$, we find an inequality of the form $\mu(B(x,h))\leq ah^b$ for some constants $a,b>0$. Thus $nah^b\geq C_1\log n$, implying that $n^2h^{-4}\leq n^2(C_1\log n/(na))^{-4/b}\leq C n^{2+4/b}$. Since $e^{-nv_\mu(h)/C}\leq n^{-C_1/C}$, the conclusion follows if $C_1$ is large enough with respect to $C$ and $b$, and if $nv_\mu(h)\geq C_1\log(1+C_\PI)$.

Using \eqref{eq:eigenvector-bound} instead of \eqref{eq:eigenvalue-bound}, we also find the following result.

\begin{theorem}\label{thm:eigenfunction}
    Under the same assumptions as in \Cref{thm}, there holds the following. Let $\varphi$ be an eigenfunction of $-\Delta^h_\mu$ associated with $\lambda^h_{\ell,\mu}$, normalized by $\|\varphi\|_{\L^2(\mu)}=1$ and represented as in \eqref{eq:representative}. Then, writing $\pi^h_{\ell,\mu_n}$ for the orthogonal projection of $\L^2(\mu_n)$ onto the $\ell$th eigenspace of $-\Delta^h_{\mu_n}$, there holds
    \begin{equation}\label{eq:bound_eigenfunction}
      \E\Big[\|(1-\pi^h_{\ell,\mu_n})\varphi\|^2_{\dot\H^1_h(\mu_n)}+\lambda^h_{\ell,\mu}\|(1-\pi^h_{\ell,\mu_n})\varphi\|^2_{\L^2(\mu_n)}\Big]\leq\frac C{\gamma^2}\lambda^h_{\ell,\mu}\frac{1+C_\PI}{nv_\mu(h)}
    \end{equation}
    for some $C$ depending on $C_\D$ and $\kappa$.
\end{theorem}

We leave the proof to the reader; it is very similar to the proof of \Cref{thm}. There are two meaningful differences. The first one is that the second moment of $\int\phi\dd \mu_n$ must be controlled in \eqref{eq:eigenvector-bound}. This is trivial since $\phi$ is centered and satisfies $\|\phi\|_{\L^2(\mu)}=1$. The second one is that one should observe that on the event $E^c$, the random variable appearing in \eqref{eq:bound_eigenfunction} is bounded almost surely since \eqref{eq:representative} implies an $\L^\infty$-bound on $\phi$. This $\L^\infty$-bound does not have to precise: any bound of the form $h^{-a}v_\mu(h)^{-b}$ is enough, since it will be multiplied by a small probability of order $e^{-c nv_\mu(h)}$.

\appendix

\section{Stability under perturbation of the kernel}\label{sec:kernel}
The goal of this section is to prove \Cref{lem:kernel_variation}, controlling the error made by replacing the population kernel $K_\mu^h$ by the empirical kernel $K_{\mu_n}^h$. 
We split the proof of the lemma into several sublemmas. 
	\begin{lemma}\label{lem:control_E}
		For $i\in [n]$, let $Z_i = |\eta_\mu^h(X_i)-\eta_{\mu_n}^h(X_i)|/\eta_\mu^h(X_i)$. Let $i,j\in [n]$ be such that  $Z_i,Z_j\leq 1/2$. Then, 
		\[
		|K_\mu^h(X_i,X_j)-K_{\mu_n}^h(X_i,X_j)|\leq 4K_\mu^h(X_i,X_j)(Z_i+Z_j)
		\]
        and $K_\mu^h(X_i,X_j)\leq \frac 32 K_{\mu_n}^h(X_i,X_j)$. 
	\end{lemma}
	\begin{proof}
		Observe that
		\[
		K_\mu^h(X_i,X_j)-K_{\mu_n}^h(X_i,X_j) = K_\mu^h(X_i,X_j)\p{1-\sqrt{\frac{\eta_\mu^h(X_i)\eta_\mu^h(X_j)}{\eta_{\mu_n}^h(X_i) \eta_{\mu_n}^h(X_j)}}}.
		\]
		Elementary computations show that for any $a,a',b,b'>0$ such that $|a-a'|\leq a/2$ and $|b-b'|\leq b/2$,
		\[
		\left|1-\sqrt{\frac{ab}{a'b'}}\right| \leq\left|1-\frac{ab}{a'b'}\right|\leq\left|1-\frac{a}{a'}\right|+\frac{a}{a'}\left|1-\frac{b}{b'}\right| \leq \frac{|a-a'|}{a'}+2\frac{|b-b'|}{b'}.
		\]
        But also, under these conditions, $\frac{|a-a'|}{a'}\leq 2\frac{|a-a'|}{a}$ and $\frac{|b-b'|}{b'}\leq 2\frac{|b-b'|}{b}$. 
		Plugging in the correct values of $a,a',b,b'$ yields the first inequality. 
        
        The second inequality is a direct implication of the inequality $\eta_{\mu_n}^h(X_i)\leq (1+Z_i)\eta_\mu^h(X_i)\leq \frac 32 \eta_\mu^h(X_i)$.
	\end{proof}
	We let $E_2$ be the event where $\max_{i\in [n]} Z_i \leq 1/2$. 
	
	\begin{lemma}
		Let $C_\D\geq 1$, $h>0$ and $\mu\in \mathrm{DM}_h(C_\D)$. Let $n$ be an integer such that 
		 $nv_\mu(h)\geq 16C_\D$. 
		Then, there holds $\P(E_2)\geq 1-2ne^{-cnv_\mu(h)/C_\D^2}$, where $c$ is an absolute constant.
	\end{lemma}
	
	\begin{proof}
		Let us recall Bernstein's inequality, which states that if $Y_1,\dots,Y_n$ are i.i.d. centered random variables with $\E[Y_i^2]\leq v$ and $|Y_i|\leq b$, then, for all $t>0$,
		\[
		\P\p{\left|\frac 1n \sum_{i=1}^n Y_i\right|\geq t} \leq 2 e^{-\frac{n t^2}{2(v+bt/3)}}.
		\]
	For $i\in [n]$, let $U_i = \eta\p{\frac{\rho(X_1,X_i)}h}$, with expectation $p(X_1)=\E[U_i|X_1]$ for $i\geq 2$. Write $V_i=U_i-p(X_1)$.  %There holds $p(X_1)\geq \mu(B(X_1,h/2))/2\geq C_\D^{-1}v_\mu(h)$. Then, 
	Observe that 
\begin{equation} \label{eq:bound_Z1}
	Z_1 = \frac{1}{p(X_1)}\left|  \frac{1}{n}\sum_{i=2}^n V_i + \frac{1}{n}(\eta(0)-p(X_1))\right| \leq \frac{2}{np(X_1)} + \frac{1}{p(X_1)}\left|\frac{1}{n}\sum_{i=2}^n V_i \right|.
\end{equation}
There holds $p(X_1)\geq \mu(B(X_1,h/2))/2\geq C_\D^{-1}\mu(B(X_1,h))/2\geq C_\D^{-1}v_\mu(h)/2\geq 8/n$. Thus, $Z_1\leq 1/4  + \frac{1}{p(X_1)}\left|\frac{1}{n}\sum_{i=2}^n V_i \right|$. 
Conditionally on $X_1$, the random variables $V_i$ for $i\geq 2$ are centered and i.i.d., bounded by $1$, with second moment smaller than $\mu(B(X_1,h))$. Thus, by Bernstein's inequality, 
\[
\P(Z_1\geq 1/2)\leq \P\p{\left|\frac{1}{n}\sum_{i=2}^n V_i \right| \geq C_\D^{-1}\mu(B(X_1,h))/8} \leq 2\E\left[e^{-c C_\D^{-2} n\mu(B(X_1,h))} \right]
\]
for some absolute constant $c$. As $\mu(B(X_1,h))\geq v_\mu(h)$, we may conclude by a union bound.
\end{proof}
	
	Let $w\in \L^2(\mu_n)$ with $\|w\|_{\dot \H_h^1(\mu_n)}\leq 1$. On $E_2$, the quantity $\dotp{ ( \Delta_{\mu_n,K_{\mu_n}^h}-\Delta_{\mu_n,K_{\mu}^h})u,w}_{\mu_n}$ is equal to
	\begin{align*}
	&-\frac{1}{n^2} \sum_{i,j=1}^n (K_{\mu_n}^h(X_i,X_j)-K_{\mu}^h(X_i,X_j))(u(X_i)-u(X_j))(w(X_i)-w(X_j)) \\
		&\leq  \frac{4}{n^2} \sum_{i,j=1}^n K_\mu^h(X_i,X_j)(Z_i+Z_j)|u(X_i)-u(X_j)||w(X_i)-w(X_j)| \\
		&\leq \p{\frac{2}{n^2}\sum_{i,j=1}^n K_\mu^h(X_i,X_j)(Z_i+Z_j)^2 (u(X_i)-u(X_j))^2}^{1/2} \\
        &\qquad \cdot \p{\frac{1}{n^2}\sum_{i,j=1}^n K_{\mu}^h(X_i,X_j)(w(X_i)-w(X_j))^2}^{1/2},
	\end{align*}
	where we apply Cauchy-Schwarz inequality at the last line with weight $K_{\mu}^h(X_i,X_j)$. By \Cref{lem:control_E}, on $E_2$, $K_\mu^h(X_i,X_j)\leq \frac 32 K_{\mu_n}^h(X_i,X_j)$, so the second factor is smaller than $\frac 32\|w\|_{\dot H^1_h(\mu_n)}^2\leq \frac 32$. 
    Thus, letting \[A_i = \frac{1}{n}\sum_{j=1}^n K_\mu^h(X_i,X_j)(u(X_i)-u(X_j))^2= \frac{1}{n}\sum_{j=2}^n K_\mu^h(X_i,X_j)(u(X_i)-u(X_j))^2\] and using that $(Z_i+Z_j)^2\leq 2Z_i^2+2Z_j^2$, we find that
	\begin{align*}
		\dotp{  ( \Delta_{\mu_n,K_{\mu_n}^h}-\Delta_{\mu_n,K_{\mu}^h})u,w}_{\mu_n}^2 \leq  \frac{9}{n}\sum_{i=1}^n A_i Z_i^2.
	\end{align*} Thus,
	\[
	\E[ \|( \Delta_{\mu_n,K_{\mu_n}^h}-\Delta_{\mu_n,K_{\mu}^h})u\|^2_{\H^{-1}_{h,0}(\mu_n)}\one_E]\leq 9\E[A_1Z_1^2].
	\]
	Recall the definitions of $U_i$, $V_i$ and $p(X_i)$. We square the bound in \eqref{eq:bound_Z1}, develop $\p{\sum_{i=2}^n V_i}^2$, and write  $A_1 = \frac{1}{n}\sum_{j=2}^n K_\mu^h(X_1,X_j)(u(X_1)-u(X_j))^2$ to find that
	\begin{align*}
		A_1Z_1^2  \leq  \frac{8A_1}{n^2p(X_1)^2}+ \frac{2 }{n^3p(X_1)^2} \sum_{i_1,i_2,i_3=2}^n K_\mu^h(X_1,X_{i_1})(u(X_1)-u(X_{i_1}))^2V_{i_2}V_{i_3}.
	\end{align*}
	We take the expectation conditionally on $X_1$. Since the $V_i$s are centered, only the terms with $i_2=i_3$ have nonzero expectation. For those, we use the bound   $|V_i|^2\leq 2\one_{X_i\in B(X_1,h)}+2p(X_1)^2\leq 4$ to find that
        \begin{align*}
        &\E[A_1Z_1^2|X_1] \leq \frac{8\E[A_1|X_1]}{n^2p(X_1)^2} \\
        &\qquad + \frac{8}{n^3p(X_1)^2} \sum_{i_1=1}^n \E[K_\mu^h(X_1,X_{i_1})(u(X_1)-u(X_{i_1}))^2|X_1] \p{4 +\sum_{i_2\neq i_1} (p(X_1)+p(X_1)^2)} \\
        &\leq  \frac{8\E[A_1|X_1]}{n^2p(X_1)^2}  + \frac{8}{n^2 p(X_1)^2} \E[A_1|X_1] (4+2np(X_1)) \leq \frac{40\E[A_1|X_1]}{n^2p(X_1)^2} + \frac{16\E[A_1|X_1]}{np(X_1)}.
    \end{align*}
But recall that $p(X_1)\geq C_\D^{-1}v_\mu(h)/2$. Thus, taking the expectation with respect to $X_1$, we find that
\[
\E[A_1Z_1^2]\leq \frac{160C_\D^2}{n^2v_\mu(h)^2}\|u\|^2_{\dot\H^1_h(\mu)} + \frac{32C_\D}{nv_\mu(h)}\|u\|^2_{\dot\H^1_h(\mu)}.
\]
Since $nv_\mu(h)\geq 16C_\D$, we have $n^2v_\mu(h)^2\geq 16C_\D nv_\mu(h)$, concluding the proof.

\section{PI measures are also coarse PI measures}\label{sec:PI_0}

Let $C_\D,\kappa\geq 1$ and $C_{\PI}>0$. 
A metric-measure space $(\X,\rho,\mu)$ is said to be a PI space with constants $(C_\D,C_{\PI},\kappa)$ if  $\mu\in\mathrm{DM}_0(C_\D)$ for some $C_\D\geq 1$, and if for all balls $B$ of radius $t$ and all locally bounded measurable function $u$, there holds 
\begin{equation}\label{eq:1_poincare}
\int_B (u-u_B)^2 \dd\mu \leq C_{\PI} t^2 \int_{\kappa B} g^2 \dd \mu
\end{equation}
for all upper gradients $g$ of $u$, meaning that $g$ is measurable and that
\[
|u(\gamma(0))-u(\gamma(1))|\leq \int_\gamma g \dd s
\]
for all rectifiable curves $\gamma$. We then write $\mu \in \PI_0(C_\D,C_{\PI},\kappa)$. We call such an inequality a $(2,2)$-Poincaré inequality.

\begin{proposition}
    Let $\kappa,C_\D\geq 1$ and $C_{\PI},r>0$. Let $\mu \in \PI_0(C_\D,C_{\PI},\kappa)$. Then, $\mu\in \PI_r(C_\D, aC_\D^b(1+ C_{\PI}),\kappa)$ for some  absolute constants $a,b>0$. 
\end{proposition}
\begin{proof}
 According to  \cite[Lemma 6.2.5]{heinonen2015sobolev}, an upper gradient of any locally Lipschitz function  $v$ is given by 
\[
\forall x\in \X,\quad \Lip[v](x)= \limsup_{r\to 0} \sup_{y\in B(x,r)} \frac{|v(x)-v(y)|}r,
\]

  We claim that there exists   a Lipschitz continuous function $K:\X\times \X\to [0,+\infty)$  supported on the strip $\{(x,y):\ \rho(x,y)\leq r/2\}$, with $\int K(x,y)\dd\mu(y)=1$ for all $x\in \X$, and such that for all $x,x',y\in \X$ with $\rho(x,x')\leq r$, 
  \[
  |K(x,y)-K(x',y)|\leq \frac{L\rho(x,x')}{r\mu(B(x,r))}
  \]
  for some constant $L$ of the form $aC_\D^b$. 
 Indeed, we can for instance let $K(x,y) = \eta(\rho(x,y)/(r/2))/\eta_{\mu}^{r/2}(x)$, where $\eta$ is continuous and piecewise affine, equal to $1$ on $[0,1/2]$, and equal to $0$ on $[1,+\infty)$. Then, $K(x,y)$ is the ratio of two functions. Both the  numerator and the denominator are Lipschitz continuous, with Lipschitz constants of order respectively $1/r$ and $\mu(B(x,r))/r$. Moreover, the denominator for both $K(x,y)$ and $K(x',y)$ is lower bounded by $\mu(B(x,r))$ up to multiplicative constants thanks to the doubling property.

We fix such a kernel and 
    let $v:x\mapsto \int K(x,y)u(y)\dd \mu(y)$. Observe that for all $x\in \X$, $|u(x)-v(x)|\leq \int K(x,y)|u(x)-u(y)|\dd\mu(y) \leq r\Lip_{\mu,r}[u](x)$. Moreover, the function is locally bounded and (Lipschitz-)continuous. Thus, for a ball $B$ of radius $t\geq r$, since $\Lip[v]$ is an upper gradient
    \begin{align}
        \int_B |u-u_B|^2 \dd \mu &\leq \int_B |u-v_B|^2\dd\mu\leq 2\int_B |u-v|^2\dd\mu + 2\int_B|v-v_B|^2\dd\mu \nonumber\\
        &\leq 2r^2\int_B \Lip_{\mu,r}[u]^2 \dd \mu + 2C_{\PI}t^2 \int_{\kappa B}\Lip[v]^2 \dd\mu. \label{eq:poincare_0_radius}
    \end{align}
    Let us bound $\Lip[v](x)$ for $x\in\X$. Observe that $\Lip[v](x)$ is left invariant if $u$ is replaced by $u+c$ for a constant $c$, so we might assume that $u(x)=0$.   Let $x'$ be such that $\rho(x,x')\leq r/2$. Then,
    \begin{align*}
        |v(x)-v(x')| &\leq \int_{B(x,r/2+\rho(x,x'))} |K(x,y)-K(x',y)||u(y)|\dd \mu(y) \\
        &\leq \frac{L}{r\mu(B(x,r))}\rho(x,x') \mu(B(x,r))\esssup_\mu \{|u(y)-u(x)|:\ y\in B(x,r)\} \\
        &\leq L \rho(x,x')\Lip_{\mu,r}[u](x).
    \end{align*}
    Thus, $\Lip[v]\leq L\Lip_{\mu,r}[u]$. Plugging in this inequality in \eqref{eq:poincare_0_radius}, we find that 
    \[
    \int_B |u-u_B|^2 \dd \mu \leq 2t^2 \int_{\kappa B} \Lip_{\mu,r}[u]^2 \dd \mu + 2C_{\PI}L^2 t^2 \int_{\kappa B} \Lip_{\mu,r}[u]^2 \dd \mu,
    \]
    concluding the proof.
\end{proof}

\section{Existence of coarse almost-Lipschitz spline systems}\label{sec:spline}

The constructions of \cite{auscher2013orthonormal,hytonen2014almost} produce spline systems on metric spaces that are doubling at every scale. They rely on the dyadic cubes of Christ \cite{christ1990b}, which are obtained through an infinite refinement procedure, and therefore depend on the geometry of $\X$ at arbitrarily small scales. We show here that a truncated construction, based on a Voronoi partition at the scale $\delta^J \geq 2r$, yields spline systems up to depth $J$ with all the properties required in \Cref{sec:poincare}. 

Throughout this appendix,  $C_{\D} \geq 1$, $r>0$, and  $\mu \in \mathrm{DM}_r(C_{\D})$, $\diam(\X)<1$ and $2r \leq 1$. The measure $\mu$ plays no role, besides its existence granting doubling properties of the underlying space (\Cref{lem:doubling_metric}). We set $C_0 = 37 C_{\D}^4$, fix $\delta \in (0,1/C_0)$, and let $J \geq 0$ be an integer such that $\delta^J \geq 2r$, and
\[
  \vartheta = 1 - \frac{\log C_0}{\log(1/\delta)} \in (0,1),
  \qquad \text{so that } C_0\delta = \delta^{\vartheta}.
\]
Note that $\vartheta \to 1$ as $\delta \to 0$. In particular $\vartheta > 1/2$ as soon as $\delta < C^{-2}$.

\begin{proposition}\label{prop:splines}
There is a $\delta$-spline system $(s^j_k)_{0 \leq j \leq J, k \in [L_j]}$ on $(\X,\rho)$ up to depth $J$, in the sense of \Cref{def:splines}, such that for all $j$, $k$ and all $x,y \in \X$,
\begin{equation}\label{eq:holder-splines}
  \left| s^j_k(x) - s^j_k(y) \right| \leq \frac1\delta
  \p{ \frac{\max(\rho(x,y), \delta^J)}{\delta^{j}} }^{\vartheta}.
\end{equation}
\end{proposition}

Recall the family $(\X^j)_{0 \leq j \leq J}$ fixed in \Cref{sec:poincare}:
$\X^0$ is a maximal $1$-separated subset of $\X$ and, inductively, $\X^{j} = \{x^j_k\}_{k \in [L_j]} \supseteq \X^{j-1}$ is a maximal $\delta^{j}$-separated subset of $\X$.  Each  set $\X^j$ is finite by \Cref{lem:covering_ball}. Since $\operatorname{diam}(\X)<1$, $\X^0$ is a single point.

Let $(V_k)_{k \in [L_J]}$ be the Voronoi partition associated with $\X^J$, the cell $V_k$ being the set of points of $\X$ whose closest point in $\X^J$ is $x^J_k$ (with ties broken using the lexicographic order). The cells are Borel sets and
\begin{equation}\label{eq:voronoi}
  B(x^J_k, \delta^J/2) \subseteq V_k \subseteq \bar B(x^J_k, \delta^J).
\end{equation}
Indeed, the right-hand inclusion holds by maximality of $\X^J$, and if $\rho(x,x^J_k) < \delta^J/2$ then $\rho(x,x^J_\ell) \geq \rho(x^J_k,x^J_\ell) - \rho(x,x^J_k) > \delta^J/2 > \rho(x,x^J_k)$ for $\ell \neq k$.

Let $a_0,\dots,a_{J-1}$ be independent and uniformly distributed on $\{0,\dots,\lfloor \delta^{-1}\rfloor\}$, on a finite probability space $(\Omega,\P)$, and define the random variables 
\begin{equation}\label{eq:radii}
  r_j = \frac14 (\delta^{j} + a_j \delta^{j+1}) \in \left[\frac{\delta^{j}}{4}, \frac{\delta^{j}}{2}\right], \qquad 0 \leq j < J.
\end{equation}
For $0 \leq j < J$ and $k \in [L_{j+1}]$, the parent $p(x^{j+1}_k) \in \X^j$ is the point $x^j_\ell$ with $\rho(x^j_\ell,x^{j+1}_k) < r_j$ if such a point exists, and otherwise the point $x^j_\ell$ with $\ell$ smallest such that $\rho(x^j_\ell,x^{j+1}_k) < 4r_j$. Note that the parent map is random since the $r_j$s are random.

\begin{lemma}\label{lem:parent}
The parent map is well defined and satisfies $\rho(p(x^{j+1}_k),x^{j+1}_k) < 4r_j \leq 2\delta^{j}$. Moreover $p(z) = z$ for every $z \in \X^{j} \subseteq \X^{j+1}$.
\end{lemma}

\begin{proof}
In the first case the point is unique: two such points would satisfy $\rho(x^j_\ell,x^j_{\ell'}) < 2r_j \leq \delta^{j}$. In the second case the set of admissible indices is nonempty by maximality of $\X^j$, which gives $\rho(x^{j+1}_k,\X^j) < \delta^{j} \leq 4r_j$, and finite since $\X^j$ is.
Finally $4r_j = \delta^j + a_j\delta^{j+1} \leq 2\delta^j$, and $\rho(z,z) = 0 < r_j$ for $z \in \X^j$.
\end{proof}

The (random) cubes are now defined from the finest generation upwards:
\[
  Q^J_k = V_k, \qquad Q^{j}_k = \bigcup_{\ell :\ p(x^{j+1}_\ell) = x^j_k} Q^{j+1}_\ell, \quad 0 \leq j < J.
\]

\begin{lemma}\label{lem:partition}
For every  $j$, the family $(Q^j_k)_{k \in [L_j]}$
is a Borel partition of $\X$ into nonempty sets, and
$Q^{j+1}_\ell \subseteq Q^j_k$ whenever $p(x^{j+1}_\ell) = x^j_k$.
\end{lemma}

\begin{proof}
The case $j = J$ is \eqref{eq:voronoi}. We then follow by induction:  every point of $\X^{j+1}$ has exactly one parent, so the level-$j$ cubes are the unions of the level-$(j+1)$ cubes along the fibres of $p$. They therefore form a Borel partition of $\X$. They are nonempty since, by \Cref{lem:parent}, $Q^j_k$ contains the level-$(j+1)$ cube centered at $x^j_k$.
\end{proof}

\begin{lemma}\label{lem:localization}
Assume $\delta \leq 1/25$. Then for  every $j$ and every $k$,
\begin{equation}\label{eq:two-sided}
  B(x^j_k, \delta^{j}/8) \subseteq Q^j_k \subseteq B(x^j_k, 3\delta^{j}).
\end{equation}
\end{lemma}

\begin{proof}
Let us first prove the second inclusion. Let $x \in Q^j_k$. By construction there are indices $k=k_j , k_{j+1}, \dots, k_J$ with $p(x^{i+1}_{k_{i+1}}) = x^{i}_{k_i}$ and $x \in V_{k_J}$. By \eqref{eq:voronoi} and \Cref{lem:parent},
\[
  \rho(x,x^j_k) \leq \rho(x,x^J_{k_J}) + \sum_{i=j}^{J-1} \rho(x^{i+1}_{k_{i+1}},x^{i}_{k_i}) \leq \delta^{J} + \frac{2\delta^{j}}{1-\delta}.
\]
For $j = J$ the sum is empty and this is $\delta^J$. For $j < J$ one has
$\delta^J \leq \delta^{j+1}$, so the bound is $\delta^{j}(\delta + 2/(1-\delta)) \leq 3\delta^{j}$.

Let us prove the first inclusion. The case $j = J$ is \eqref{eq:voronoi}. Let $j<J$ and $\rho(y,x^j_k) < \delta^{j}/8$, and let $x^{j+1}_\ell$ be the center of the level-$(j+1)$ cube containing $y$. The upper inclusion at level $j+1$ gives $\rho(y,x^{j+1}_\ell) \leq 3\delta^{j+1}$. Thus,
\[
  \rho(x^{j+1}_\ell,x^j_k) \leq 3\delta^{j+1} + \frac{\delta^{j}}{8} < \frac{\delta^{j}}{4} \leq r_j.
\]
The first case in the definition of the parent map therefore applies, and by uniqueness $p(x^{j+1}_\ell) = x^j_k$. Hence, $y \in Q^{j+1}_\ell \subseteq Q^j_k$.
\end{proof}

For $0 \leq j \leq J$ and $k \in [L_j]$, set $s^j_k:x\in \X \mapsto \P(x \in Q^j_k)$. Since $\Omega$ is finite, $s^j_k = \sum_{\omega\in \Omega} \P(\omega) \one_{Q^j_k(\omega)}$ is measurable and takes values in $[0,1]$.

\begin{proposition}\label{prop:spline-axioms}
The family $(s^j_k)_{0 \leq j \leq J, k \in [L_j]}$ is a $\delta$-spline system on $(\X,\rho)$ up to depth $J$.
\end{proposition}

\begin{proof}
 The inclusions \eqref{eq:two-sided} hold for all realizations of the random cubes, so $\one_{B(x^j_k,\delta^j/8)} \leq s^j_k \leq\one_{B(x^j_k,3\delta^j)} \leq \one_{B(x^j_k,8\delta^j)}$.
Moreover, by \Cref{lem:partition}, $\sum_k s^j_k(x) = 1$ for every $x$. By \eqref{eq:two-sided}, $x^j_m \in Q^j_m$ for every realization, so $s^j_m(x^j_m) = 1$; and $x^j_m \notin Q^j_k$ for $k \neq m$ since the cubes are disjoint, so $s^j_k(x^j_m) = 0$.

It remains to prove that for $j<J$, $s_k^j$ is a linear combinations of the functions $s_m^{j+1}$. By \Cref{lem:partition}, $\one_{Q^j_k} = \sum_{\ell} \one_{\{p(x^{j+1}_\ell) = x^j_k\}} \one_{Q^{j+1}_\ell}$.
The event $\{p(x^{j+1}_\ell) = x^j_k\}$ is $\sigma(a_j)$-measurable while $Q^{j+1}_\ell$ is determined by $a_{j+1},\dots,a_{J-1}$, so taking expectations and using independence,
\[
  s^j_k = \sum_{\ell \in [L_{j+1}]} p_{k\ell} s^{j+1}_\ell,
  \qquad p_{k\ell} := \P(p(x^{j+1}_\ell) = x^j_k) \geq 0,
\]
with $\sum_k p_{k\ell} = 1$ since every point has exactly one parent. 
\end{proof}

Fix $x,y \in \X$. For $0 \leq j \leq J$, let $D_j$ be the event that $x$ and $y$ lie in different cubes of level $j$, and let $\cF_j = \sigma(a_j,\dots,a_{J-1})$, so that the cubes of level $j$ are $\cF_j$-measurable. Since level $j+1$ refines level $j$, one has $D_j \subseteq D_{j+1}$ for $j<J$. Furthermore, on $D_j^c$ the two indicators $\one_{x \in Q^j_k}$ and $\one_{y \in Q^j_k}$ agree, so that
\begin{equation}\label{eq:spline-vs-Dj}
  | s^j_k(x) - s^j_k(y) | \leq \P(D_j).
\end{equation}

\begin{lemma}\label{lem:one-step}
Assume that $\delta \leq 1/25$ and let $j \leq i < J$. If $\rho(x,y) \leq \delta^{i+1}$ then $\P(D_i | \cF_{i+1}) \leq 37C_\D^4 \delta \one_{D_{i+1}}$.
\end{lemma}

\begin{proof}
On $D_{i+1}$, let $u$ and $v$ be the centers of the level-$(i+1)$ cubes containing $x$ and $y$. They are $\cF_{i+1}$-measurable and distinct. A point lies in a level-$i$ cube if and only if its level-$(i+1)$ cube does, so $D_i = D_{i+1} \cap \{p(u) \neq p(v)\}$. As $a_i$ is independent of $\cF_{i+1}$,
$ \P(D_i | \cF_{i+1}) = \one_{D_{i+1}} g(u,v)$, 
where $g(w,w')= \P(p(w) \neq p(w'))$ for deterministic $w,w' \in \X^{i+1}$. Define $d=\rho(u,v)$.  By \Cref{lem:localization} at level $i+1$, on $D_{i+1}$ we have
$d \leq \rho(x,y) + 6\delta^{i+1} \leq 7\delta^{i+1} \leq \delta^{i}$.

Let us now assume that $D_{i+1}$ is satisfied and let us bound $g(u,v)$.
 Suppose that for every $z \in \X^{i}$,
\begin{equation}\label{eq:no-separation}
  \one_{\rho(z,u)<r_i} = \one_{\rho(z,v)<r_i} \quad \text{and} \quad \one_{\rho(z,u)<4r_i} = \one_{\rho(z,v)<4r_i}.
\end{equation}
Then $p(u) = p(v)$. Indeed, if $\rho(z,u) < r_i$ for some $z$, then also $\rho(z,v)<r_i$, and by uniqueness this $z$ is the parent of both. Otherwise no $z$ satisfies $\rho(z,u)<r_i$, hence none satisfies $\rho(z,v)<r_i$, so both parents are given by the second rule. The sets $\{z : \rho(z,u)<4r_i\}$ and $\{z : \rho(z,v)<4r_i\}$ coincide by \eqref{eq:no-separation}, hence so do their smallest elements. Thus, $g(u,v)$ is smaller than the sum over all $z\in \X_i$ of the probability that the condition in \eqref{eq:no-separation} is not satisfied.

 Let us first bound  the number of elements $z\in \X^i$ where this condition is not satisfied. Since $4r_i \leq 2\delta^{i}$, both indicators in
\eqref{eq:no-separation} vanish for every value of $a_i$ as soon as $\min(\rho(z,u),\rho(z,v)) \geq 2\delta^{i}$. Only the centers in $\cN = \{z \in \X^{i}:\ \rho(z,u) < 2\delta^{i}+d\} \subseteq B(u,3\delta^{i})$ may therefore violate \eqref{eq:no-separation}. As $\cN$ is $\delta^{i}$-separated and $\delta^{i}/2 \geq r$, \Cref{lem:covering_ball} gives
$|\cN| \leq C_{\D}^{\lceil \log_2 13 \rceil} = C_{\mathrm D}^{4}$.

Let $z \in \mathcal N$. One has $\one_{\rho(z,u)<t} \neq \one_{\rho(z,v)<t}$ if and only if $t$ lies in the interval with endpoints $\rho(z,u)$ and $\rho(z,v)$, of length at most $d$. By \eqref{eq:radii}, $r_i$ is uniform on a grid of $\lfloor \delta^{-1}\rfloor + 1 \geq \delta^{-1}$ points with spacing $\delta^{i+1}/4$, and an interval of length $L$ contains at most $4L\delta^{-i-1}+1$ of them, so that
\[
  \P(r_i \in I) \leq \delta (4|I| \delta^{-i-1}+1),
  \qquad
  \P(4 r_i \in I) \leq \delta (|I| \delta^{-i-1}+1),
\]
where the second bound holds because $\{4r_i \in I\} = \{r_i \in I/4\}$. Taking
$|I| \leq d \leq 7\delta^{i+1}$ and summing over $z \in \cN$ and
over the two thresholds gives
$g(u,v) \leq |\cN| \delta \left(5 d \delta^{-i-1}+2\right) \leq 37 C_\D^{4} \delta$.
\end{proof}

\begin{proof}[Proof of \Cref{prop:splines}]
By \Cref{prop:spline-axioms} it remains to prove \eqref{eq:holder-splines}. First, observe that the condition $\delta\leq 1/25$ is satisfied since $\delta<1/C_0$ with $C_0=37C_\D^4$. Let $\bar\eps = \max(\rho(x,y),\delta^{J})$.
For $j \leq i \leq J-1$ one has $\delta^{i+1} \geq \delta^{J}$, so that $\delta^{i+1} \geq \rho(x,y)$ if and only if $\delta^{i+1} \geq \bar\eps$. Let
$L$ be the number of indexes $i\in  \{j,\dots,J-1\}$ with $\delta^{i+1} \geq \bar\eps$, the set of such indexes being  equal to $\{j,\dots,j+L-1\}$. Applying \Cref{lem:one-step} for $i = j,\dots,j+L-1$ and using $D_i \subseteq D_{i+1}$, we obtain
\[
  \P(D_j) \leq C_0\delta \P(D_{j+1}) \leq \dots \leq (C_0\delta)^{L} = \delta^{\vartheta L}.
\]
If $L = J-j$ then $\bar\eps \leq \delta^{J}$, hence $\bar\eps = \delta^{J}$ and $\delta^{L} = \bar\eps/\delta^{j}$. If $L < J-j$ then $j+L \leq J-1$ and, by maximality of $L$, $\delta^{j+L+1} < \bar\eps$, so $\delta^{L} < \delta^{-1} \bar\eps \delta^{-j}$. In both cases, since $\vartheta \leq 1$,
\[
  \delta^{\vartheta L}
  \leq \delta^{-\vartheta} \p{\frac{\bar\eps}{\delta^{j}}}^{\vartheta}
  \leq \frac1\delta \p{ \frac{\bar\eps}{\delta^{j}} }^{\vartheta},
\]
which together with \eqref{eq:spline-vs-Dj} is \eqref{eq:holder-splines}.
\end{proof}

\bibliographystyle{alpha}
\bibliography{biblio}

\end{document}